\documentclass{amsart}
\usepackage{amssymb, latexsym, amsmath, eucal,cite,amsthm}
\usepackage{graphicx}
\usepackage{wrapfig}

\usepackage{tikz}
\usepackage{caption}
\usepackage[labelformat=simple,labelfont={}]{subcaption}

\usepackage{setspace}
\usepackage[letterpaper, left=2.5cm, right=2.5cm, top = 1in, bottom = 1in]{geometry}
\theoremstyle{plain}
\newtheorem{thm}{Theorem}[section]
\newtheorem{lemma}[thm]{Lemma}
\newtheorem{cor}[thm]{Corollary}

\usepackage{times}
\usepackage[italic,defaultmathsizes]{mathastext}

\theoremstyle{definition}
\newtheorem{defn}{Definition}

\newtheorem*{remark}{Remark}

\newcommand{\R}{\mathbb{R}}
\newcommand{\ub}{\ddot u}

\title{Knot Projections with a Single Multi-crossing}
\author[Adams]{Colin Adams}
\address{Colin Adams,Williams College}
\email{Colin.C.Adams@williams.edu}
\author[Crawford]{Thomas Crawford}
\address{Thomas Crawford, Boston College}
\email{tomc1390@gmail.com}
\author[DeMeo]{Benjamin DeMeo}
\address{Benjamin DeMeo, Williams College}
\email{bd2@williams.edu}
\author[Landry]{Michael Landry}
\address{Michael Landry, University of Calfornia, Berkeley}
\email{michaellandry@berkeley.edu}
\author[Lin]{Alex Tong Lin}
\address{Alex Tong Lin, University of California, Santa Barbara}
\email{axtgln@gmail.com}
\author[Montee]{MurphyKate Montee}
\address{MurphyKate Montee, Notre Dame University}
\email{mmontee@nd.edu}
\author[Park]{Seojung Park}
\address{Seojung Park, Korea Advanced Institute of Science and Technology}
\email{micha82@kaist.ac.kr}
\author[Venkatesh]{Saraswathi Venkatesh}
\address{Saraswathi Venkatesh, California Institute of Technology}
\email{sarsjv@gmail.com} 
\author[Yhee]{Farrah Yhee}
\address{Farrah Yhee, Wellesley College}
\email{farrah.yhee@gmail.com} 
\date{}

\begin{document}
\maketitle

\begin{abstract}
Introduced recently, an $n$-crossing is a singular point in a projection of a link at which $n$ strands cross such that each strand travels straight through the crossing. We introduce the notion of an \"ubercrossing projection, a knot projection with a single $n$-crossing. Such a projection is necessarily composed of a collection of loops emanating from the crossing. We prove the surprising fact that all knots have a special type of \"ubercrossing projection, which we call a petal projection, in which no loops contain any others. The rigidity of this form allows all the information about the knot to be concentrated in a permutation corresponding to the levels at which the strands lie within the crossing. These ideas give rise to two new invariants for a knot $K$: the \"ubercrossing number $\ub(K)$, and petal number $p(K)$. These are the least number of loops in any \"ubercrossing or petal projection of $K$, respectively. We relate $\ub(K)$ and $p(K)$ to other knot invariants, and compute $p(K)$ for several classes of knots, including all knots of nine or fewer crossings.
\end{abstract}

\section{Introduction}

Classically, so-called \emph{regular} projections of knots, in which each crossing consists of one overstrand and one understrand, have played a central role in knot theory. In \cite{Ad}, Adams  deviates from this norm by considering an $n$-crossing (also known as a multi-crossing), which he defines to be a singular point in a projection at which $n$ strands cross, such that each strand bisects the crossing. We say an $n$-crossing has \emph{multiplicity} $n$, and identify the levels of the strands in an $n$-crossing with integers $1, 2, \dots, n$, where $i > j$ indicates that strand $i$ crosses over strand $j$. Figure \ref{fig:4crossing} shows an example of a 4-crossing viewed slightly from the side as well as from the top.

\begin{figure}[!htbp]
\includegraphics[height=30mm]{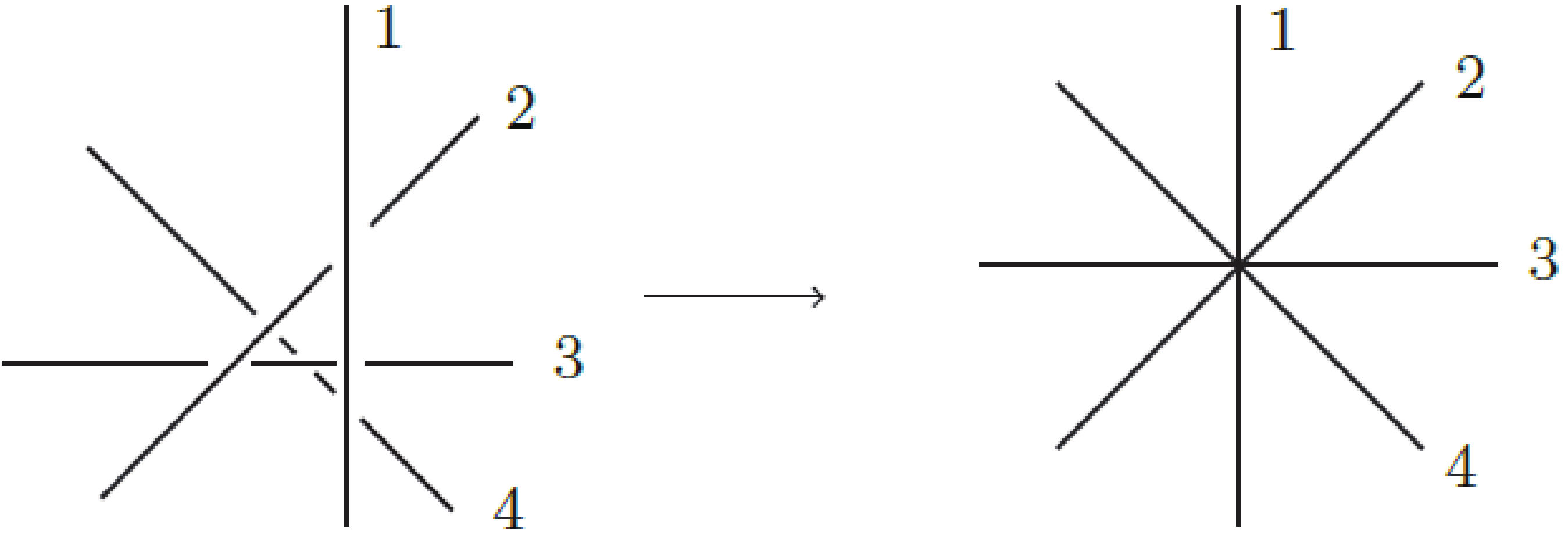}
\caption{An example of a 4-crossing.}
\label{fig:4crossing}
\end{figure}
Let $L$ be a knot or link. Theorem 4.3 in \cite{Ad} states that $L$ has an \emph{$n$-crossing projection} for every $n$, i.e. a projection where each crossing has multiplicity $n$.  Hence we have  the notion of the \emph{$n$-crossing number} of $L$, denoted $c_n(L)$, which is the least number of $n$-crossings in any $n$-crossing projection of $L$.

With these new ideas in mind, it seems natural to ask if there is a projection of $L$ with a single $n$-crossing for some $n$, as in Figure \ref{fig:tref}. Note that the existence of arc presentations of knots gives the existence of knot projections with only one singularity, but our requirement that each strand bisects the crossing is considerably stricter than this. In Section 2 we answer this question in the affirmative. We call this an \emph{\"ubercrossing projection} and call the single crossing an \emph{\"ubercrossing}. This gives a new invariant of knots and links, the \emph{\"ubercrossing number}, which we denote $\ub(L)$ and define to be the least $n$ such that $c_n(L)=1$.

\begin{figure}[htbp!]
	\centering
	\includegraphics[height=40mm]{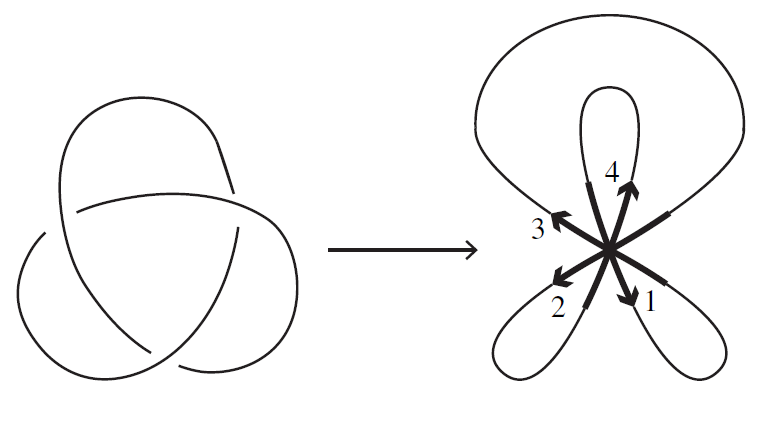}
	\caption{Trefoil knot.}
	\label{fig:tref}
\end{figure}

Let $P$ be an \"ubercrossing projection of $L$ in the plane; it is composed of an \"ubercrossing and a collection of loops emanating from the crossing. A \emph{nesting loop} of $P$ is a loop with at least one other loop in its interior.  
 
A \emph{petal projection} is an \"ubercrossing projection that has no nesting loops (see Figure \ref{fig:petal.pdf}). In Section 2, we prove the surprising fact that any knot $K$ has a petal projection. Thus we can define yet another invariant for knots called the \emph{petal number}, denoted $p(K)$, which is the least number of loops in any petal projection of $K$.

\begin{figure}
\centering
\includegraphics[height=30mm]{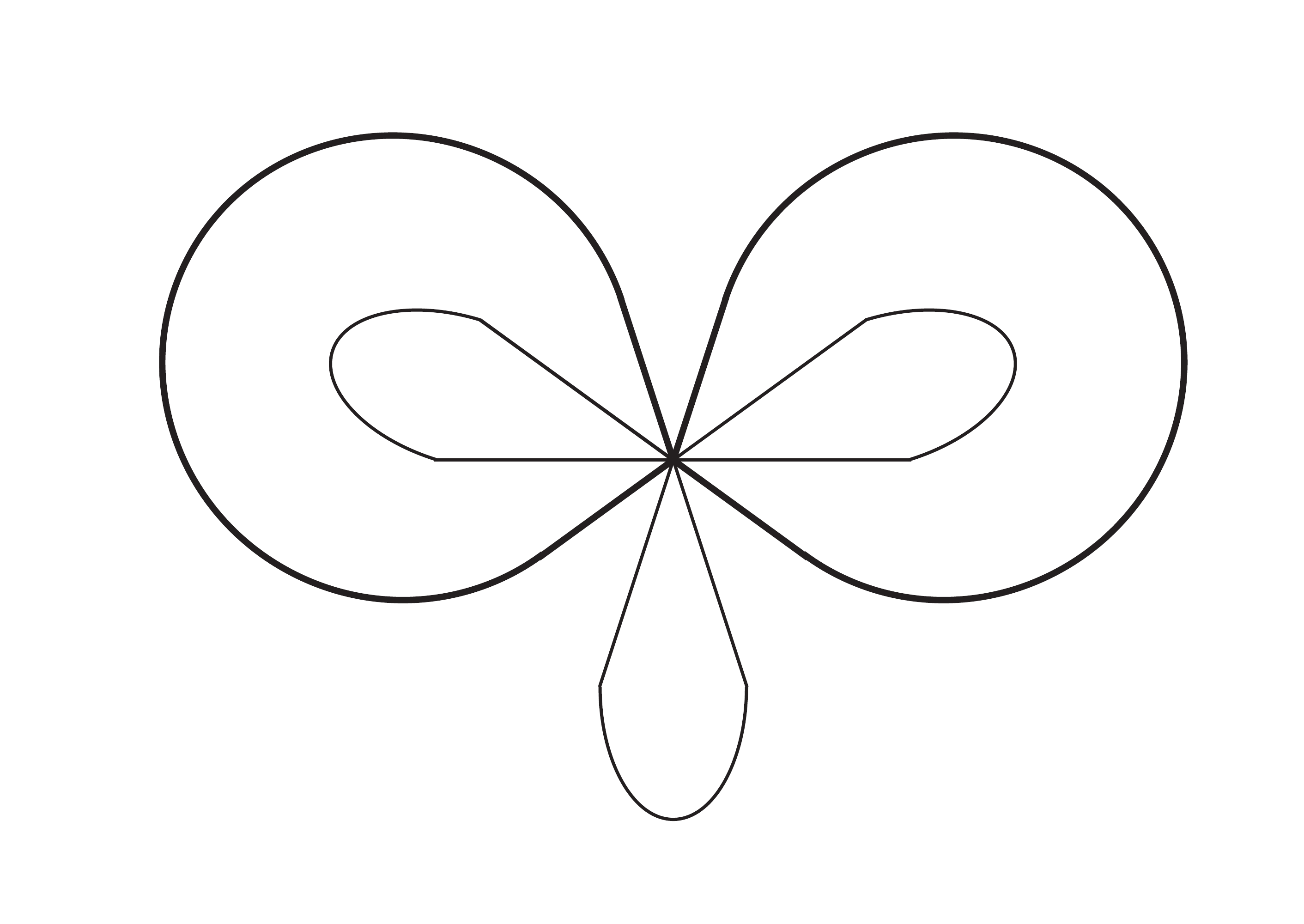}
\caption{The loops shown in bold are nesting loops.}
\end{figure}

\begin{figure}[htbp!]
\includegraphics[height=40mm]{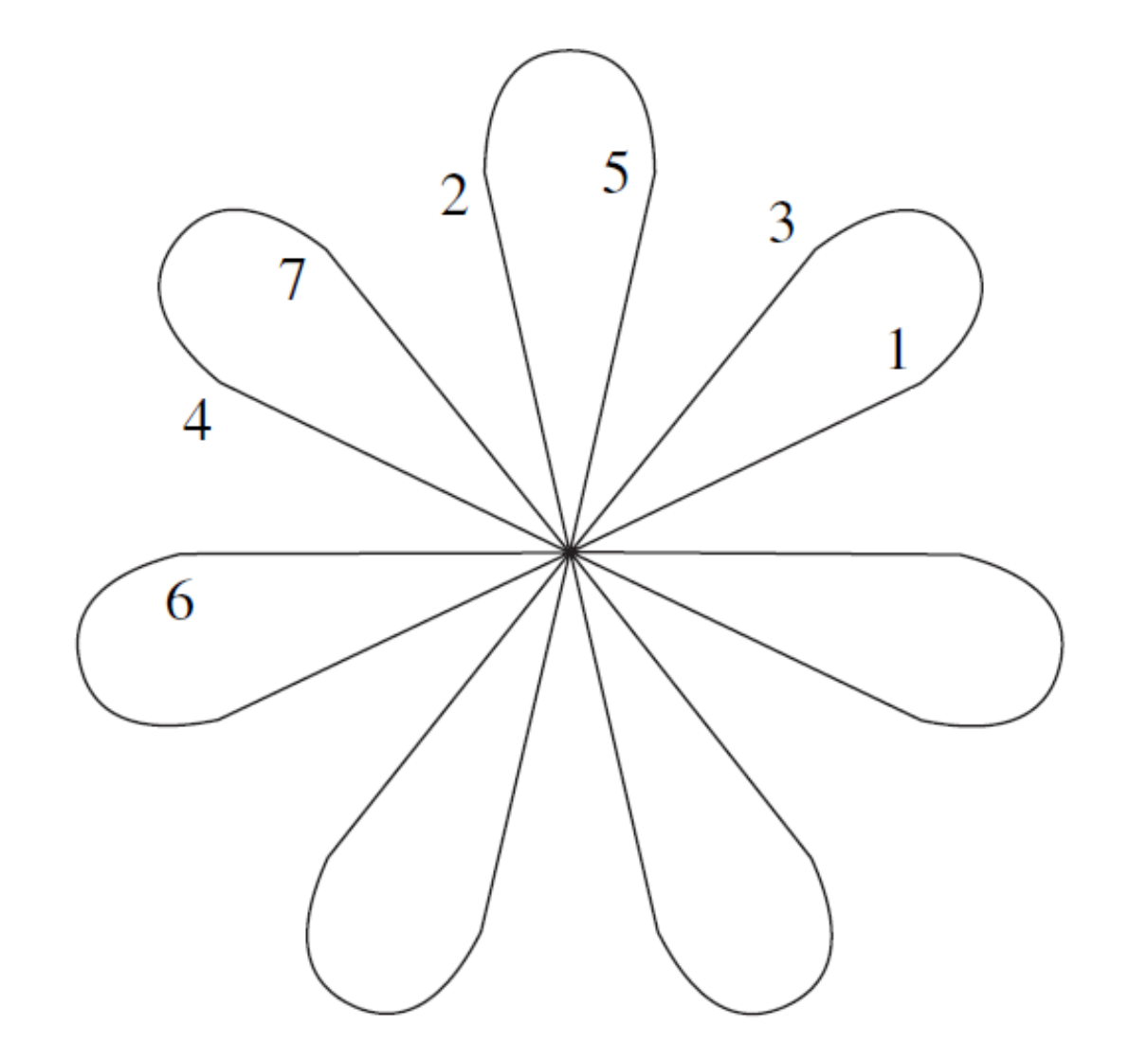}
\caption{Petal projection of $4_1$}
\label{fig:petal.pdf}
\end{figure}

In Section 3, we discuss relations between $\ub(K)$, $p(K)$, and other knot invariants. In particular we relate them to arc index, braid index, and stick number.  We obtain bounds on $\ub(K)$ and $p(K)$ using these invariants.

In Section 4, we introduce a collection of moves that transform regular projections of certain classes of knots into \"ubercrossing and petal projections. In this way we obtain upper bounds on $\ub(K)$ in terms of $c(K)$ for these knots, and in some cases actually determine $p(K)$.

In Section 5 we list the petal number of all knots with nine or fewer crossings.


\section{Existence of \"Ubercrossing Projections}

We now present an algorithm which places any knot in an \"ubercrossing projection. Steps 4 and 5 are justified in the proof which follows.
\begin{enumerate}
\item[\emph{Step 1.}] Isotope a projection of the knot so that all crossings are on a vertical line oriented upwards, which we call $A$.  Orient the knot, and label the rightmost point of the projection the base point, $p$.  Starting from the base point and following the orientation, label each crossing $o$ or $u$ if the crossing is first traversed as an overpass or underpass respectively.  Then isotope the projection so that the overcrossings are on the right side of $A$, and the undercrossings are on the left side of $A$.
\item[\emph{Step 2.}] Fixing the crossing points, rectilinearize the knot so that all line segments are parallel or perpendicular to $A$ in the plane. Notice that $p$ is now on a rightmost line segment which therefore has no crossings on it.  Choosing any point along this line segment as a base point will yield the same labeling of overcrossings and undercrossings, so choose a new basepoint, denoted $\tilde{p}$, to be the endpoint of this segment which is furthest along according to the orientation. Label the other endpoint of this vertical segment $q$. Call the segments that intersect $A$ \emph{intersecting segments}, and all other segments \emph{non-intersecting segments}.
\item[\emph{Step 3.}] Further isotope this projection so that the arc from $\tilde{p}$ to the first intersection of the projection with $A$ is a straight line. This is possible since all crossings of this path are first encountered as overcrossings and hence the arc is unknotted.  Similarly, isotope the arc from $q$ to the last intersection of the projection with the axis so that it is straight.  We call this projection $\widetilde{P}$, and we will use it to obtain a conformation of $K$ in $\mathbb{R}^3$ such that projection down $A$ yields an \"ubercrossing projection.
\item[\emph{Step 4.}] Begin at $\tilde{p}$ and follow the orientation. We leave the first segment (note that it is horizontal) as it is in the projection plane. Each subsequent intersecting segment is rotated counterclockwise around $A$ a bit farther out of the projection plane than the previous one. We connect the endpoints of these segments as follows. As we move along a connected path of non-intersecting segments that starts at the final endpoint of an intersecting segment and ends at the initial endpoint of the subsequent intersecting segment, we continuously rotate monotonically in such a manner as to maintain the fact that the projection back to the plane is $\widetilde P$.
\item[\emph{Step 5.}] Connect the image of $q$ to $\tilde{p}$ with a straight line segment.
\end{enumerate}
		

The projection back to the original projection plane yields $\widetilde{P}$. The further rotations of subsequent intersecting segments ensures that the o and u labels placed on the crossings are respected. Hence, the knot type is unchanged. The following proof makes rigorous this method of rotation by explicit construction of a map $\varphi$ which performs the desired rotation without changing the knot type. The technical details lie in the construction of $\varphi$.

\begin{thm} \label{thm:uberalgorithm}
Every knot has an \"ubercrossing projection.
\end{thm}

\begin{proof}  
Given a projection $P$ of a knot, perform Steps 1-3 of the above algorithm to obtain $\widetilde P$. Let $n$ be the number of intersecing segments, and let $e_j^-$ and $e_j^+$ be the endpoints of the $j$-th intersecting segment (with respect to $\tilde p$ and the orientation) to the left and right of $A$, respectively.  Notice that $\tilde{p}=e_0^+$ and $q = e_{n-1}^+$.
Let the projection plane be the $xz$-plane embedded in $\R^3$, with the positive $y$-axis pointing into the plane. Assign a rectangular coordinate $(x_j^\pm, z_j)$ to each endpoint $e_j^\pm$ by choosing $A$ to be the $z$-axis and the lowest intersection of $\widetilde P$ and $A$ to be the origin. Finally, let $f\colon I\to \R^2$, $t\mapsto (f_x(t),f_z(t))$, be a regular parametrization of $\widetilde P$ which agrees with the orientation such that $f(0)=f(1)=\tilde p$, and define $t_{j}^\pm=f^{-1}(e_j^\pm)$, $t_q=f^{-1}(q)$.

We define the rotation by letting the $i$th intersecting segment of $\widetilde P$ rotate $i/(n-1)$ radians about the $z$-axis, and connecting their endpoints by appropriately rotating the non-intersecting segments with respect to their arclength. Let $\gamma(a, b)$ be the arclength of $\widetilde{P}$ from $a$ to $b$. Define a function $\Gamma_{a,b}(r) = \gamma(a,r)/ \gamma(a,b)$ for $r \in [a, b]$, which measures how far $r$ is along the path from $a$ to $b$ on $\widetilde P$. Notice that $\Gamma_{a,b}(r) \in [0,1]$.

Let $I'\subset I$ denote the points in $I$ which are mapped into non-intersecting segments. Given a point $t \in I'$, if $j$ is the unique index such that $t_{j-1}^\pm\leq t\leq t_{j}^\pm$, define a function $\theta\colon I'\to [\frac{j-1}{n-1}, \frac{j}{n-1}]$ as
\[
\theta(t) = \frac{j-1+ \Gamma_{e_{j-1}^\pm, e_j^\pm}(f(t))}{n-1},
\] 
which we will use to define the rotation of the nonintersecting segments. We now define the map $\varphi \colon I\to \mathbb{R}^3$ as follows:
\[ 
\varphi(t) = \begin{cases}
(\frac{f_x(t)}{\cos(j/(n-1))}, j/(n-1), f_z(t)) & \text{if  $t \in[t_j^\pm$, $t_j^\mp]$, \, $j = 0, 1, \dots, n-1$,} \\
(\frac{f_x(t)}{\cos(\theta(t))}, \theta(t), f_z(t))	& \text{if $t \in [t_j^\pm, t_{j+1}^\pm]$, \, $j=0,\dots,n-2$,} \\
(\frac{f_x(t)}{\cos(1-\Gamma_{f(t)}(q,\tilde{p}))},1-\Gamma_{q,\tilde p}(f(t)) , z)& \mbox{if } t \in [t_q, 1]
\end{cases}
\]
(in cylindrical coordinates $(r,\theta, z)$). Since the standard projection of $(\frac{x}{\cos\psi}, \psi, z)$ to the $xz$-plane is precisely $(x,z)$, we see that the projection of $\varphi(I)$ to the $xz$-plane is precisely $\widetilde P$, with possibly some crossing changes. Consider a crossing point $f(\tau_1)=f(\tau_2)$, $\tau_1<\tau_2$, labeled $o$ in the $xz$-plane. Since the $y$-coordinate of $\varphi(\tau_1)$ is greater than that of $\varphi(\tau_2)$ by construction, the projection does not change the crossing. Crossings labeled $u$ are unchanged by $\varphi$ and projection for the same reason, so in fact no crossings are changed.

Moreover, the projection of $\varphi(I)$ to the $xy$-plane is injective save at the origin. Let $m_1$ and $m_2$ be two points in $\varphi(I)$. If $m_1$ and $m_2$ correspond to the same intersecting segment then $\varphi(m_1)$ and $\varphi(m_2)$ are projected to points different distances from the origin. Otherwise, the only case in which $m_1$ and $m_2$ do not make different angles with the origin is if at least one corresponds to the final segment. However, it is easy to see that the line in the $xy$-plane corresponding to the final segment remains disjoint from the rest of $\varphi(I)$ under projection because we chose the final segment to be rightmost in the $xz$-plane. The projections of the images under $\varphi$ of intersecting segments are straight lines through the origin, so they intersect in an \"ubercrossing at the origin in the $xy$-plane.
\end{proof}

An example of this method applied to the trefoil knot appears in Figure \ref{fig:algorithm}.

\begin{figure}[h]
	\begin{subfigure}[b]{.2\textwidth}
		\centering
		{\includegraphics[height=30mm]{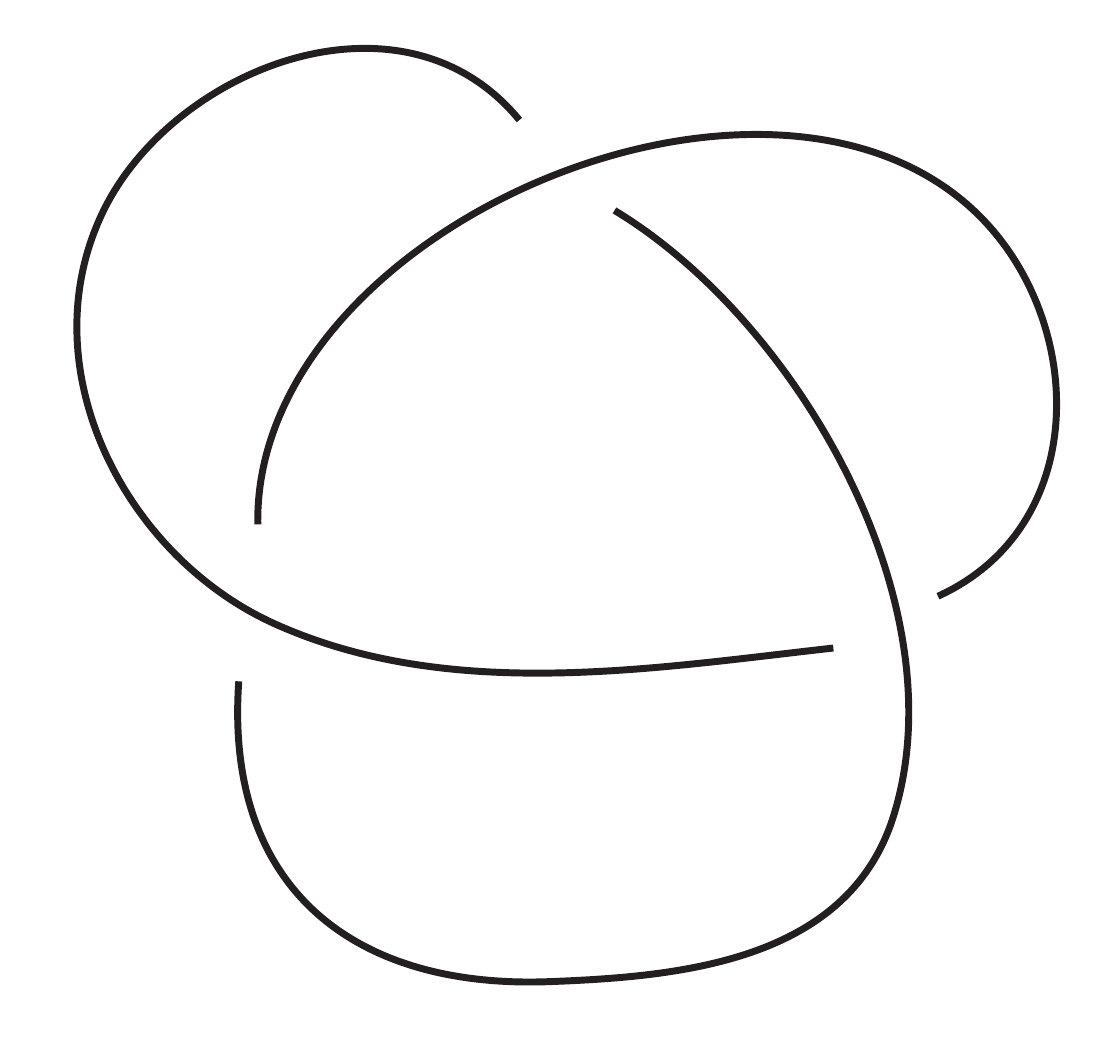}}
		\label{fig:algorithm1}
		\caption{}
	\end{subfigure}
	\begin{subfigure}[b]{.2\textwidth}
		\centering
		{\includegraphics[height=30mm]{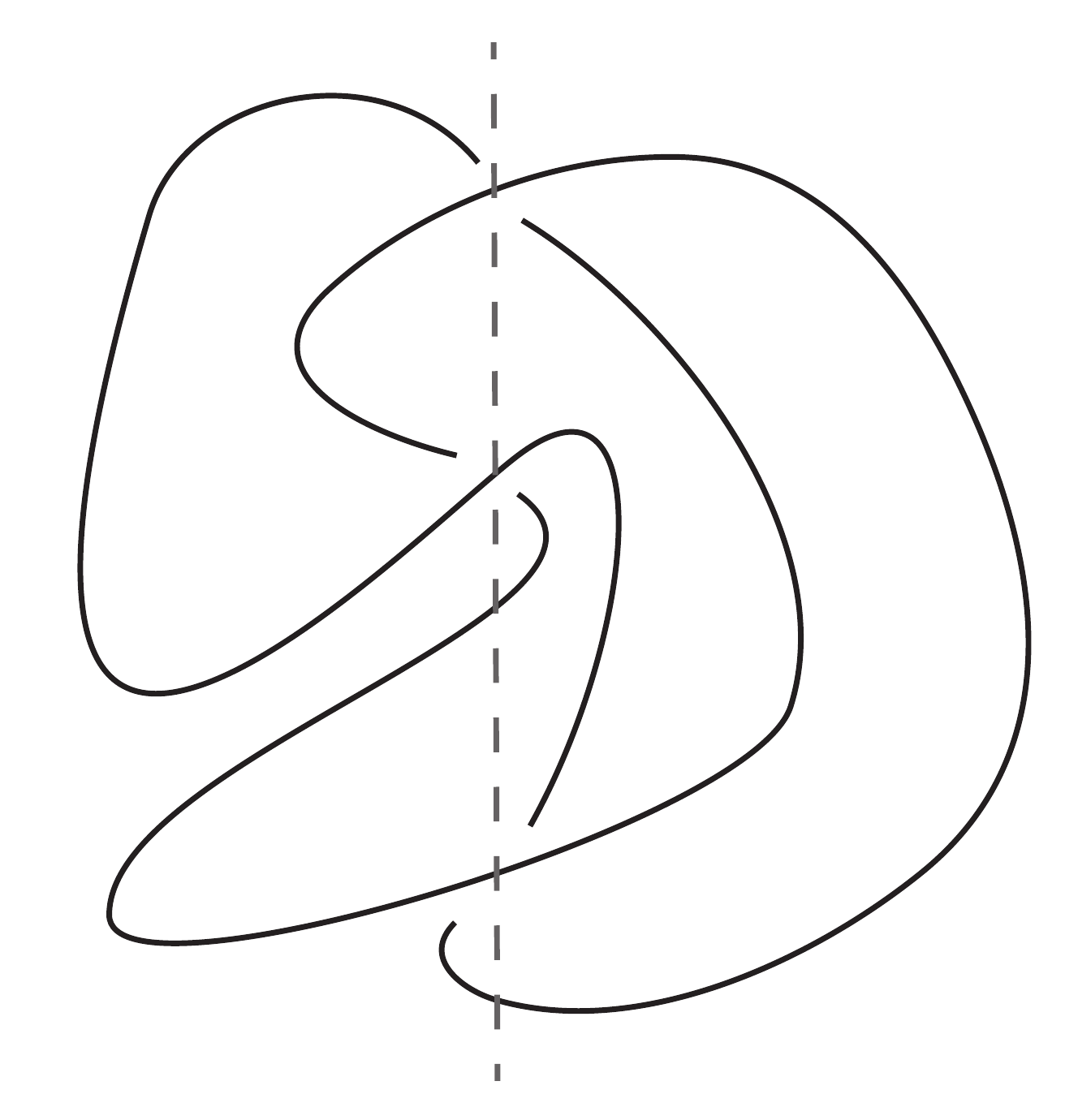}}
		\label{fig:algorithm2}
		\caption{}
	\end{subfigure}
	\begin{subfigure}[b]{.2\textwidth}
		\centering
		{\includegraphics[height=30mm]{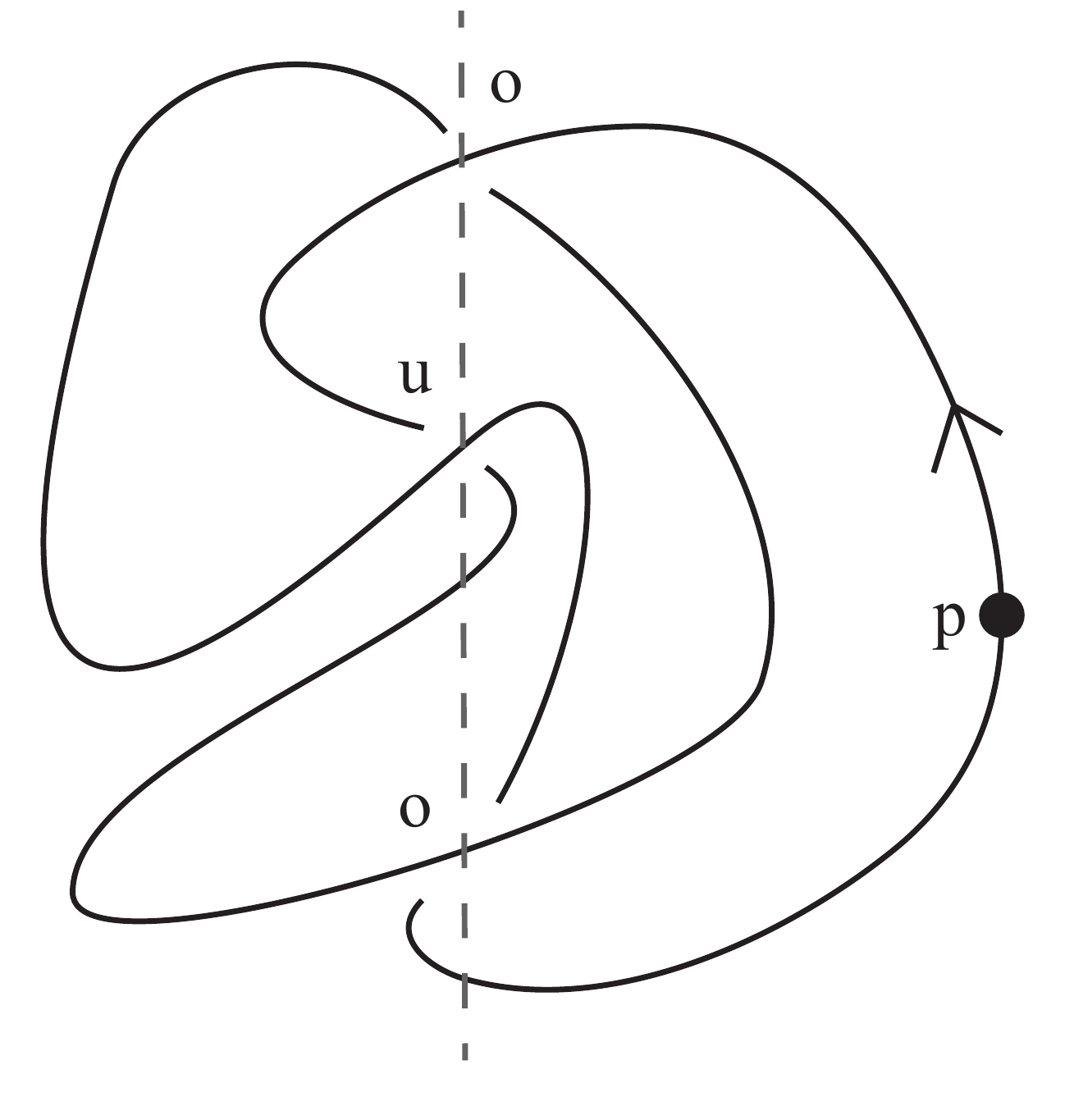}}
		\label{fig:algorithm3}
		\caption{}
	\end{subfigure}
\vspace{3mm}
	
	\begin{subfigure}[b]{.2\textwidth}
		\centering
		{\includegraphics[height=30mm]{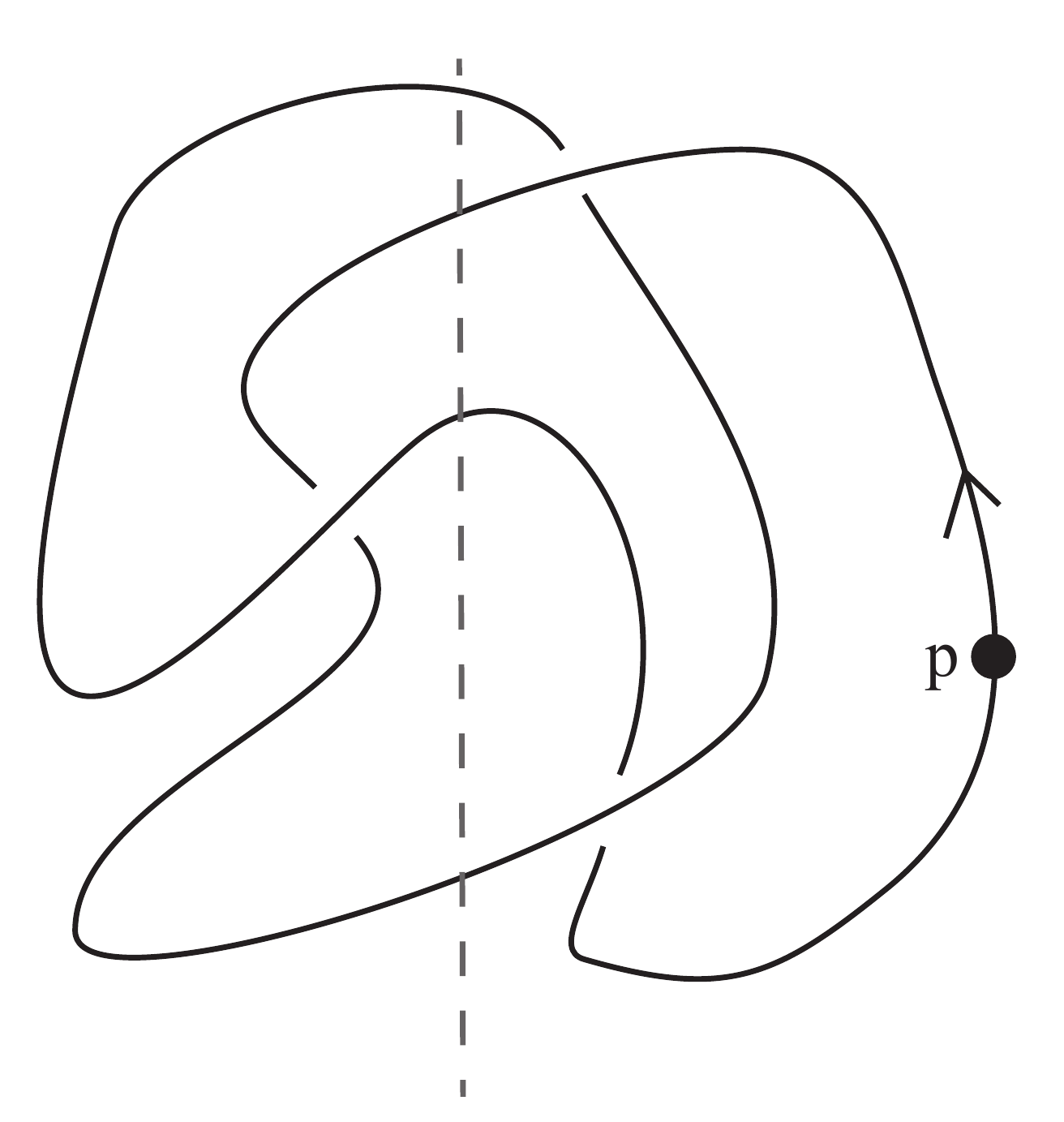}}
		\caption{}
		\label{fig:algorithm3b}
	\end{subfigure}
	\begin{subfigure}[b]{.2\textwidth}
		\centering
		{\includegraphics[height=30mm]{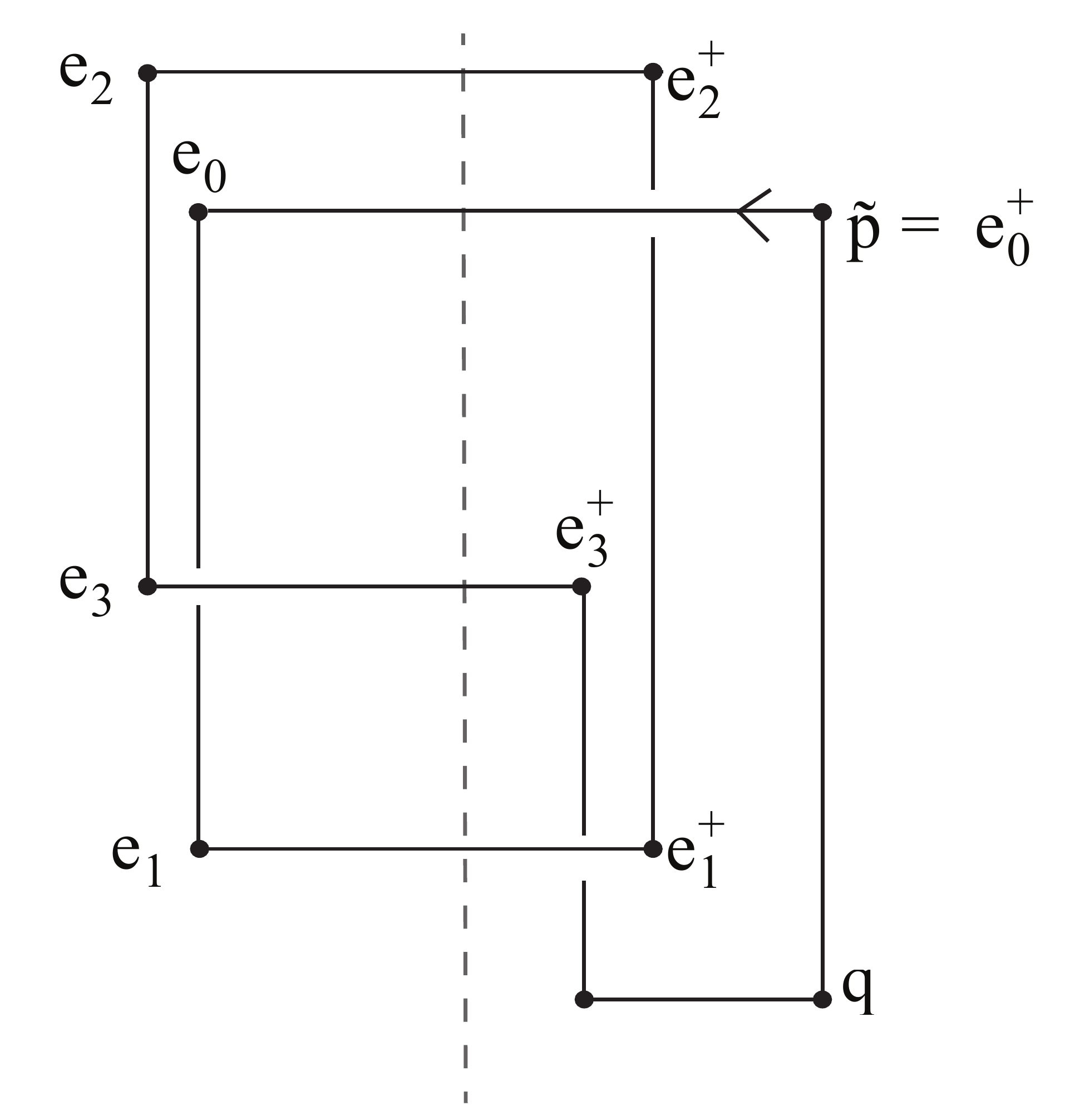}}
		\caption{}
		\label{fig:algorithm4}
	\end{subfigure}
	\begin{subfigure}[b]{.2\textwidth}
		\centering
		{\includegraphics[height=30mm]{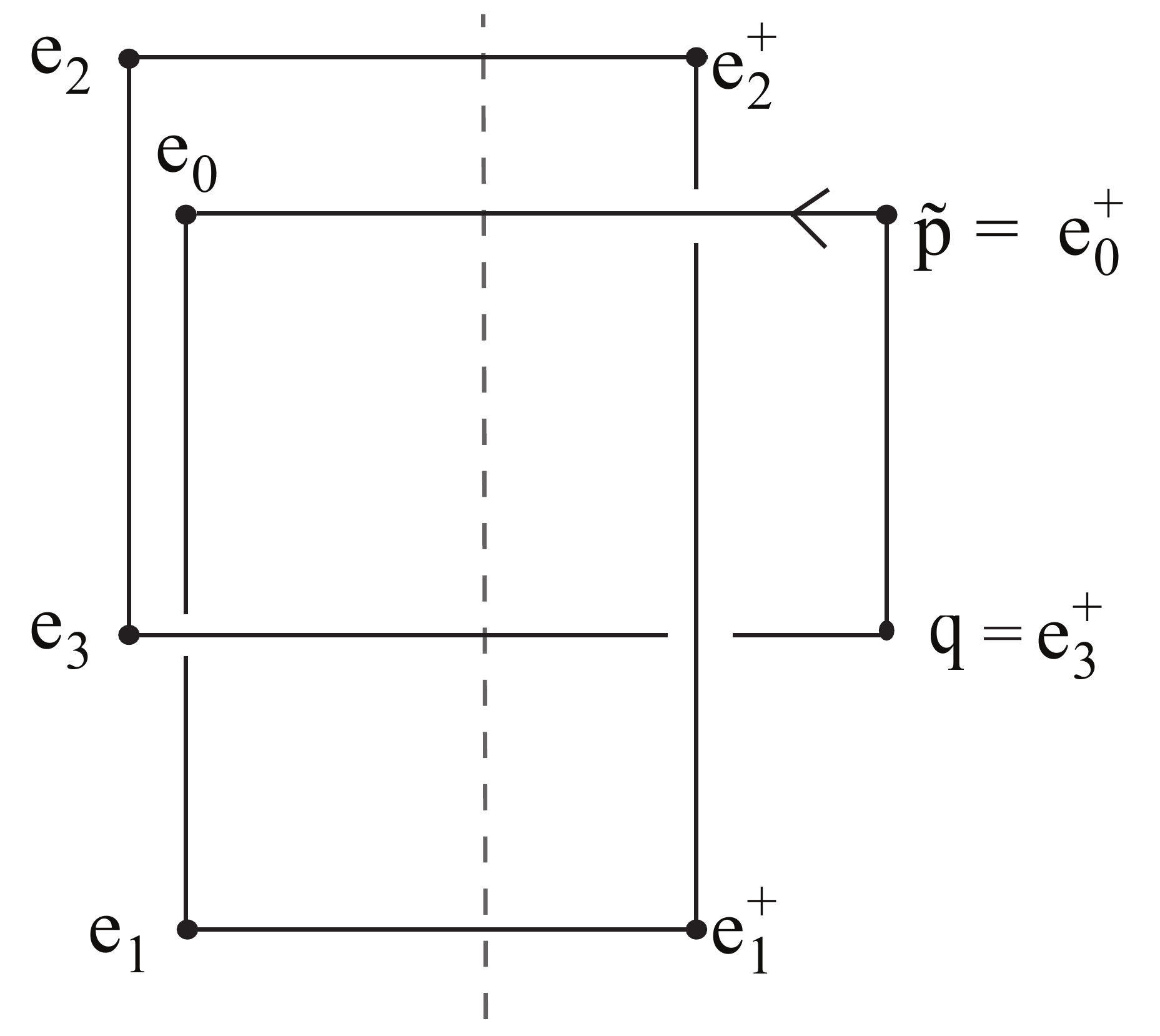}}
		\caption{}
		\label{fig:algorithm4b}
	\end{subfigure}
\vspace{3mm}
	
	\begin{subfigure}[b]{.2\textwidth}
		\centering
		{\includegraphics[height=40mm]{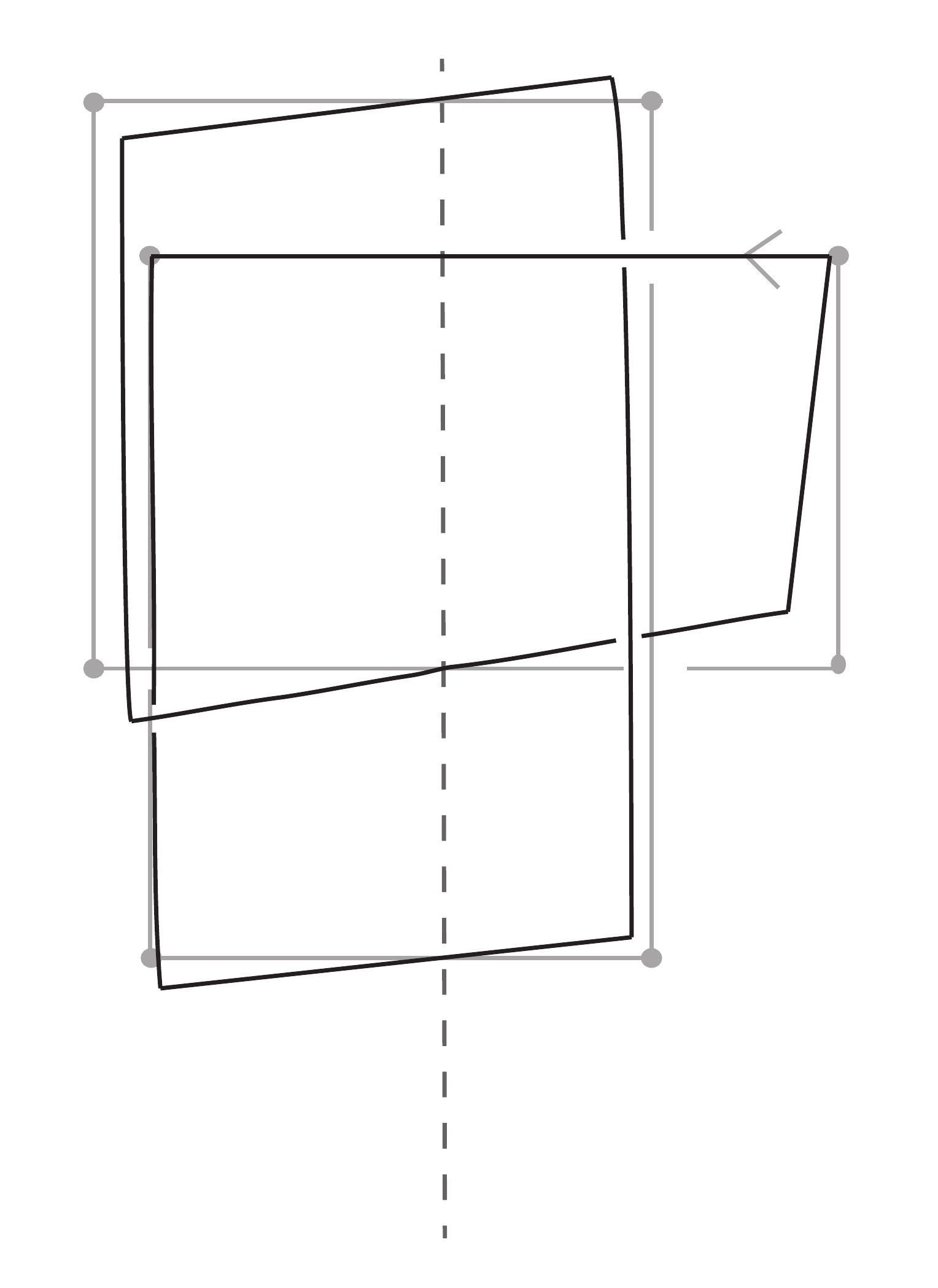}}
		\caption{}
		\label{fig:algorithm4c}
	\end{subfigure}
	\begin{subfigure}[b]{.2\textwidth}
		\centering
		{\includegraphics[height=40mm]{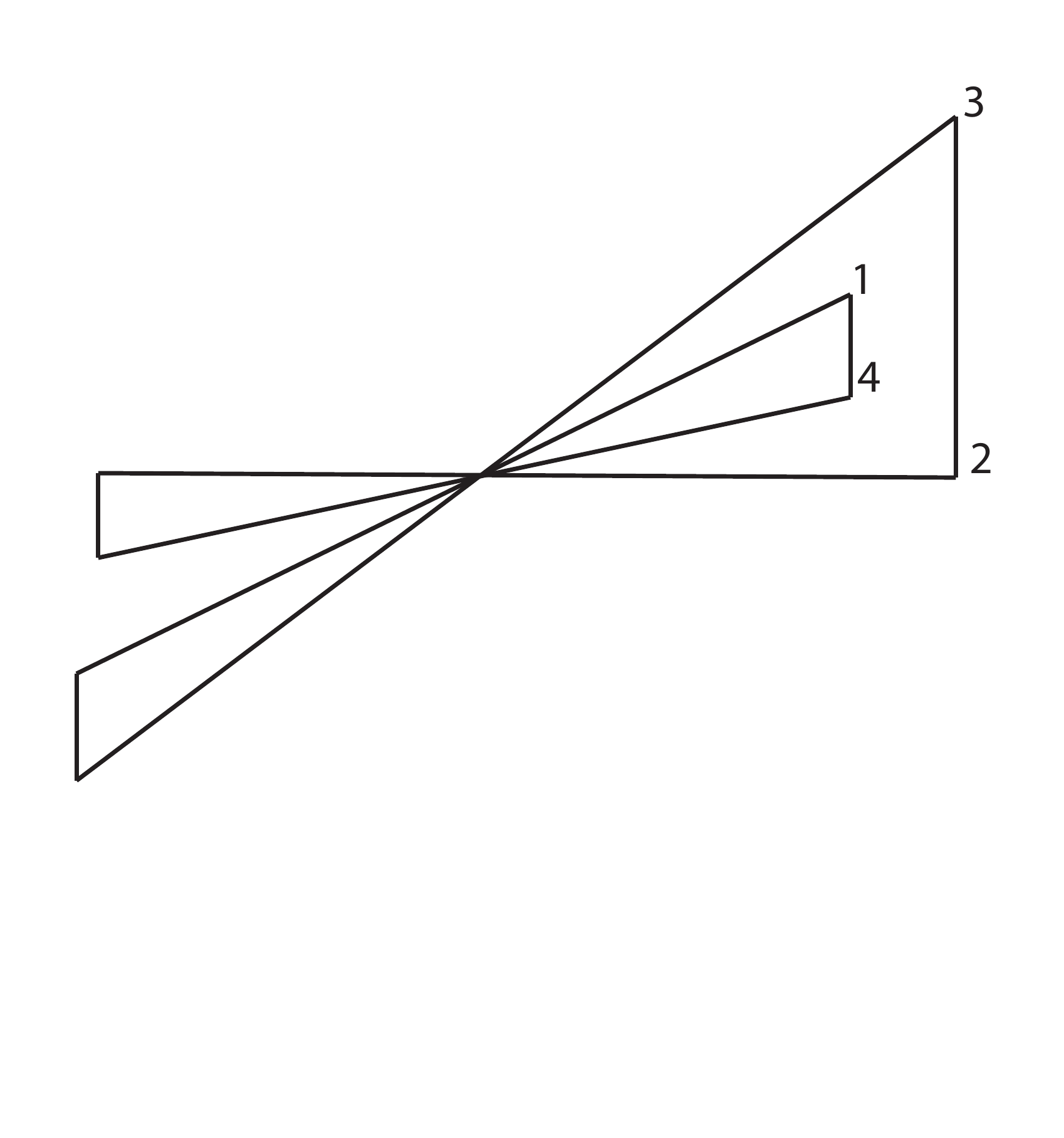}}
		\caption{}
		\label{fig:algorithm5b}
	\end{subfigure}

\caption{(a) We begin with the trefoil.  (b) A projection of the trefoil so that all crossings are on a straight line, indicated by the 
dashed line.  (c) The point $p$ 
and an orientation are selected, and the crossings are labelled $o$ or $u$.  (d)  After planar isotopy so that $o$-crossings are right of the axis, 
$u$-crossings are left of the axis. (e) Rectilinearizing and labeling the endpoints of the intersecting segments and the final segment. 
(f) Straightening the arcs from $\tilde{p}$ and $q$ to $A$.  (g) Rotating intersecting arcs and connecting their endpoints to obtain the conformation in $\R^3$. (h) Looking down $A$ yields the  \"{u}bercrossing projection.  }
\label{fig:algorithm}
\end{figure}


\begin{cor}\label{cor:uberlinks} Every link $L$ has an \"{u}bercrossing projection.
\end{cor}

\begin{proof} Label each link component $L_1, L_2, \dots, L_\ell$.  Pick an orientation on each link component, and a base point 
$p_i$ which is rightmost in the $i$-th component.  Beginning at $p_1$ and following the orientation of $L_1$, label each crossing $o$ or $u$ as above.  
Continue with $L_2$, and so on.  Isotope the $o$-crossings to the right and the $u$-crossings to the left of a line $A$, then
rectilinearize each component as above.  Label $q_i$ and $\tilde{p_i}$ on each component as above, so 
that $q_i$ and $\tilde{p_i}$ bound the rightmost vertical segment of $L_i$.  Notice that these final segments may have crossings, but 
do not cross any segments from their own component.

Label the endpoints of the intersecting segments $e_j^\pm$ as above, beginning with $L_1$, then $L_2$, and so on.  Note that this labeling 
is consecutive, in the sense that if $e_1^\pm, \dots, e_k^\pm$ are in $L_1$, then $e_{k+1}^\pm$ is the first endpoint pair of $L_2$.  
We can then perform the algorithm on each link component. Projecting down the axis of rotation gives an \"{u}bercrossing projection.
\end{proof}

Note that the \"{u}bercrossing projection of a link obtained from this algorithm looks like \"{u}bercrossing projections of several knots of the form shown in Figure \ref{fig:corpetal}(a) offset and stacked on top of each other. 

Theorem \ref{thm:uberalgorithm} is now extended to show that every knot has a petal projection 
via a simple isotopy that turns the \"{u}bercrossing projection 
obtained from the above theorem into a petal projection.    
It is easy to verify that a petal projection of a nontrivial knot must have an odd number of loops. Notice, however, that these results do not apply to links. In general, links do not have petal projections.  

\begin{cor} \label{cor:petal} Every knot $K$ has a petal projection.
\end{cor}

\begin{proof} The algorithm from the proof of Theorem \ref{thm:uberalgorithm} yields an \"{u}bercrossing knot with $n$ loops, where $n$ is the 
number of intersecting strands in the projection $\widetilde{P}$.  This number is always even.  Further, this \"{u}bercrossing projection 
has $n-1$ innermost loops, and one loop that nests $\frac{n}{2} - 1$ loops, as in Figure \ref{fig:corpetal}(a).  This outermost loop can be folded 
over the \"{u}bercrossing, as in Figure \ref{fig:corpetal}\subref{fig:corpetal2}, to yield a petal projection with $n+1$ petals.
\end{proof}

\begin{figure}[h]
\centering
	\begin{subfigure}[b]{.2\textwidth}
		\centering
		\includegraphics[height=30mm]{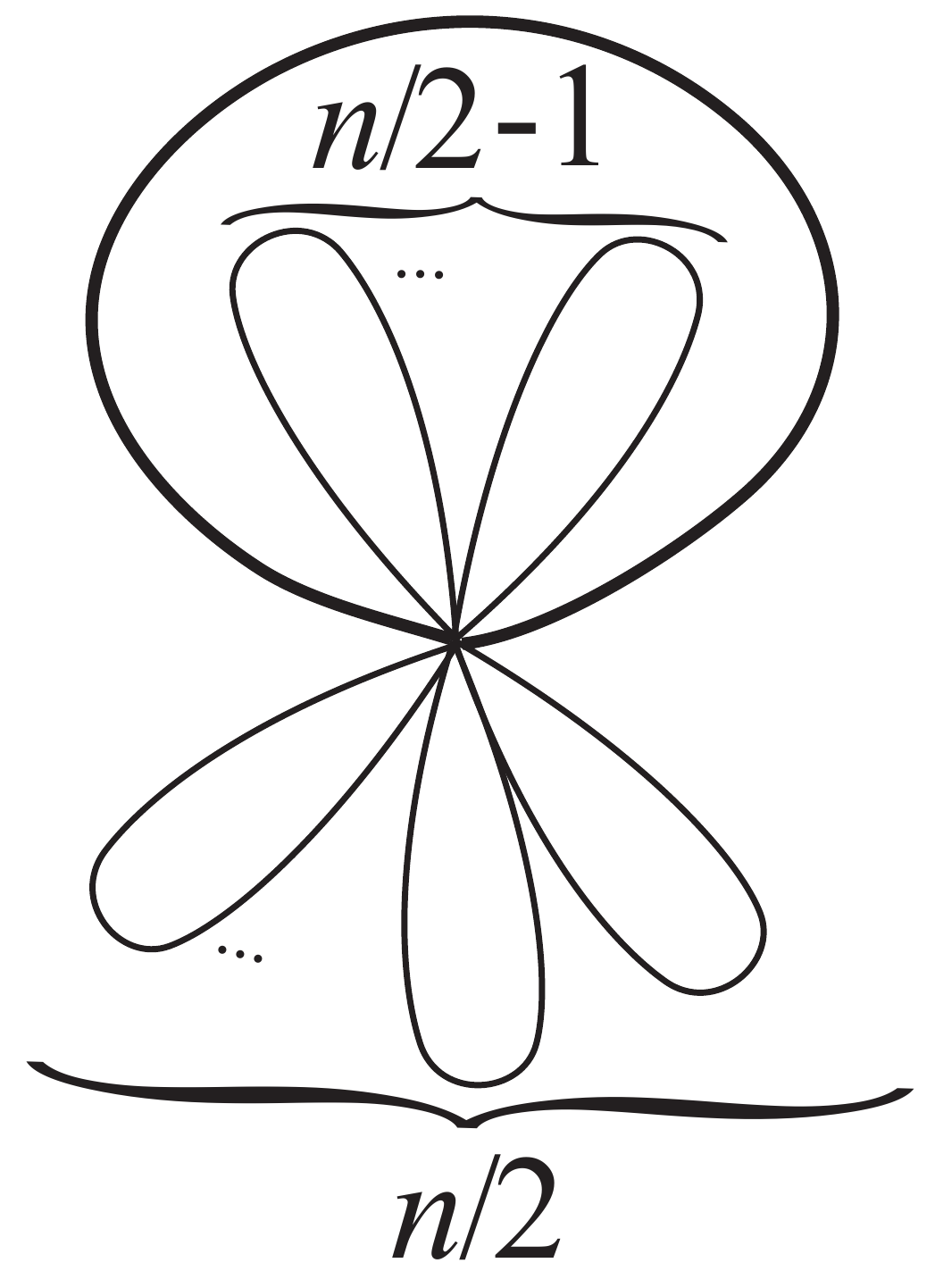}
		\label{fig:corpetal1}\caption{}
	\end{subfigure}
	\begin{subfigure}[b]{.2\textwidth}
		\centering
		\includegraphics[height=30mm]{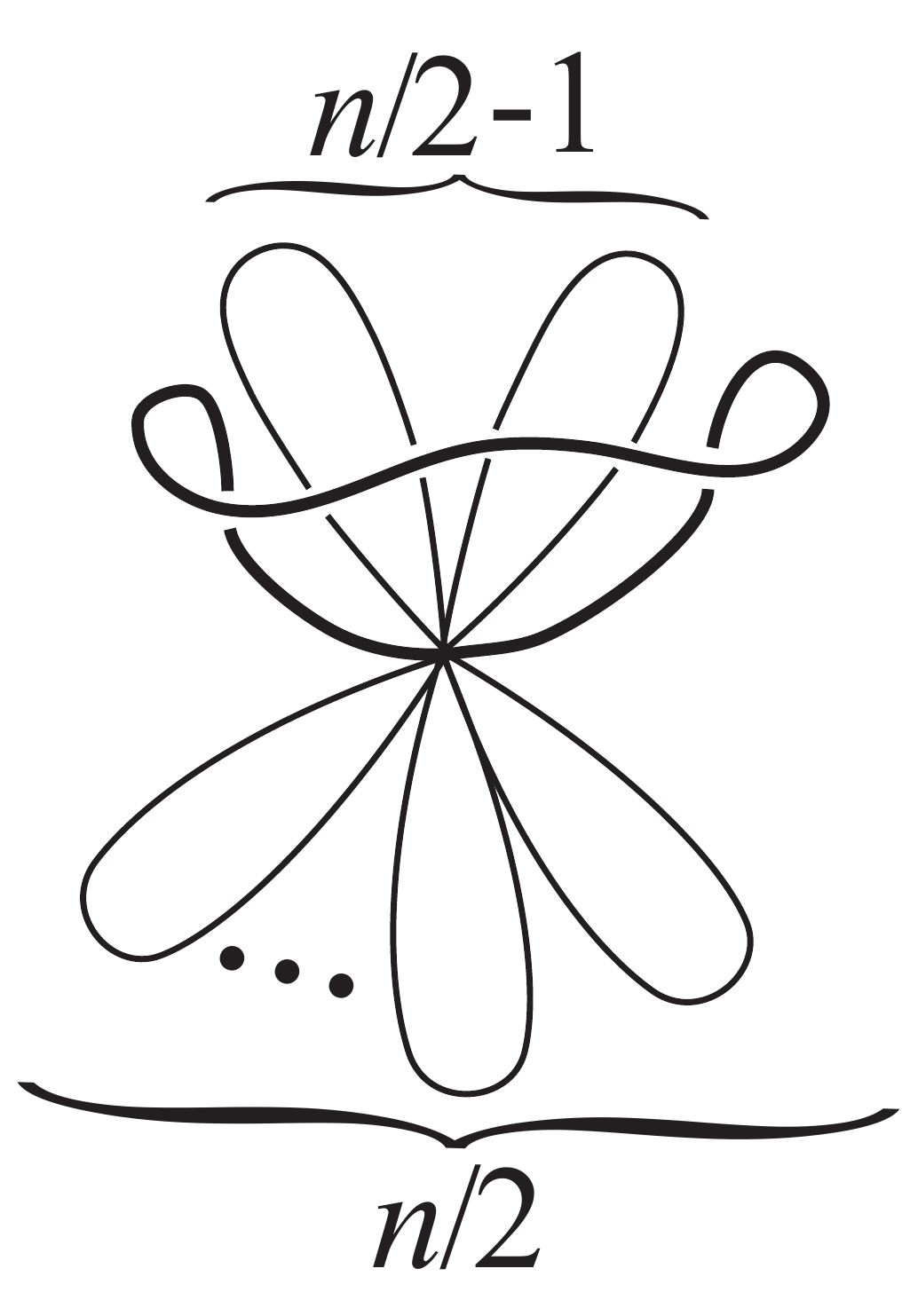} 
		\label{fig:corpetal1b}\caption{}
	\end{subfigure}
	\begin{subfigure}[b]{.2\textwidth}
		\centering
		\includegraphics[height=30mm]{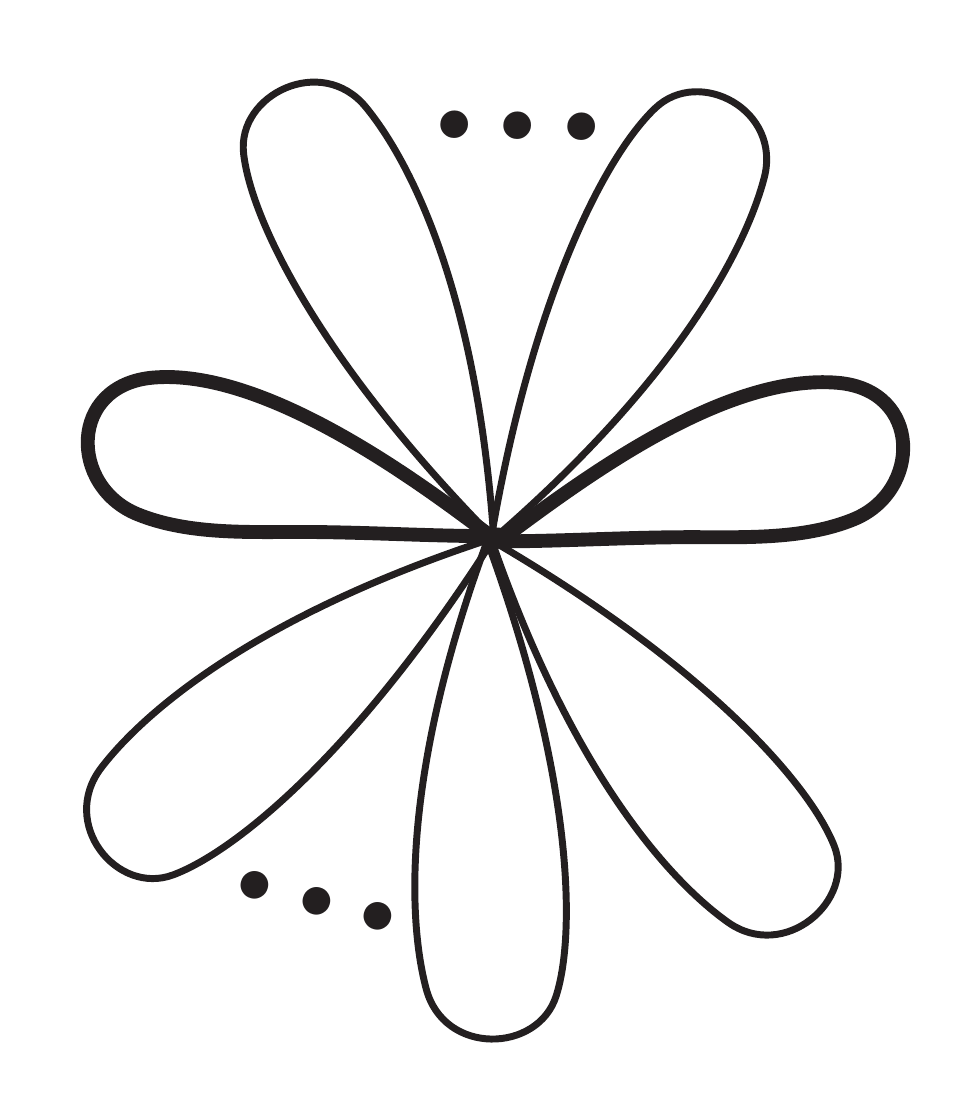}
		\label{fig:corpetal2}\caption{}
	\end{subfigure}
\caption{Turning a pre-petal projection obtained from Theorem \ref{thm:uberalgorithm} into a petal projection by folding the nesting loop over the middle of the \"{u}bercrossing projection.}
\label{fig:corpetal}
\end{figure}
\qquad

Note that the reverse of this isotopy shows that $\ub(K)<p(K)$.  Given a petal projection with $p$ loops, we can move to an \"ubercrossing projection with one less strand by pulling off the top strand so 
that it is a nesting loop containing $\frac{p(K)-1}{2}$ loops. Hereafter we shall call such a projection a \emph{pre-petal projection}.

The next theorem examines how petal number behaves under composition.

\begin{thm}\label{compostionthm}
If $K_1$ and $K_2$ are knots, then
\[
p(K_1\#K_2)\leq p(K_1)+p(K_2)-1.
\]
\end{thm}

\begin{figure}
\centering
\includegraphics[height=30mm]{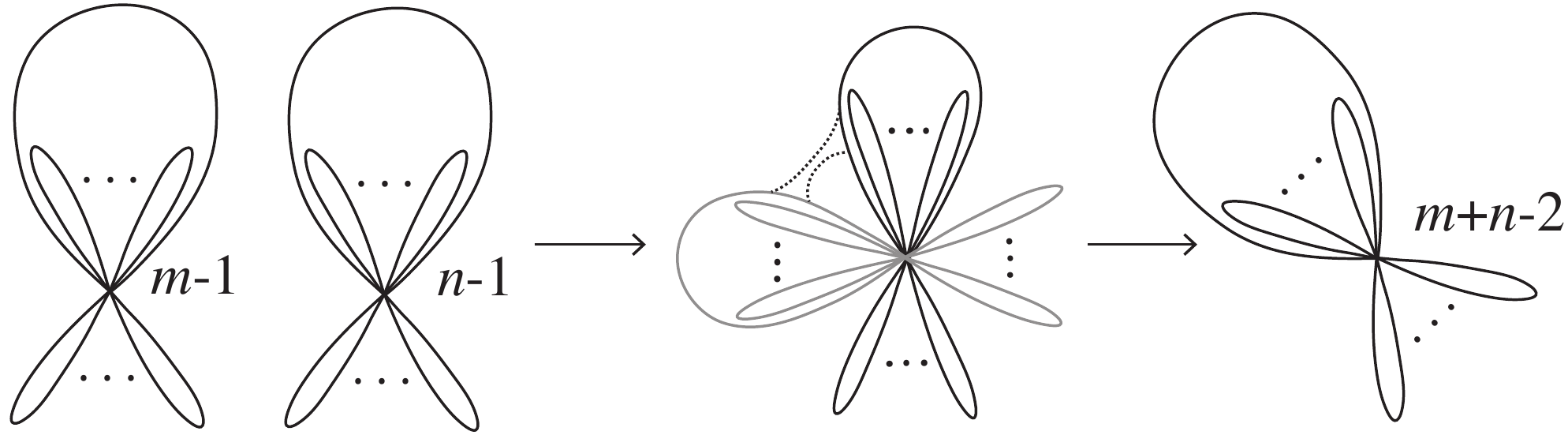}
\caption{Stacking pre-petal projections to bound petal numbers of compositions.}
\label{fig:composition}
\end{figure}

\begin{proof}
Let $m=p(K_1)$ and $n=p(K_2)$. Take petal projections of $K_1$ and $K_2$ which realize $m$ and $n$, respectively. Lifting the top strands off of the \"ubercrossing gives a pre-petal projection of each knot, and we may stack these on top of each other as shown in Figure \ref{fig:composition}. We then compose the two knots along their nesting loops as shown, forming a pre-petal projection $m+n-2$ loops, which gives rise to a petal projection of $K_1\#K_2$ with $m+n-1$ loops.
\end{proof}

\vspace{5mm}


\section{invariants}


In this section, we relate $\ub(K)$ and $p(K)$ to stick number, braid index, and arc index.

\begin{thm}
Let $s(K)$ be the stick number of a knot $K$. Then the following inequality holds:
\[
s(K)\leq \frac{3p(K)-1}{2}.
\]
\end{thm}

\begin{proof}
$~$Let $p(K)=n$. By unfolding the top strand of the minimal petal projection of $K$, we can obtain a pre-petal diagram with $n-1$ loops (Figure~\ref{fig:petal}(a)). We can use this diagram to obtain a stick conformation of $K$ with $2(n-1)$ sticks (Figure~\ref{fig:petal}(b)). Assuming that the projection plane is the  $xy$-plane and the $z$-axis is vertical, exactly half of the sticks are horizontal. Looking along the line marked in Figure \ref{fig:petal}(b), we see a view of the stick conformation as shown in \ref{fig:petal}(c). We draw a vertical axis $A$ on this sideview of the stick conformation as shown in Figure ~\ref{fig:petal}(c), and
set the distance between level $k$ and $k+1$ of the axis be $1$ for all $1\leq k\leq n-2$. We then say that there are $n-1$ horizontal and $n-1$ tilted sticks. 

\begin{figure}[htb]
\centering
\begin{subfigure}[b]{.3\textwidth}
\centering
\includegraphics[height=35mm]{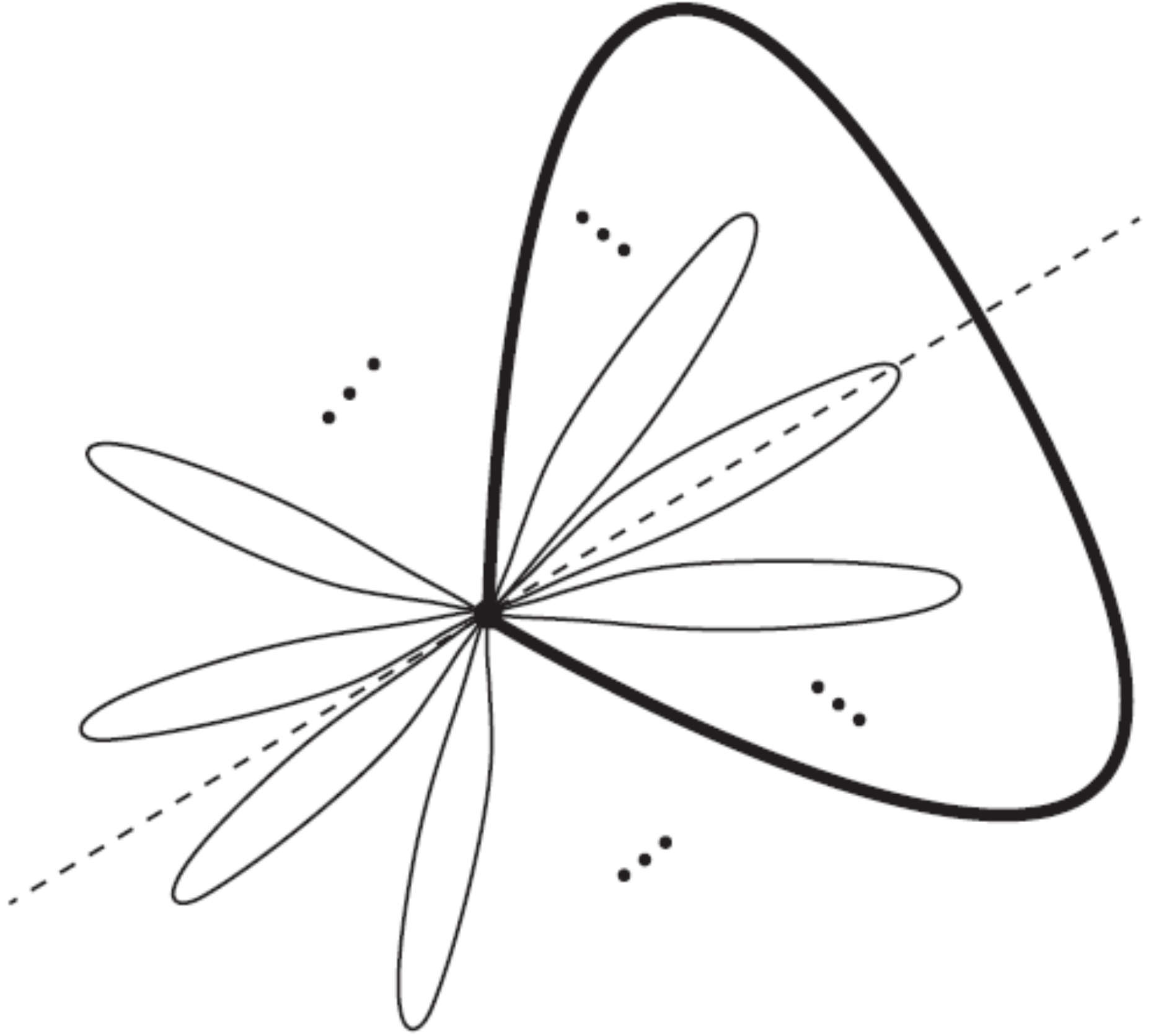}
\quad
\caption{}
\qquad
\end{subfigure}
\begin{subfigure}[b]{.3\textwidth}
\centering
\includegraphics[height=35mm]{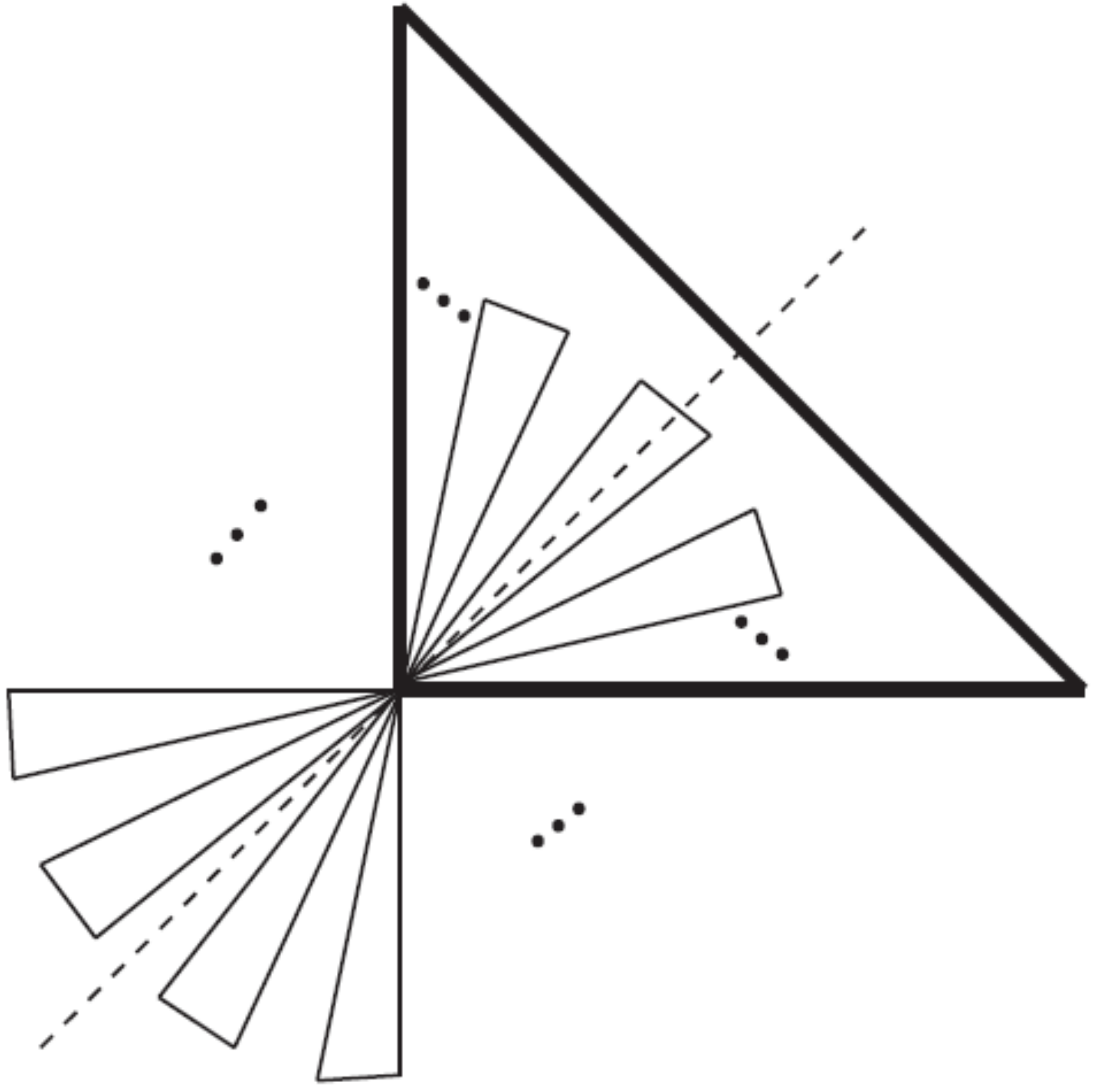}
\quad
\caption{}
\qquad
\end{subfigure}
\begin{subfigure}[b]{.3\textwidth}
\centering
\includegraphics[height=40mm]{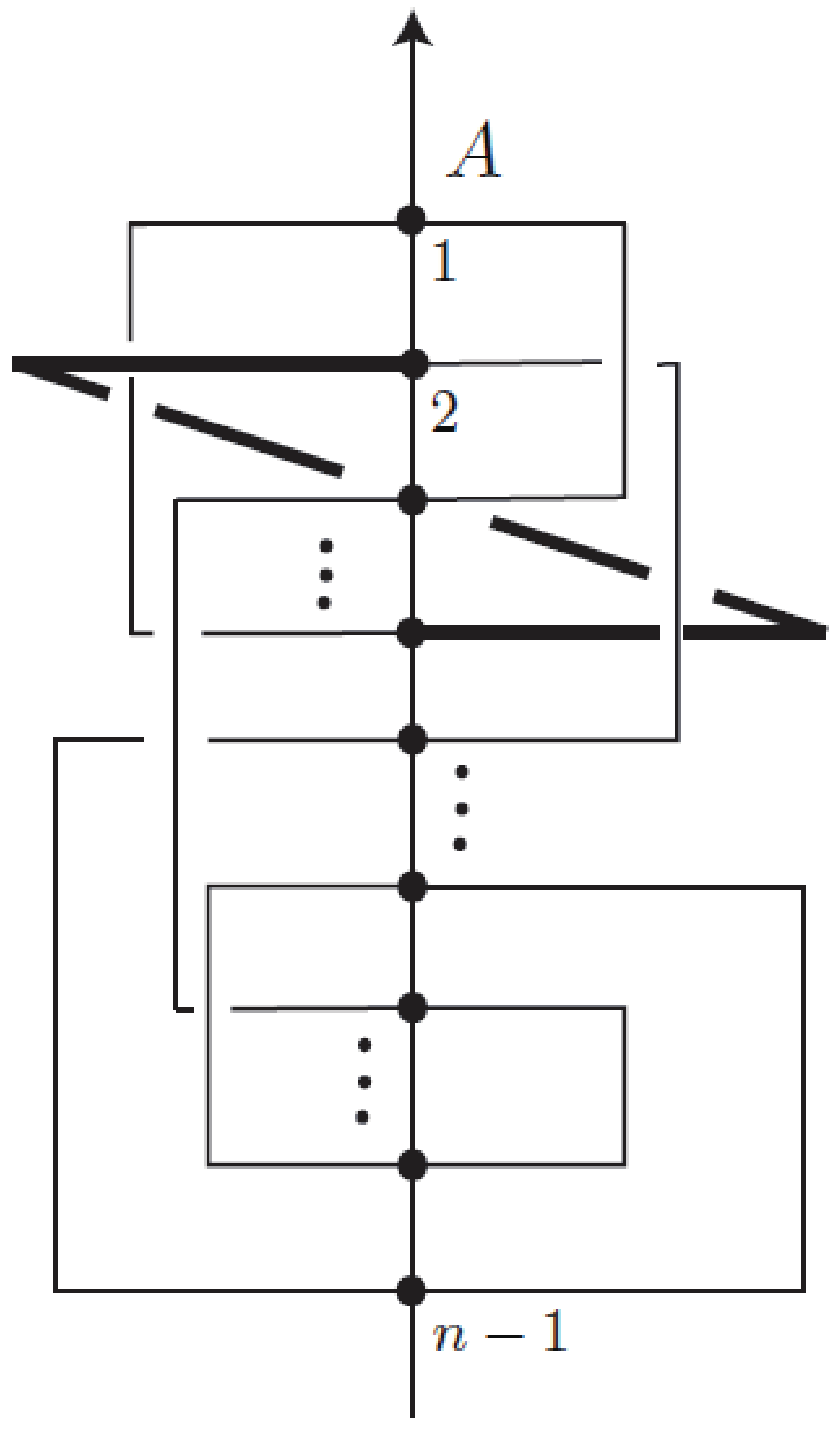}
\caption{}
\qquad
\end{subfigure}
\caption{Turning a pre-petal projection into a stick conformation.}\label{fig:petal}
\qquad
\end{figure}

Now we classify horizontal sticks of $K$ into three types. If a horizontal stick is a local maximum, that is, the two sticks connected to the horizontal stick point down (Figure \ref{fig:types}(a)), then we call it a \emph{max-horizontal stick}. Similarly, if a horizontal stick is a local minimum (Figure~\ref{fig:types}(b)), then we call it a \emph{min-horizontal stick}. If a horizontal stick is neither a maximum nor a minimum, then we call it an \emph{inflection-horizontal stick} (Figure~\ref{fig:types}(c)).

\begin{figure}[htb]
\centering
	\begin{subfigure}[b]{.25\textwidth}
		\includegraphics[height=40mm]{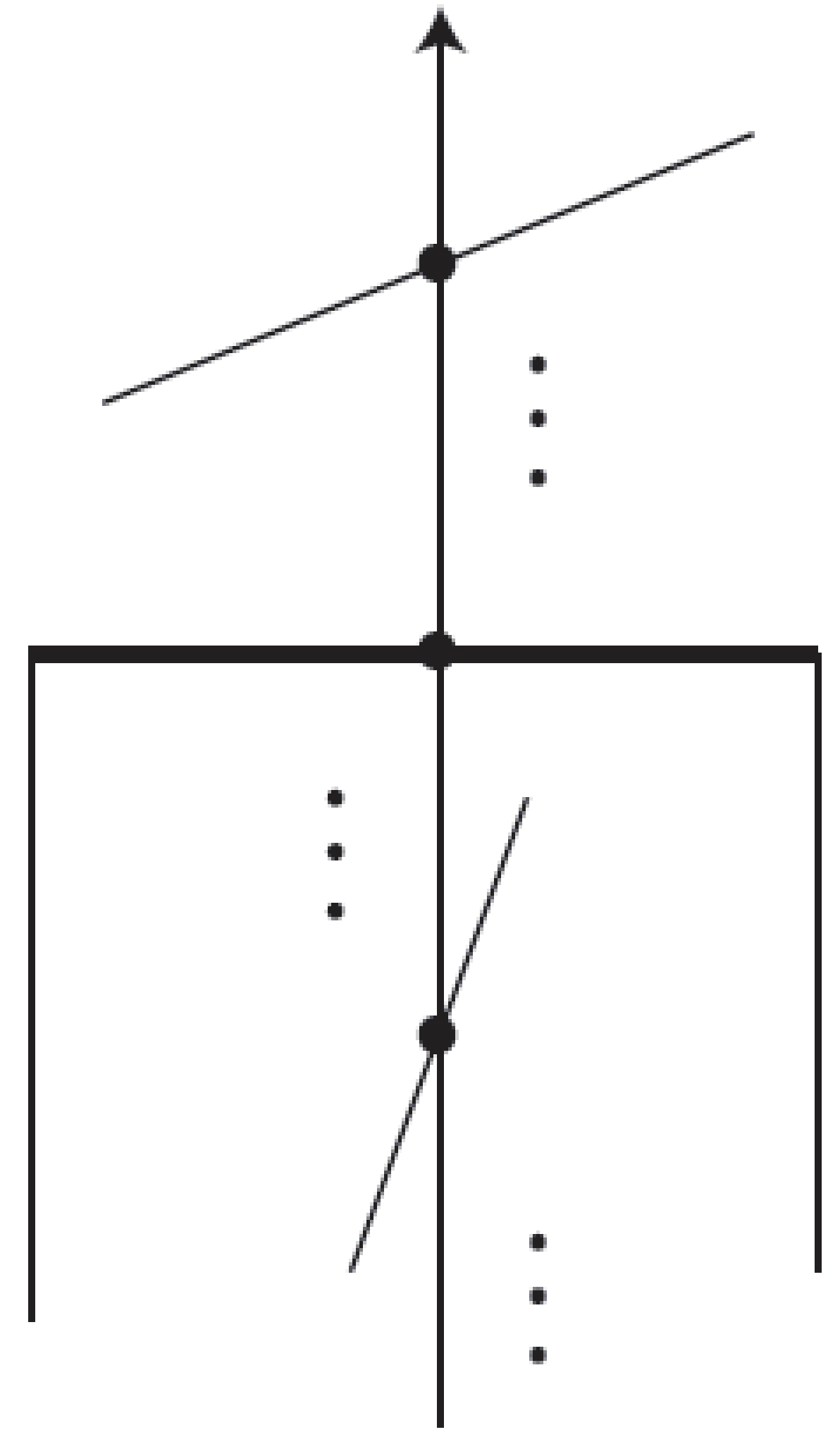}
		\caption{Max-horizontal stick.}
		\qquad
	\end{subfigure}
	\begin{subfigure}[b]{.25\textwidth}
		\includegraphics[height=40mm]{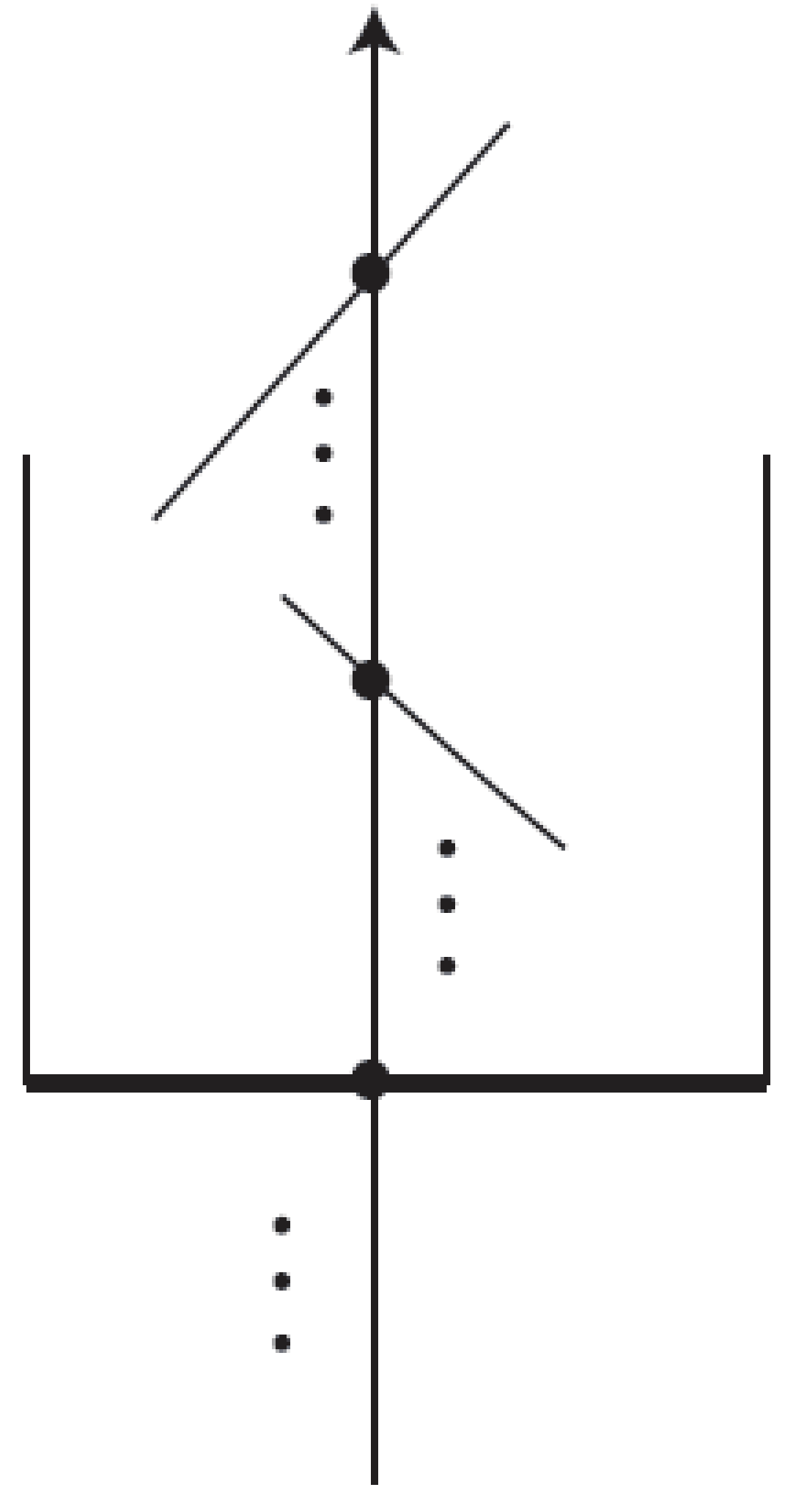}
		\caption{Min-horizontal stick.}\label{}
		\qquad
	\end{subfigure}
	\begin{subfigure}[b]{.25\textwidth}
	\includegraphics[height=40mm]{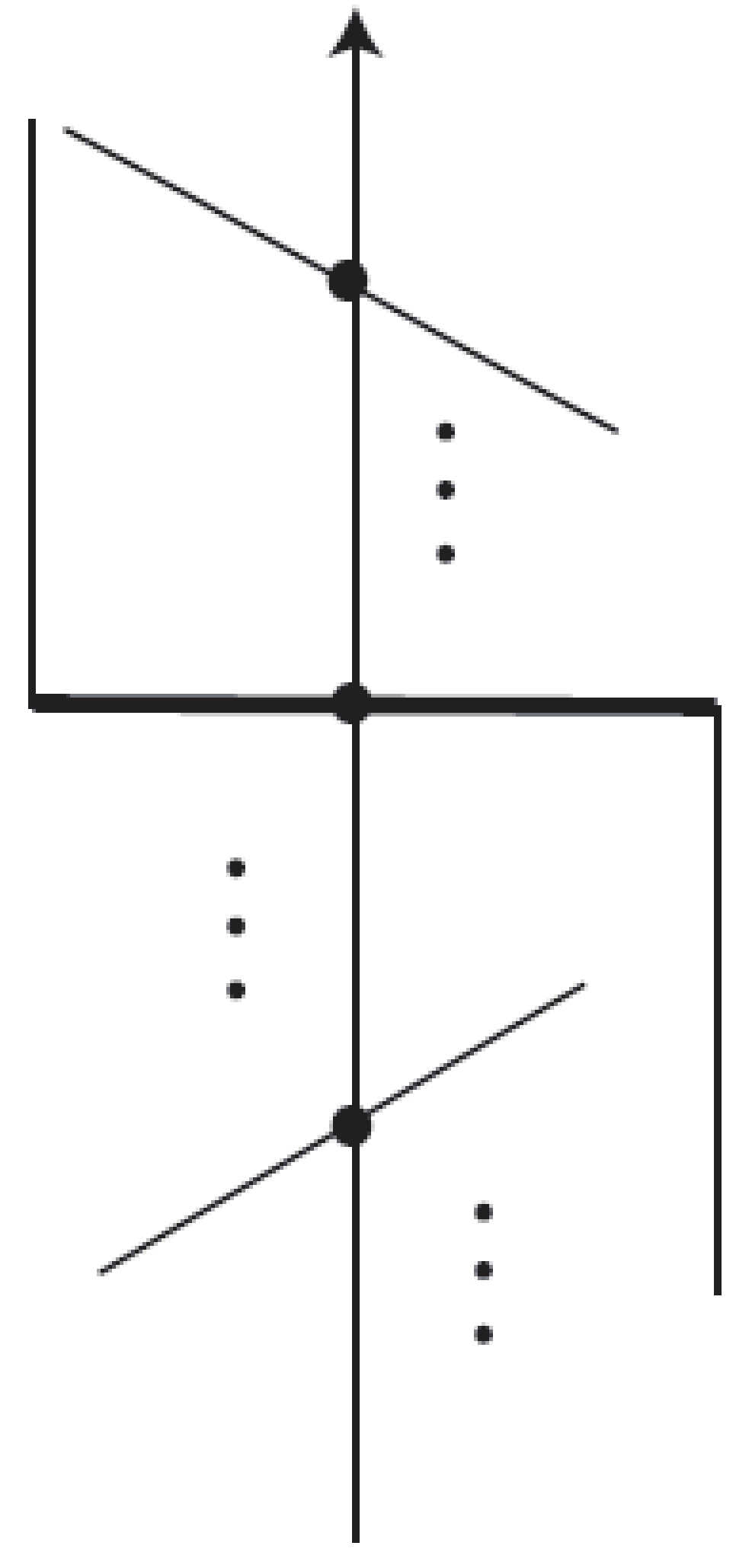}
	\caption{Inflection-horizontal stick.}\label{}
	\qquad
	\end{subfigure}
\caption{Types of horizontal sticks.}\label{fig:types}
\qquad
\end{figure}

Let $e_{k}$ denote the horizontal stick at  level $k$.  Suppose that the two horizontal sticks $e_{k}$ and $e_{j}$ are connected by a tilted stick $e_{kj}$ and $k > j$. Let $e_{k}'$ and $e_{j}'$ be the closure of the components of $e_{k}-A$ and  $e_{j}-A$ that intersect $e_{kj}$. Note that the three segments $e_{k}'$, $e_{kj}$ and $e_{j}'$ determine a tetrahedron $T_{kj}$ (Figure~\ref{fig:tetra}(a)) that does not intersect $K$ except on $A$ and along $e_{k}'$, $e_{kj}$, $e_{j}'$ since they project to a non-nesting loop.

\begin{figure}[htb]
\centering
\begin{subfigure}[b]{.2\textwidth}
\includegraphics[height=50mm]{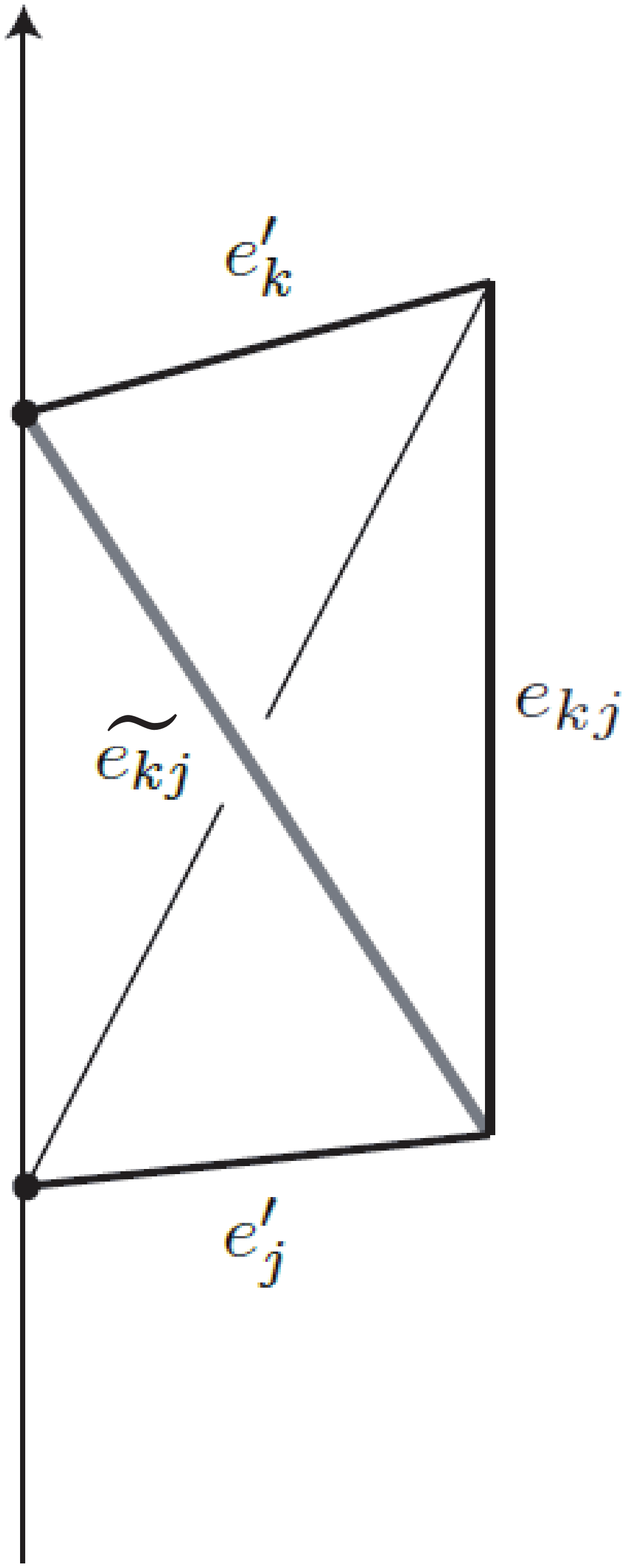}\hspace{2.5cm}
\caption{The tetrahedron $T_{kj}$.}
\qquad
\end{subfigure}
\begin{subfigure}[b]{.2\textwidth}
\includegraphics[height=50mm]{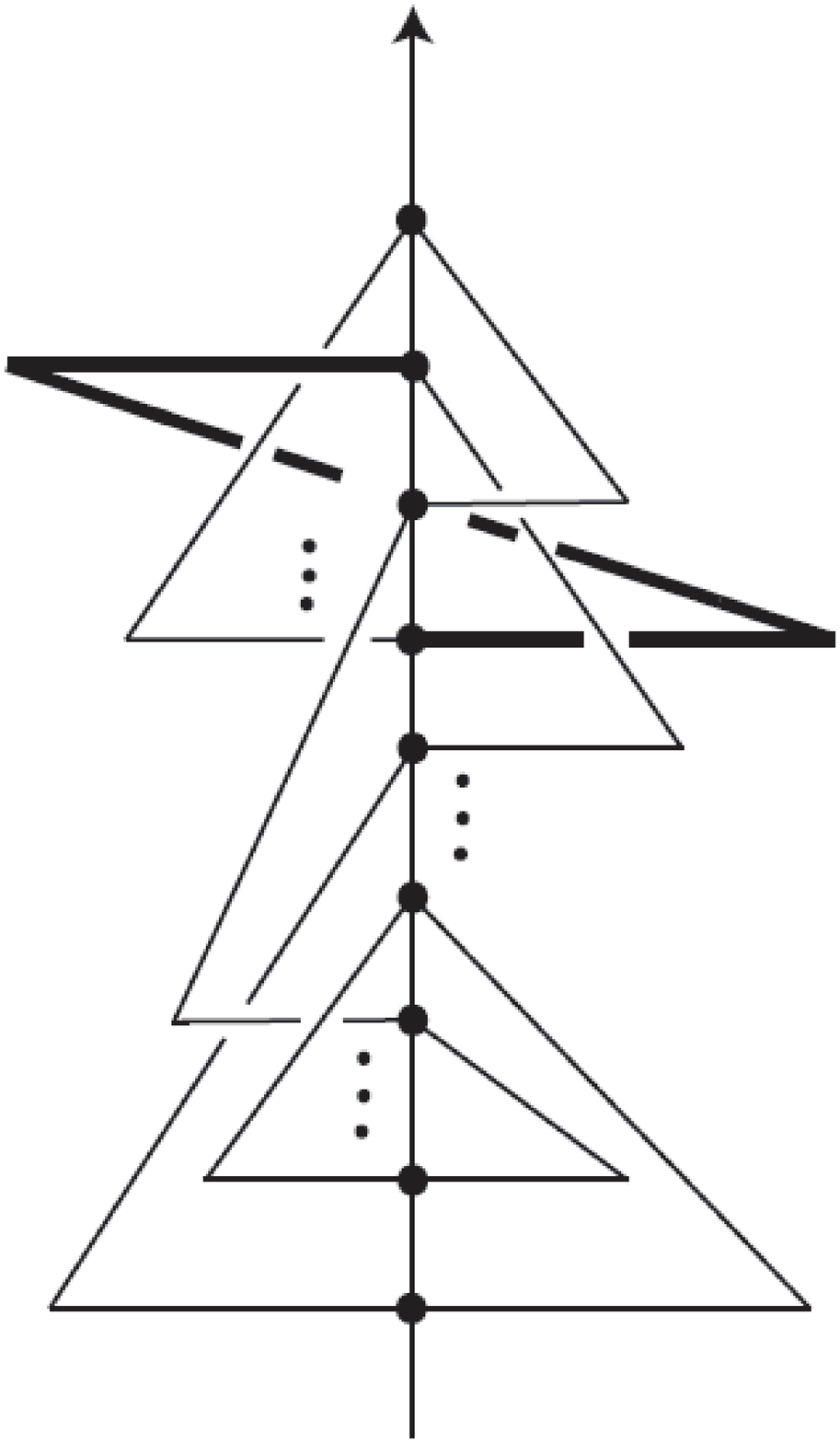}
\caption{New stick diagram.}
\qquad
\end{subfigure}
\caption{Eliminating horizontal sticks.}\label{fig:tetra}
\qquad
\end{figure}

We can exploit this tetrahedron by replacing the two sticks $e_{k}'$ and $e_{kj}$ by the edge of the tetrahedron that shares a face with them, which we denote $\widetilde{e_{kj}}$. (See Figure~\ref{fig:tetra}(a).)

After performing the above steps for all loops except the nesting loop, we obtain a new stick conformation as shown in Figure \ref{fig:tetra}(b). In this conformation, all max-horizontal sticks of the original stick diagram have been eliminated. 
    
Now we consider the case of inflection-horizontal sticks. Let $e_{h}''$ be the closure of a component of $e_{h}-A$ that was not eliminated by the above process. Let $e_{i}'$ and $e_{g}'$ be the two horizontal sticks to which it is attached by tilted sticks, the first below it and the second above it. We can choose an arbitrarily small number $\epsilon_{hi}>0$ such that $\epsilon_{hi}<\frac{1}{2}$min$\{d(\widetilde{e_{hi}},e)\}$ for any edge $e$ of $K$ except for $e_{h}''$ and $e_{i}'$. Consider the tubular neighborhood $N_{\epsilon_{hi}}$ of $\widetilde{e_{hi}}$ and the triangle $t_{hi}$ contained within it that is determined by $e_{hi}$, and the part of $e_{h}''$ at distance $\epsilon_{hi}$ from $A$. Since $N_{\epsilon_{hi}}$ does not intersect $K$ except along $e_{h}''$ and $e_{i}'$, $t_{hi}$ also does not intersect $K$ except along $e_{h}''$ and $e_{i}'$.  Then we note that the interior of the triangle determined by $\widetilde{e_{gh}}$ and $e_{h}''$ does not intersect $K$. Thus we can reduce the length of $e_{h}''$ to $\epsilon_{hi}$ without any crossing changes. Therefore, we can replace the  three edges $e_{h}''$, $\widetilde{e_{hi}}$  and $\widetilde{e_{gh}}$ by two edges, one edge joining the endpoint of the shortened $e_{h}''$ with the endpoint of  $e_{i}'$ not on $A$, and the other joining the same endpoint of the shortened $e_{h}''$ and the intersection of $\widetilde{e_{gh}}$ with $A$, as in Figure~\ref{fig:tubular}. If we apply the above method to all inflection-horizontal  sticks from bottommost to topmost, we can eliminate the inflection-horizontal sticks as well.
 
\begin{figure}[htb]
\centering
\includegraphics[height=80mm]{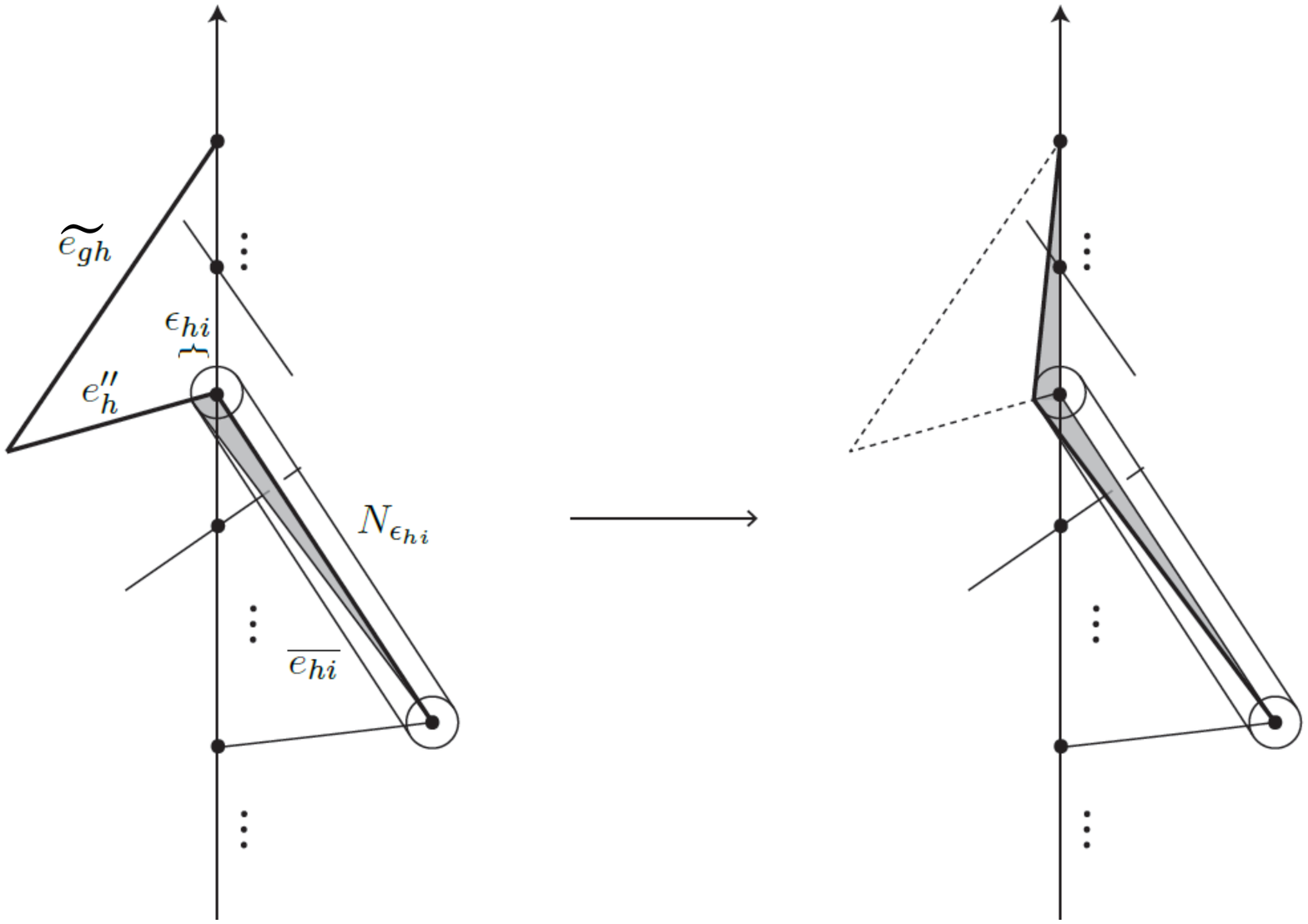}
\caption{Eliminating an inflection-horizontal stick.}\label{fig:tubular}
\qquad
\end{figure}

 
Note that we cannot eliminate the horizontal sticks corresponding to the nesting loop with this method, because the tetrahedron determined by the loop may intersect the other loops. Hence we can eliminate all max-horizontal and inflection-horizontal sticks of $K$ except possibly the top and bottom horizontal sticks of the nesting loop. (See Figure \ref{fig:example}.)

Suppose that before our stick elimination, our stick conformation of $K$ contains at least one inflection-horizontal stick. Since the number of min-horizontal sticks is equal to the number of max-horizontal sticks, the existence of at least one inflection-horizontal stick implies that the number of min-horizontal sticks is less than $\frac{n-1}{2}$, and the number of max-horizontal sticks and inflection sticks is at least $\frac{n+1}{2}$. If neither of the horizontal sticks of the nesting loop is min-horizontal, then we can eliminate at least $\frac{n+1}{2}-2=\frac{n-3}{2}$ sticks; otherwise we can eliminate more.

Now suppose that before our stick elimination, there are no inflection-horizontal sticks. Then the top and bottom horizontal sticks of the nesting loop must be max-horizontal and min-horizontal respectively, and the top one is the only max-horizontal stick in the conformation that we cannot eliminate. Thus, in this case, we can also reduce at least  $\frac{n-3}{2}$ sticks of $K$. Therefore we get the following inequality: $s(K)\leq 2(n-1)-(\frac{n-3}{2})=\frac{3}{2}n-\frac{1}{2}=\frac{3p(K)-1}{2}$.
\end{proof}

\begin{figure}[htb]
\centering
\includegraphics[height=50mm]{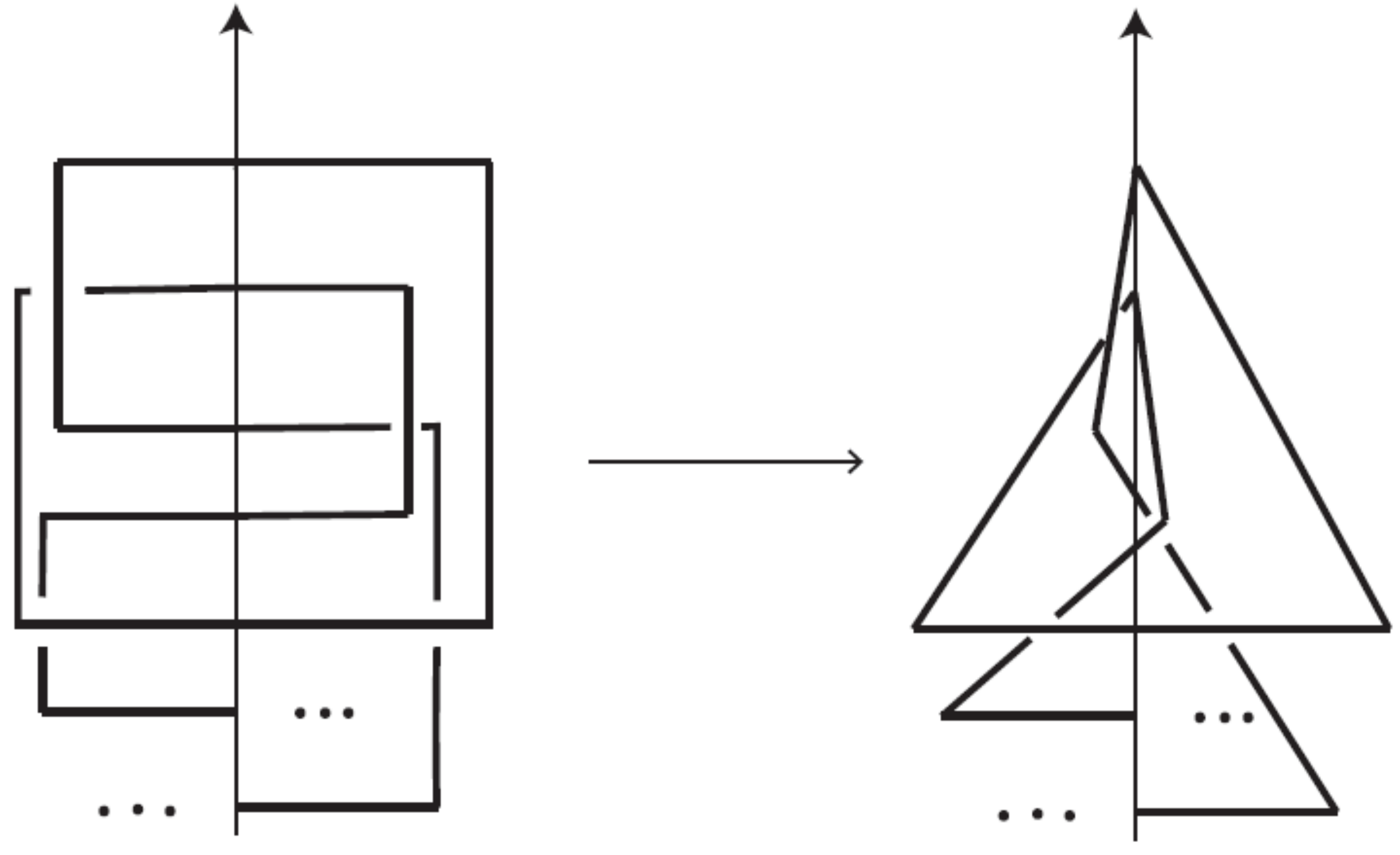}
\caption{An example of the elimination of non min-horizontal sticks of $K$.}\label{fig:example}
\qquad
\end{figure}

We now apply this result to torus knots. By ~\cite{Jin}, the stick number of an $(r,s)$ torus knot with $2\leq r<s<2r$ is equal to $2s$. Thus we obtain the following corollary.
\medskip
\medskip
\medskip
\medskip
\medskip

\begin{cor} Let $T_{r,s}$ be an (r,s)-torus knot. Then if  $\hspace{0.1cm}2\leq r<s<2r$,
 $p(T_{r,s})\geq\frac{4s+1}{3}$. 
\end{cor}

At the end of this paper, we determine the petal number for $(r, r+1)$-torus knots exactly. The next lemma will prove useful.


\begin{lemma}\label{monogonlemma}
Any nontrivial \"ubercrossing projection of a knot contains at least three non-nesting monogons.
\end{lemma}
\begin{proof}
We embed the \"ubercrossing projection on $S^2$, partitioning the surface. In this partition, let $e$ represent the number of edges, $v$ the number of vertices, and $f$ the number of faces. Since $S^2$ has an Euler Characteristic of 2, we have $v-e+f=2$. This partition has one vertex at the \" ubercrossing; hence $v=1$. Let $p_i$ represent the number of regions with $i$ edges.

Since each edge borders two faces, we have $p_1+2p_2+3p_3+...=2e$. Euler's formula then gives $2-(p_1+2p_2+3p_3+4p_4+...)+2(p_1+p_2+p_3+p_4+...)=4$. Hence
\begin{equation}\label{Euler}
p_1=2+p_3+2p_4+3p_5+\cdots.
\end{equation}
Therefore the projection contains at least two monogons. However, for the partition to have only two monogons, the rest of the faces must be bigons by \eqref{Euler}, and one can check that all partitions whose faces are all bigons save for two monogons correspond to either the trivial knot or a link.

\begin{figure}[h]
\centering
\includegraphics[height=40mm]{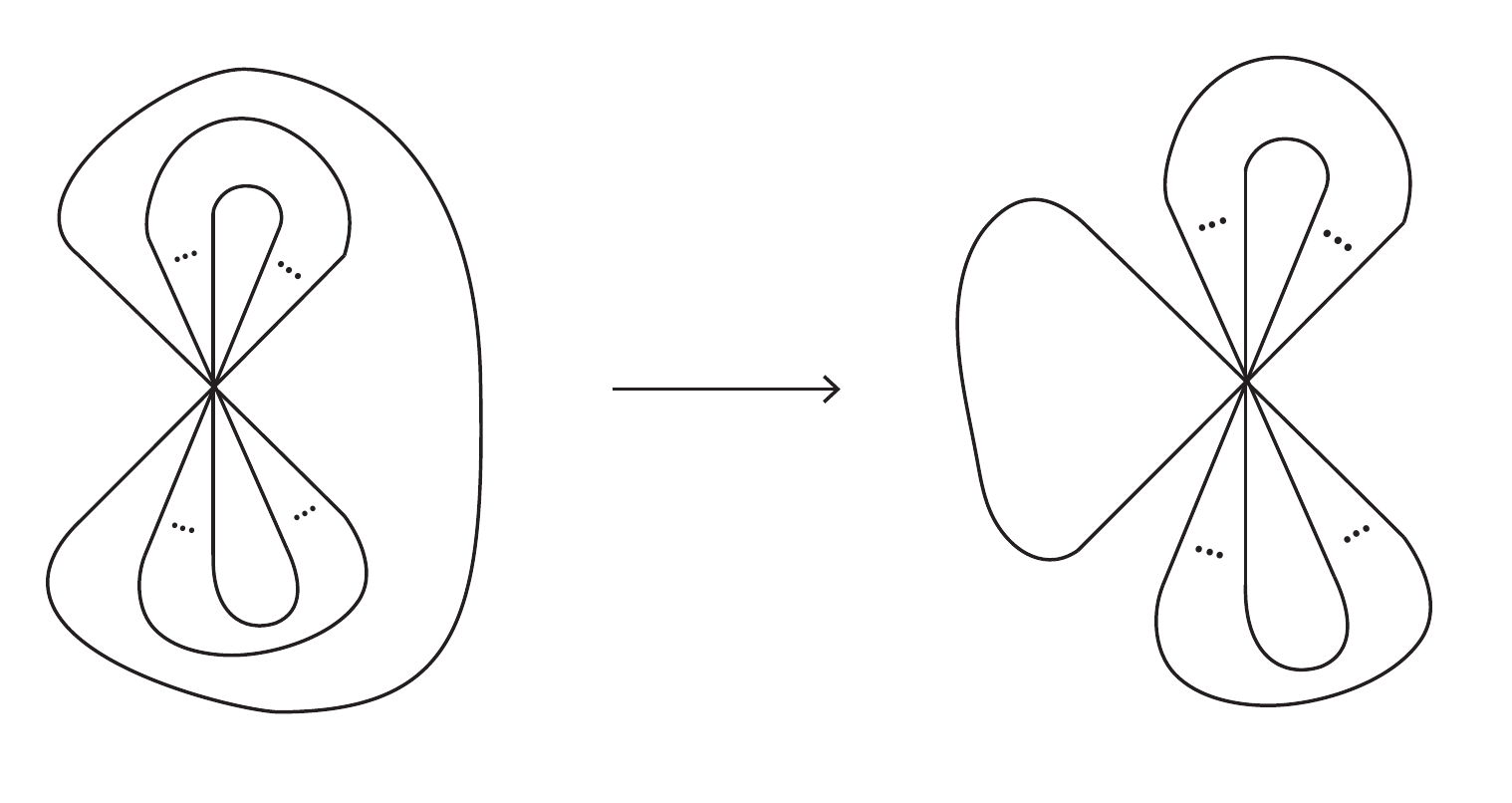}
\caption{Isotoping a monogon to border a finite area region.}
\label{fig:lemmaisotopy}
\qquad
\end{figure}

In the planar projection one monogon may correspond to the outside face. If this occurs, we isotope the monogon as shown in Figure \ref{fig:lemmaisotopy} to ensure that it encloses a finite region of the plane.
\end{proof}

\begin{remark}
Given a minimal \"ubercrossing projection of $L$, we may perturb the strands at the \"ubercrossing to create a regular projection of $L$ with $\binom{\ub(L)}{2}$ double crossings, so we have $c(K)\leq \frac{\ub(L)(\ub(L)-1)}{2}$. We can improve this bound by noting that any innermost loops may be untwisted before the perturbation, eliminating one crossing per monogon from the perturbed projection. Hence by Lemma \ref{monogonlemma},  when $K$ is a knot, $c(K)\leq \frac{\ub(L)(\ub(L)-1)}{2}-3$. We note that the trefoil knot realizes this upper bound on crossing number.
\end{remark}


\begin{defn}
The \emph{nesting number} of $L$, $n(L)$, is the least number of nesting loops in any \"ubercrossing projection of $L$ that realizes $\ub(L)$.
\end{defn}

 An \emph{arc presentation} of a knot is a conformation of the knot lying in a set of half planes hinging on a central axis, each containing a single simple arc of the knot. The arc index of a knot $K$, denoted $\alpha(K)$, is the least number of half-planes necessary. For a more thorough discussion of arc index, we refer the reader to \cite{Crom2}.

Note that if we have a petal projection, then by placing each loop in a half-plane hinging on an axis passing vertically through the central crossing, we generate an arc presentation. Hence, $\alpha(K) \le p(K)$. We call any such arc presentation a \emph{petal arc presentation}.  We would further like to consider the relation between $\ub(K)$ and $\alpha(K)$.

\begin{thm}\label{nestinglemma}
Let $L$ be a knot or link and let $\alpha(L)$ be the arc index of $L$. Then the following inequality holds:
\[
\alpha(L) \le \ddot{u}(L) + n(L).
\]
\end{thm}
  		
\begin{proof}
Given a knot or link $L$, we construct an arc presentation of $L$ in the following way. Consider a minimal \"{u}bercrossing projection of $L$ with $n(L)$ nesting loops. For each nesting loop, we isotope the loop towards the \"{u}bercrossing without producing new crossings as shown in Figure \ref{fig:nesttoarc}. 
			
\begin{figure}[h]
\centering
\includegraphics[height=25mm]{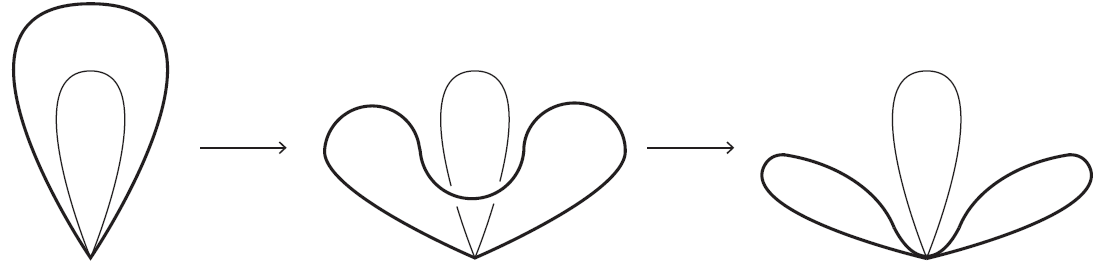}
\caption{Isotoping a nesting loop to turn an  \"{u}bercrossing conformation into an arc presentation.}
\label{fig:nesttoarc}
\end{figure}
			
After isotoping each loop, we have a projection of $L$ with one more loop than the previous projection. At the end of this process, we have a projection of $L$ that can be turned into an arc presentation with $\ddot{u}(L) + n(L)$ number of pages (and loops); each loop in its own page.
\end{proof}

\begin{cor}
Let L be a link. Then the following inequality holds:
\[
\ddot{u}(L) \ge \frac{1}{2}(\alpha(L) + 2).
\]
If $L$ is a knot, then $\ddot{u}(L) \ge \frac{1}{2}(\alpha(L) + 3)$.

\end{cor}

\begin{proof}
By the reasoning in the proof of Lemma \ref{monogonlemma} the number of non-nesting loops in an \"ubercrossing projection of $L$ is greater than or equal to 2. Therefore,
\[
n(L) \le \ddot{u}(L) - 2
\]
and by Theorem \ref{nestinglemma}, the corollary follows. Furthermore, if $L$ is  a knot then $n(L) \le \ddot{u}(L) - 3$ by Lemma \ref{monogonlemma}.
\end{proof}
	
\begin{cor}
For a non-split alternating link L, the following inequality holds: 
\[
\ddot{u}(L) \ge c(L) + 2 - n(L).
\]
\end{cor}
	
\begin{proof}
By Corollary 8 of \cite{BP} for a non-split alternating link $L$, 
\[
\alpha(L) = c(L) + 2.
\]
The corollary now follows from Theorem \ref{nestinglemma}.
\end{proof}

\begin{cor}\label{braidthm}

If a knot $K$ has \"ubercrossing number $\ddot{u}(K)$ and braid index $\beta(K)$, then 
\[
\beta(K)\leq\ddot{u}(K)-2.
\]
\end{cor}

\begin{proof}
In \cite{Crom}, Cromwell proves that $\beta(K)\leq\frac{\alpha(K)}{2}$. By Theorem \ref{nestinglemma}, $\alpha(K)\leq\ddot{u}(K)+n(K)$. Lemma \ref{monogonlemma} implies that $n(K)\leq\ddot{u}(K)-3$. Thus, $\beta(K)\leq\ddot{u}(K)-\frac{3}{2}.$ Since $\beta(K)$ is an integer, we have 
$\beta(K)\leq\ddot{u}(K)-2$.
\end{proof}

\begin{remark} We mention in passing that an arc presentation with $\alpha$ pages can be represented by a \emph{grid diagram}, which presents the knot as a set of horizontal and vertical line segments connecting a set of points on the intersections of a square grid, such that every row and column of the grid contains exactly one segment and vertical segments cross over horizontal segments. The length of a given horizontal or vertical segment is the number of rows or columns, respectively, that it spans. Grid diagrams have become very important in the calculation of knot homologies. Petal projections yield very specific grid diagrams.

Given a petal arc presentation $P$ with $p$ pages, one can show that its associated grid diagram has the following properties:

\begin{enumerate}
\item There is exactly one vertical stick whose adjacent horizontal sticks point in opposite directions -- one points to the left of the vertical stick, and the other points to the right. We call this stick the \emph{inflection stick}, and denote it $I$. The horizontal sticks adjacent to $I$ have length $\frac{p-1}{2}$. 
\item Each remaining vertical stick's adjacent horizontal sticks have length $\frac{p+1}{2}$ and $\frac{p-1}{2}$. 
\end{enumerate} 

\end{remark}


\section{Bounds}

\subsection{2-Bridge and Pretzel Knots}

We now obtain upper bounds on $\ub(K)$ for several classes of knots, and determine the petal numbers of twist knots and certain other rational knots, including 2-braids. To this end, we develop moves which can be applied locally in diagrams to consolidate crossings. As a useful notation we introduce the notion of an \emph{\"ubertangle}, an example of which is shown in Figure \ref{fig:ubertangle}.

\begin{figure}[here]
\centering
\includegraphics[height=25mm]{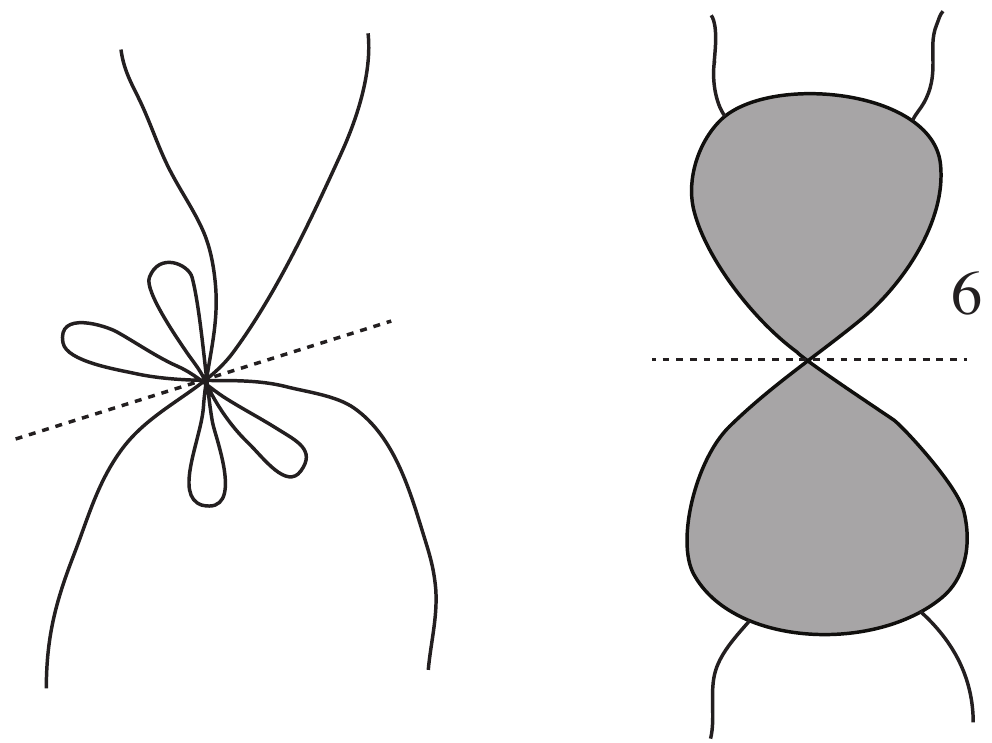}
\caption{The 6-crossing structure on the left can be represented by the Type 1 \"ubertangle on the right.}
\label{fig:ubertangle}
\end{figure}
\qquad

The center of an \"ubertangle represents a crossing; the dark regions, or \emph{lobes} hereafter, contain the same number of strands so that a strand can be passed over or under area indicated by the dotted line in Figure \ref{fig:ubertangle}, splitting the lobes from one another, and it will remain a legal crossing according to our definition. A number next to the center of an \"ubertangle indicates the multiplicity of the crossing. Initially we concern ourselves with two-lobe \"ubertangles of the three types shown in Figure \ref{fig:ubertangles}. Later it will be convenient to work with \"ubertangles of more than two lobes; this notation indicates that any straight path across the crossing that separates half the lobes from the other half also bisects the crossing. As a last convenient notation, we define a \emph{petaltangle}, which is an \"ubertangle with no nesting loops in its lobes. The lobes are drawn as triangles (see Figure \ref{fig:petaltangle}). A strand leaving a triangular lobe in a petaltangle lies on the the indicated side of the loops in the lobe.

\begin{figure}[here]
\centering
	\begin{subfigure}[b]{.2\textwidth}
		\centering
		\includegraphics[height=27mm]{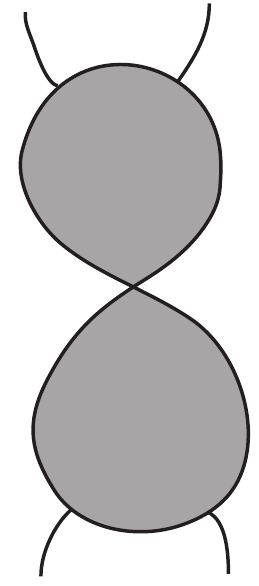}
		\label{fig:type_1_ubertangle}\caption{Type 1.}
		\qquad
	\end{subfigure}
	\begin{subfigure}[b]{.2\textwidth}
		\centering
		\includegraphics[height=25mm]{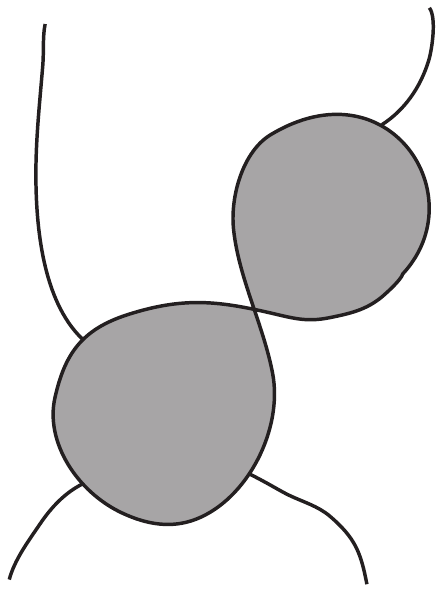}
		\label{fig:type_2_ubertangle}\caption{Type 2.}
		\qquad
	\end{subfigure}
	\begin{subfigure}[b]{.2\textwidth}
		\centering
		\includegraphics[height=15mm]{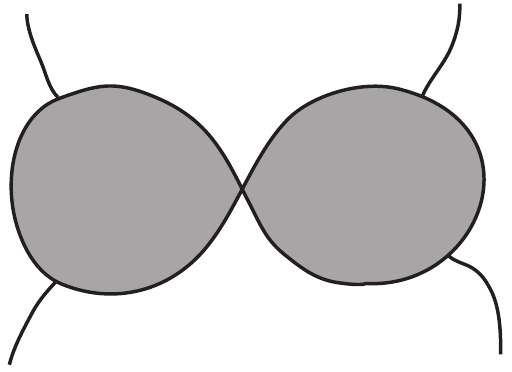}
		\label{fig:type_3_ubertangle}\caption{Type 3.}
		\qquad
	\end{subfigure}
	\caption{Types of \"ubertangles.}\label{fig:ubertangles}
\end{figure}

\begin{figure}[!h]
\centering
\includegraphics[height=30mm]{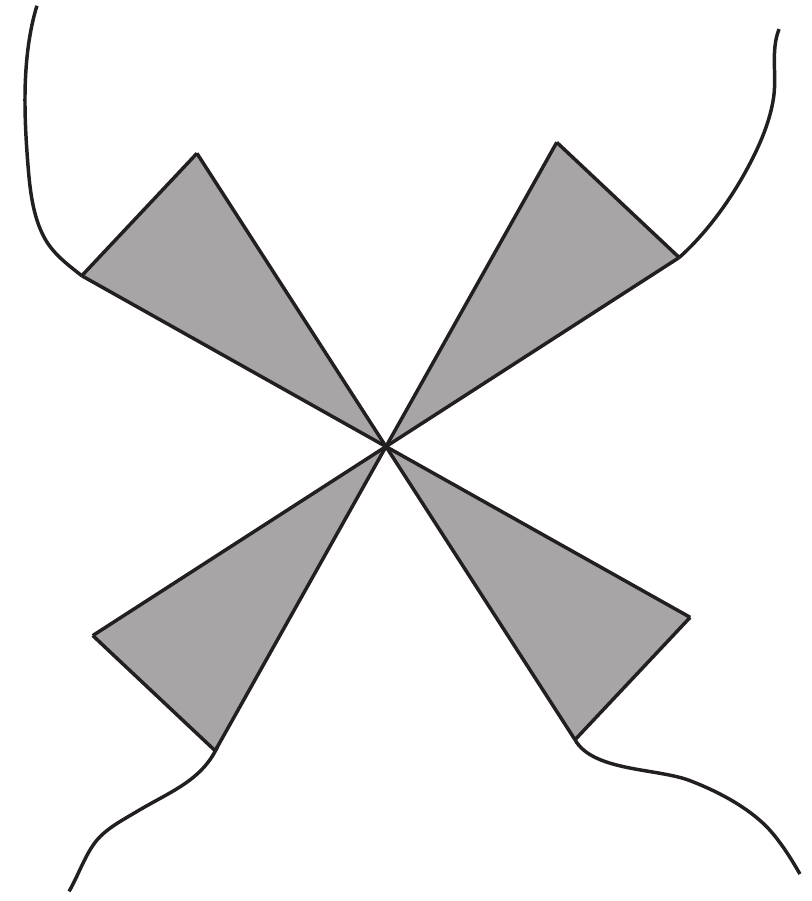}
\caption{Petaltangle.}
\label{fig:petaltangle}
\end{figure}

\begin{lemma}\label{move1}
(Moves 1a, 1b) A sequence of twists above a Type 1 \"ubertangle may be simplified as in Figure \ref{fig:move1}.
\end{lemma}

\begin{figure}[here]
\centering
\begin{subfigure}[b]{.2\textwidth}
\includegraphics[height=30mm]{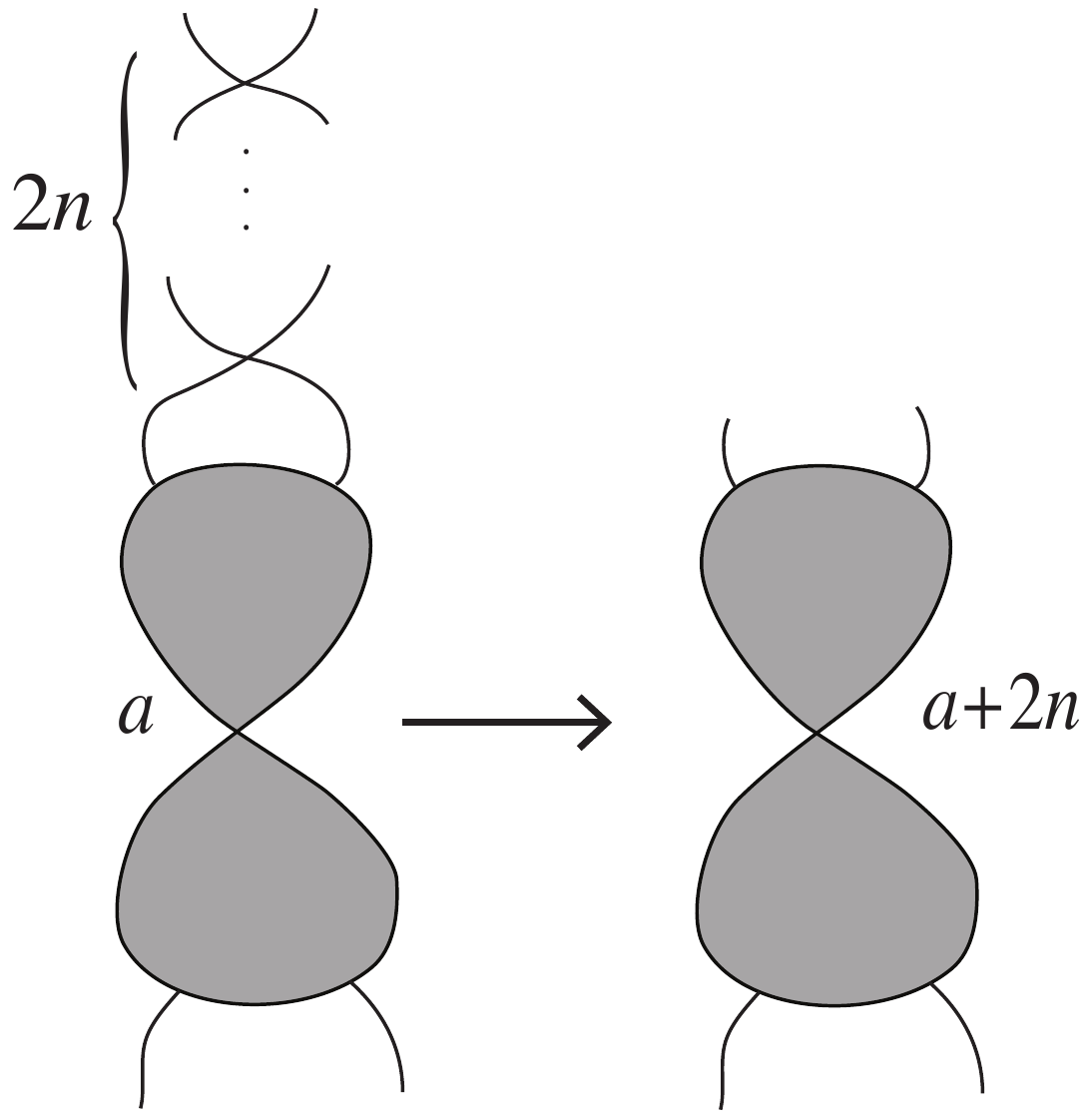}
\caption{Move 1a.}
\label{fig:move_1a}
\qquad
\end{subfigure}
\begin{subfigure}[b]{.2\textwidth}
\includegraphics[height=30mm]{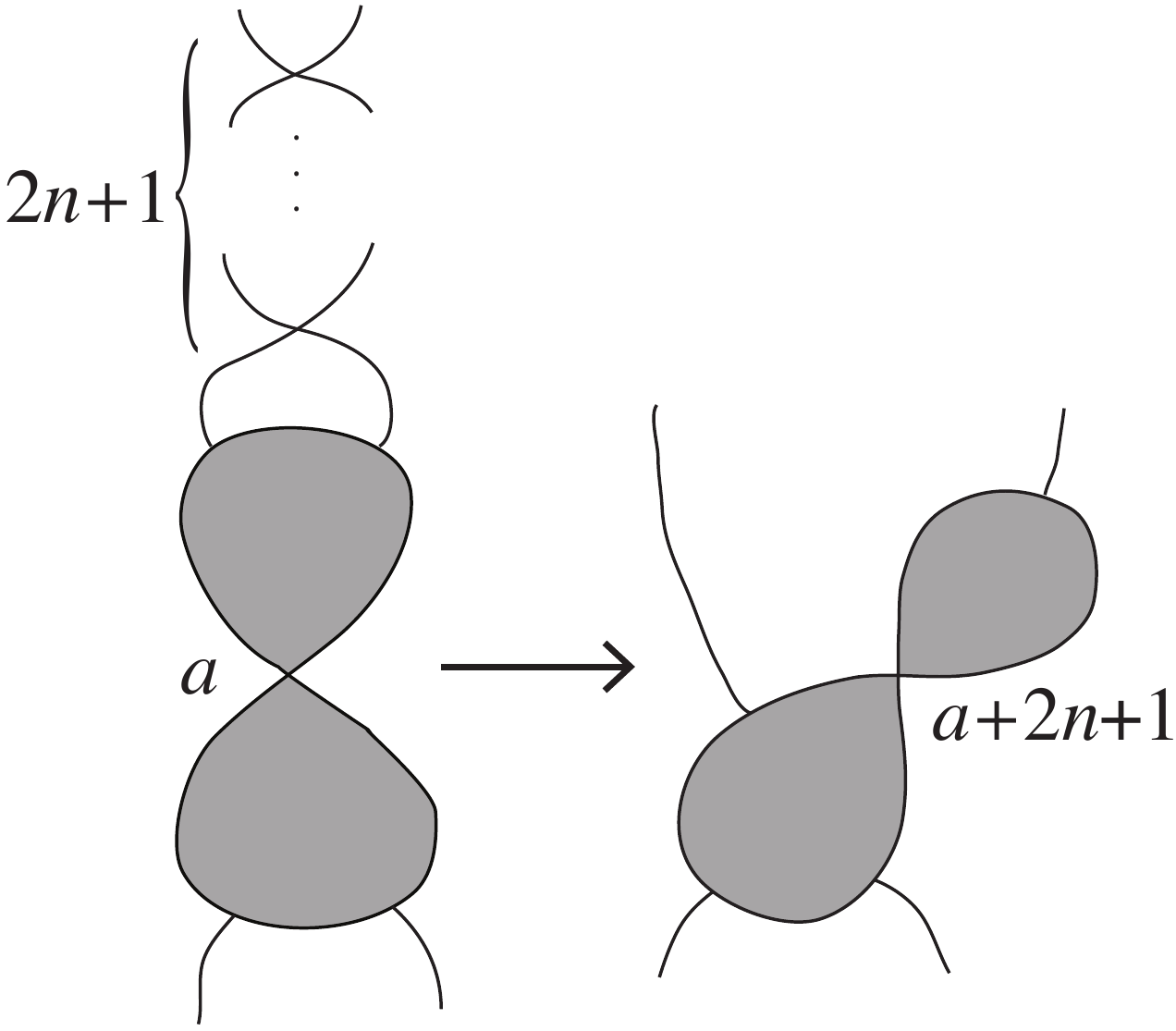}
\caption{Move 1b.}
\label{fig:move_1b}
\qquad
\end{subfigure}
\caption{Moves on \"ubertangles.}
\label{fig:move1}
\end{figure}
\qquad

\begin{proof}
We first show Move 1a is valid by induction on the number of twists above the Type 1 \"ubertangle. For 0 twists, the claim holds. Assuming the claim holds for $2n$ twists above the \"ubertangle, consider a Type 1 \"ubertangle with $2n+2$ twists above it. Apply the inductive hypothesis to the lowermost $2n$ crossings, and then eliminate the remaining 2 crossings as shown in Figure \ref{fig:move_1a_prelim}.

\begin{figure}[here]
\centering
\includegraphics[height=25mm]{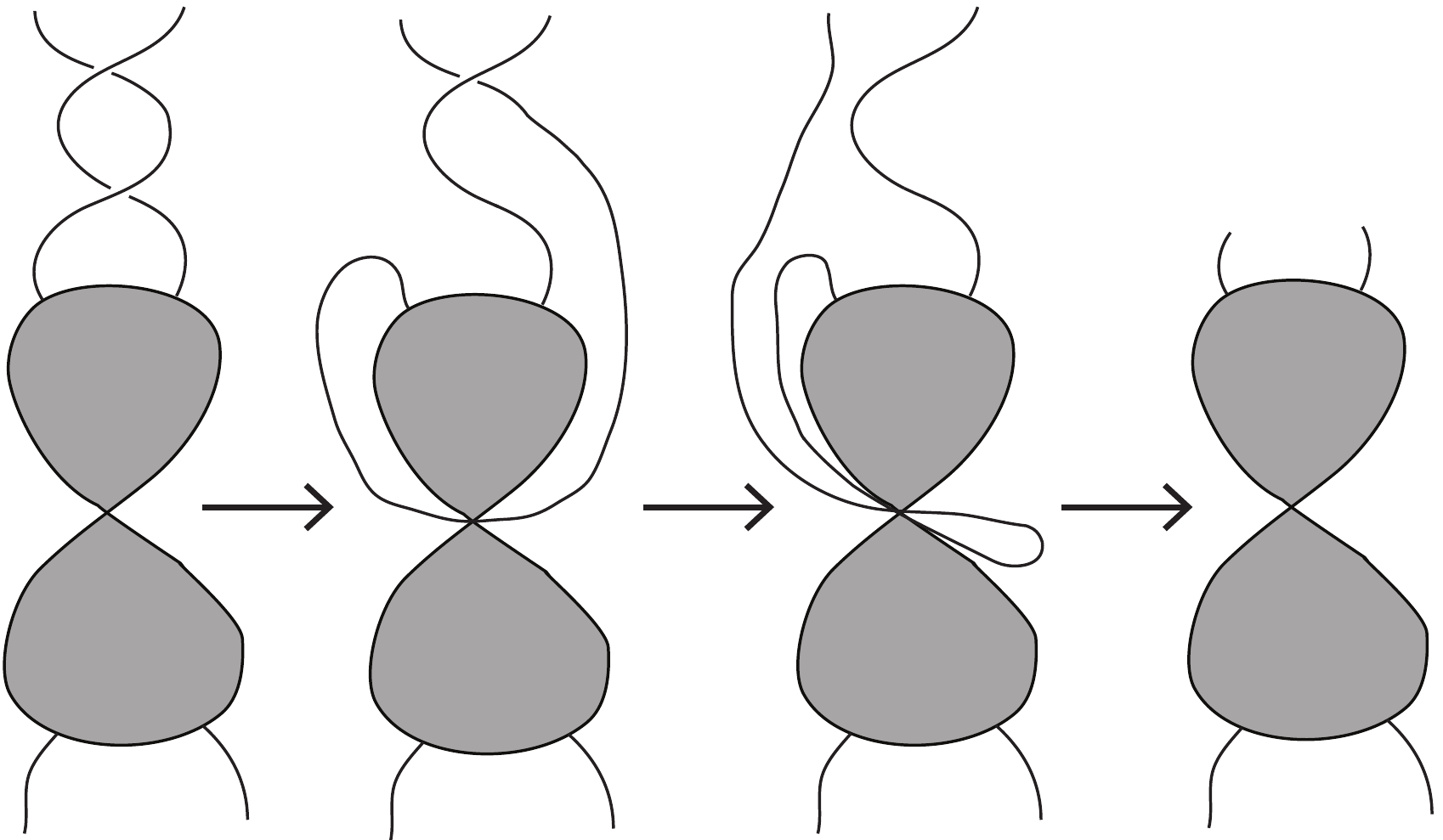}
\caption{Justification for the moves on \"ubertangles.}
\label{fig:move_1a_prelim}
\end{figure}
\qquad

For move 1b, take a Type 1 \"ubertangle with $2n+1$ twists above it. Reduce the first $2n$ using Move 1a, and then fold the last crossing over the \"ubertangle, changing it from Type 1 to Type 2.
\end{proof}


\begin{lemma}
(Moves 1a$'$, 1a$''$) A sequence of twists with an even number of crossings above a petaltangle may be reduced as shown in Figure \ref{fig:modified_1a}.
\end{lemma}

\begin{figure}[here]
\centering
	\begin{subfigure}[b]{.3\textwidth}
		\centering
		\includegraphics[height=30mm]{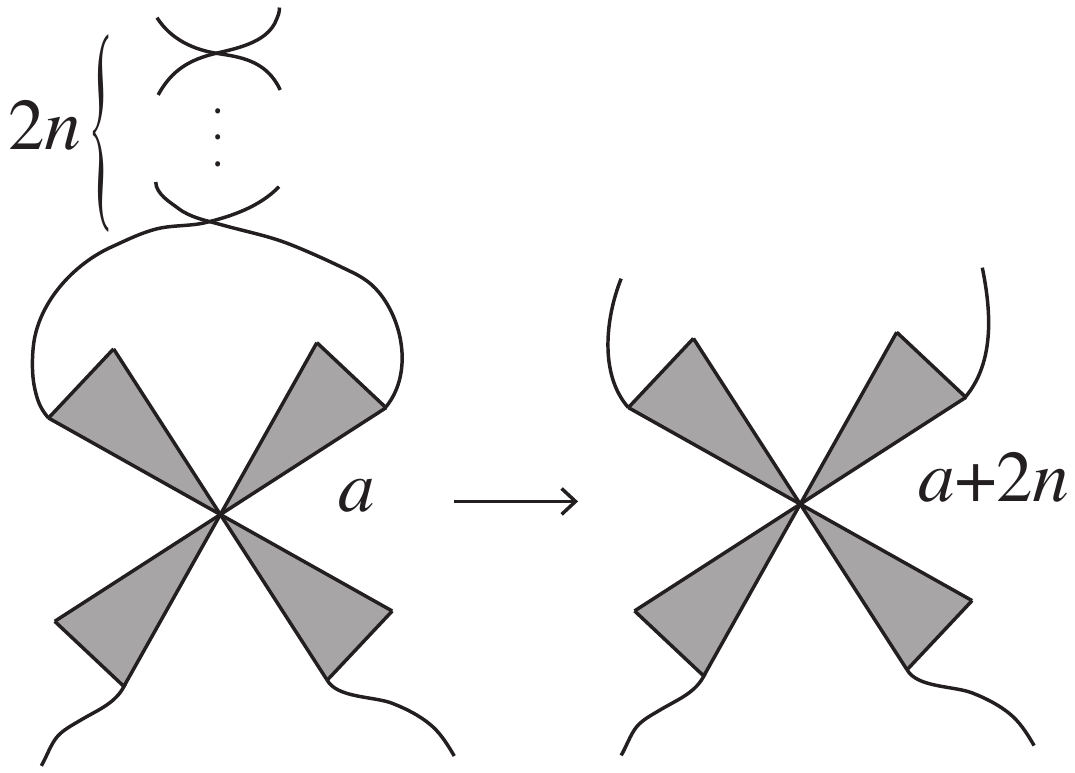}
		\caption{}\label{1aprime}
		\qquad
	\end{subfigure}
	\begin{subfigure}[b]{.3\textwidth}
		\centering
		\includegraphics[height=30mm]{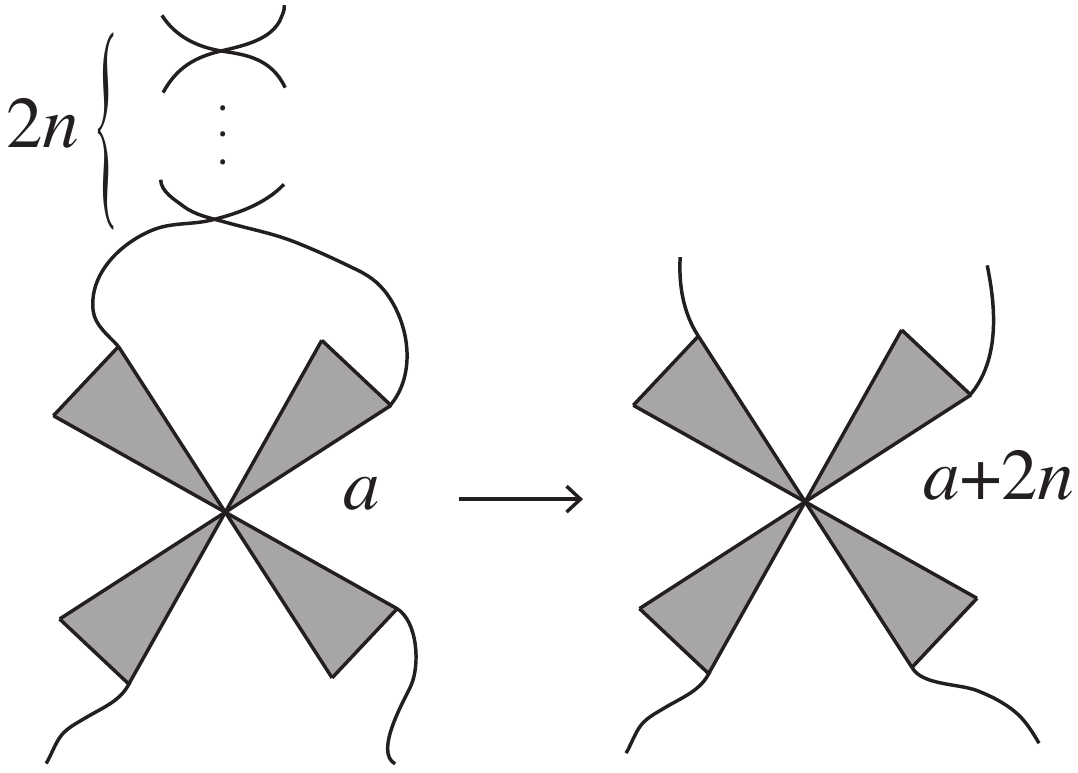}
		\caption{}\label{1aprimeprime}
		\qquad
	\end{subfigure}
\caption{Moves on petaltangles.}	
\label{fig:modified_1a}
\end{figure}
\qquad

\begin{proof}
The proof mirrors that of Lemma \ref{fig:move1}.
\end{proof}


\begin{lemma}
(Moves 2a, 2b) A sequence of twists above a Type 2 \"ubertangle may be simplified as shown in Figure \ref{fig:move_2}.

\begin{figure}[here]
\centering
\begin{subfigure}[b]{.3\textwidth}
\centering
\includegraphics[height=25mm]{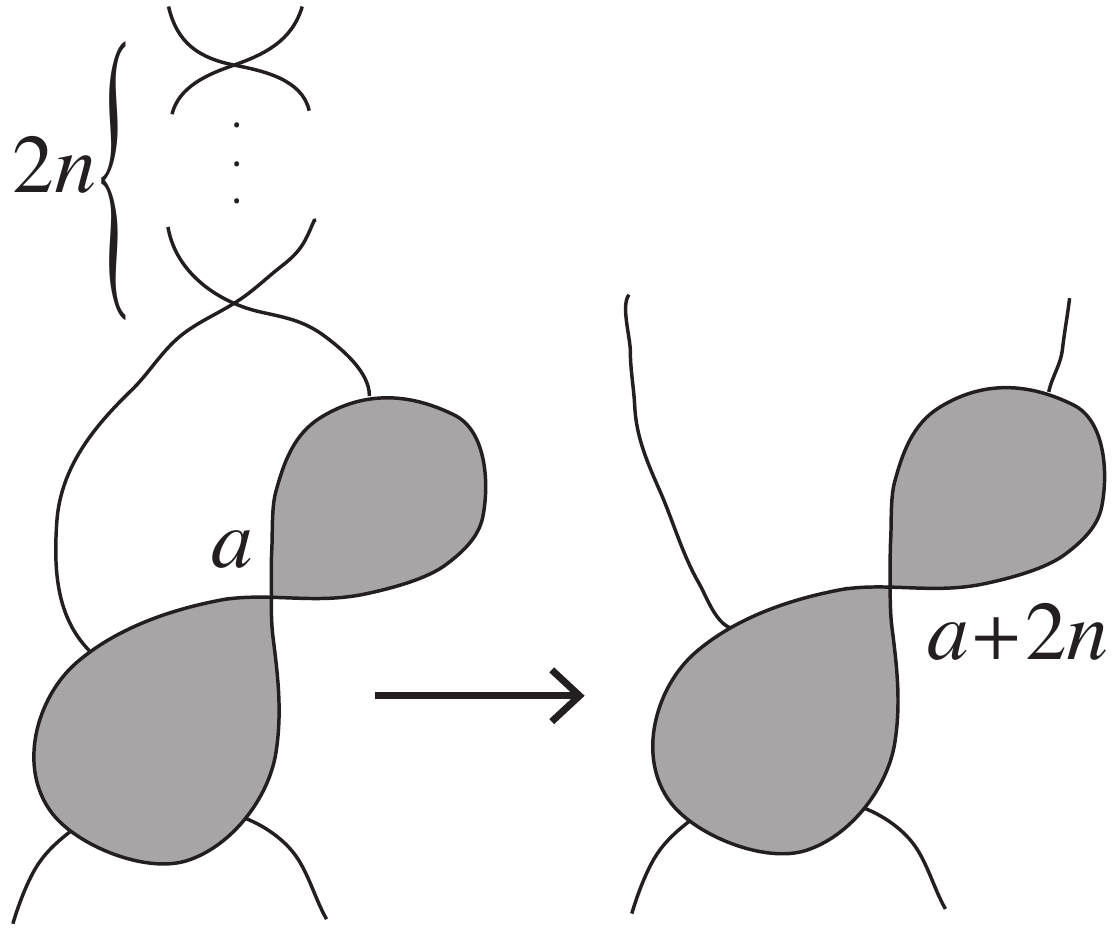}
\caption{Move 2a.}
\label{fig:move_2a}
\qquad
\end{subfigure}
\begin{subfigure}[b]{.3\textwidth}
\centering
\includegraphics[height=25mm]{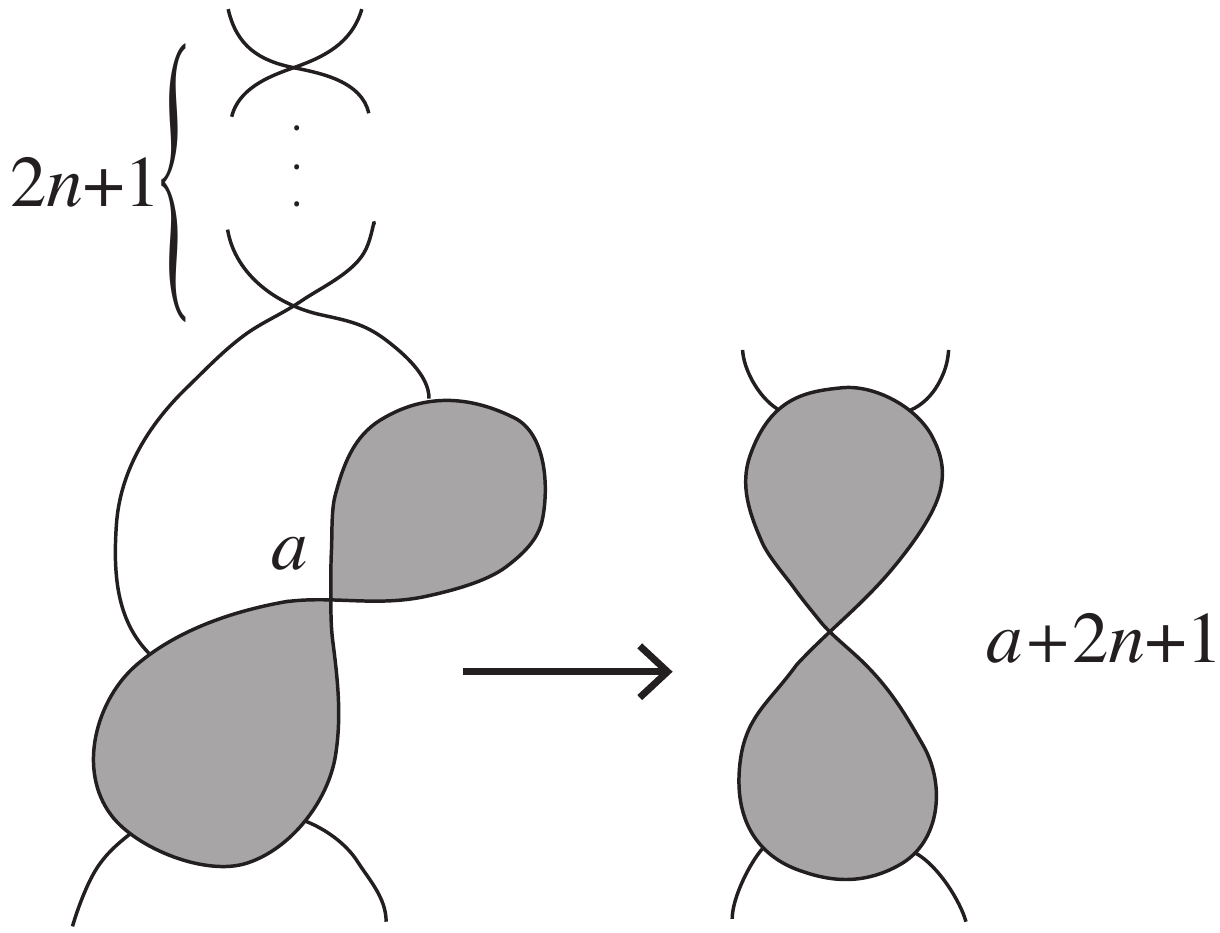}
\caption{Move 2b.}
\label{fig:move_2b}
\qquad
\end{subfigure}
\caption{Additional moves on \"ubertangles.}
\label{fig:move_2}
\quad
\end{figure}
\end{lemma}

\begin{proof}
In the string of twists above the Type 2 \"ubertangle, stretch the bottommost top strand over the \"ubertangle, changing it from Type 2 to Type 1. Then apply Move 1a or 1b depending on the parity of the number of twists.
\end{proof}


\begin{lemma}
(Move 3) A sequence of twists above a Type 3 \"ubertangle may be simplified as in Figure \ref{move_3}.

\begin{figure}[here]
\centering
\includegraphics[height=25mm]{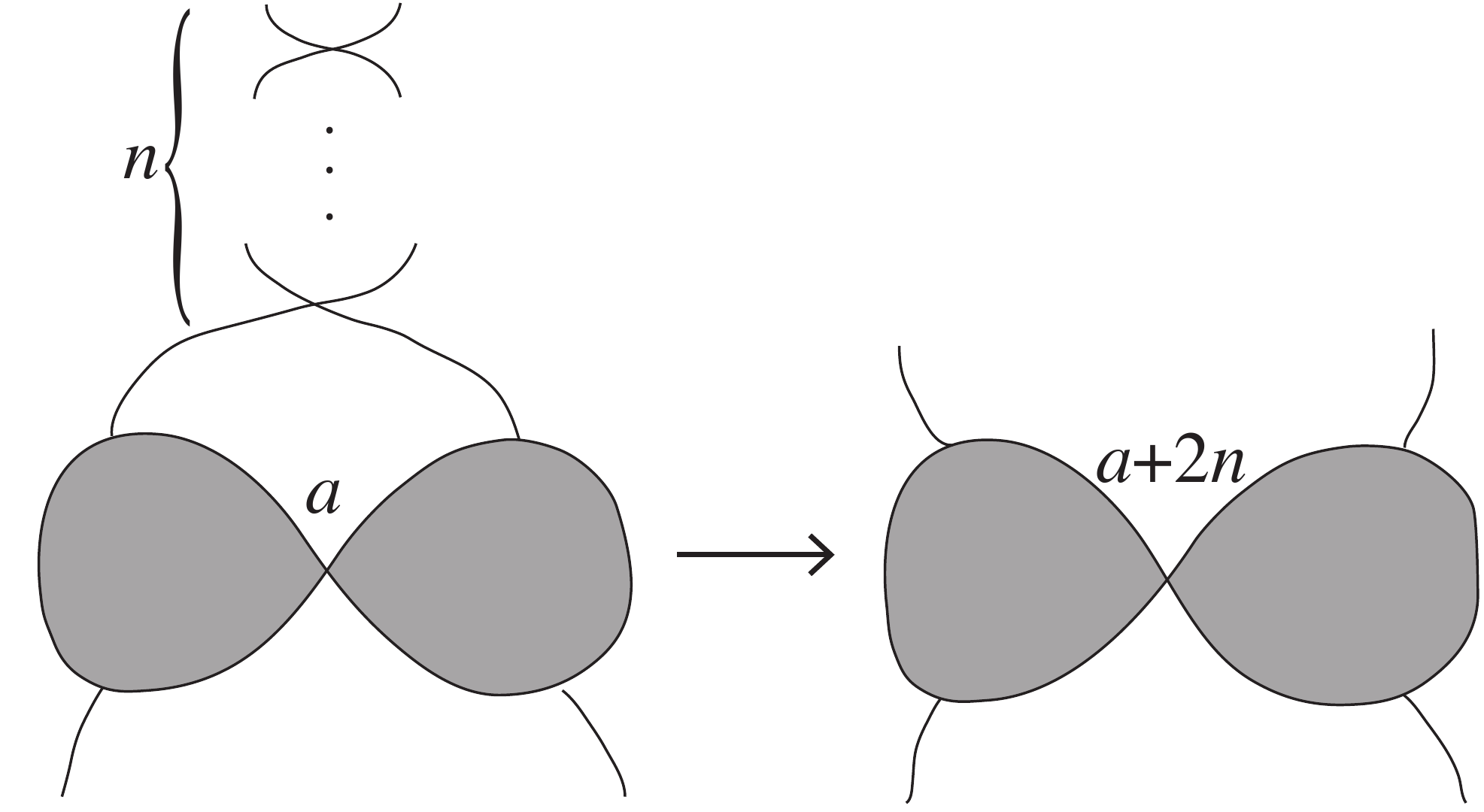}
\caption{Move 3.}
\label{move_3}
\end{figure}
\end{lemma}

\begin{proof}
We see in Figure \ref{fig:move_3_prelim} that a crossing above a Type 3 \"ubertangle may be absorbed at the price of adding two to the multiplicity of the crossing in the center of the \"ubertangle. We simply iterate this operation to obtain Move 3.
\begin{figure}[here]
\centering
\includegraphics[height=15mm]{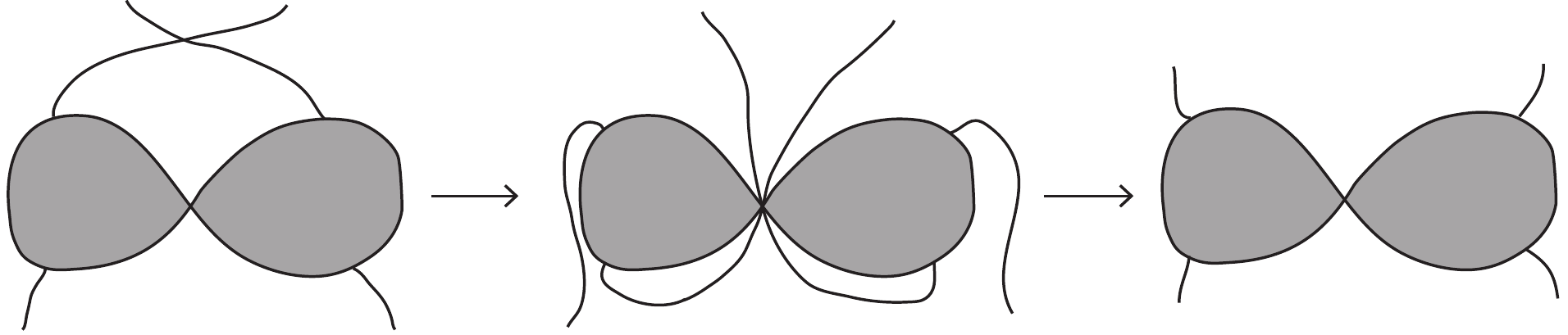}
\caption{Justification for Move 3.}
\label{fig:move_3_prelim}
\end{figure}
\end{proof}


\begin{thm}
If $K$ is a 2-braid knot or link, $\ub(K)\leq c(K)+1$.
\end{thm}

\begin{proof}
Take a closed braid representation of $K$ with a minimal number of crossings. The result follows directly from the application of Move 1a.
\end{proof}


\begin{thm}\label{rationalknotthm}
If $K$ is a 2-bridge knot, $\ub(K) \le 2c(K)-1$. \end{thm}

\begin{proof}
Place $K$ in an alternating projection of the form shown in Figure \ref{fig:rational_knot_ubertangle}, which must realize $c(K)$.

\begin{figure}[here]
\centering
\includegraphics[height=40mm]{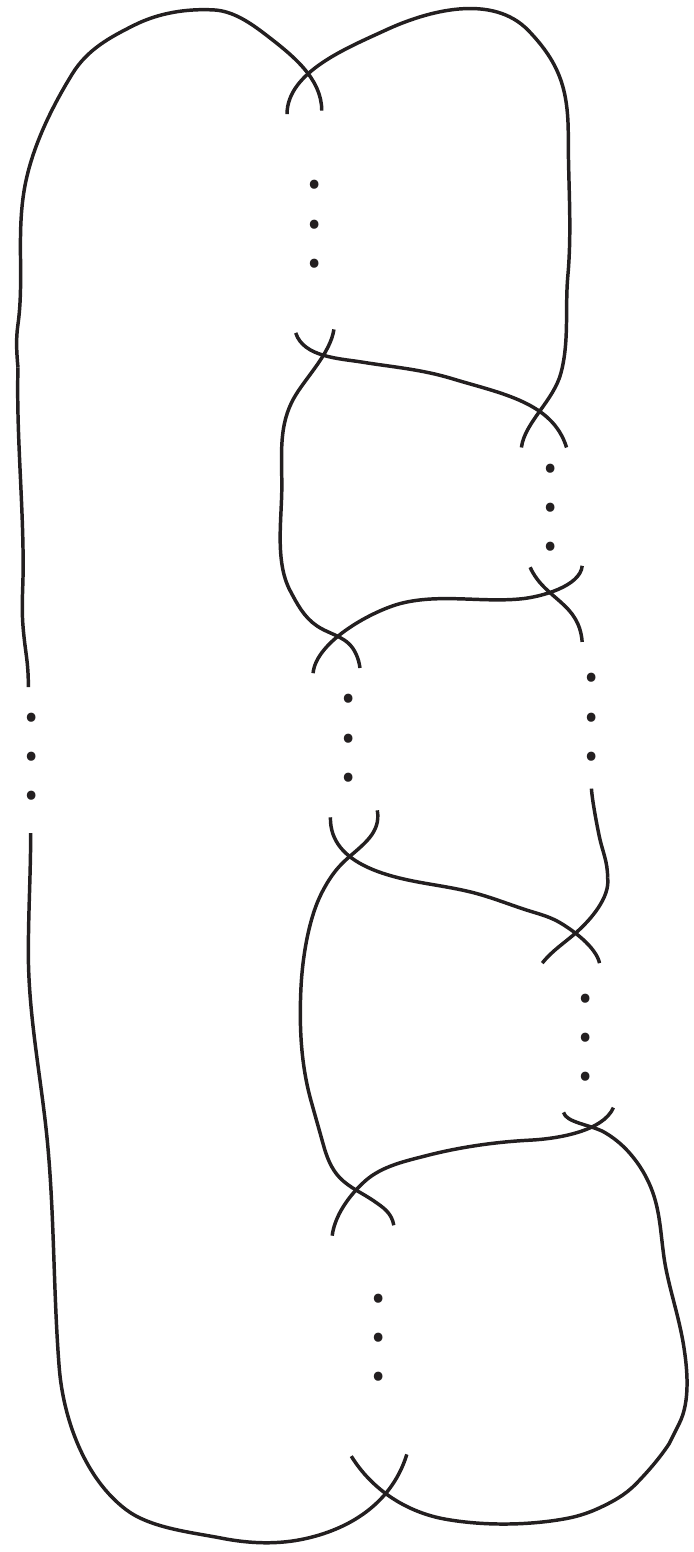}
\caption{An alternating projection of a 2-bridge knot.}
\label{fig:rational_knot_ubertangle}
\qquad
\end{figure}

Beginning at the bottom of the projection, we apply Moves 1-3 to create an \"ubertangle which we move incrementally up the projection, eliminating crossings. Moves 1 and 2 add one to the multiplicity of the \"ubertangle for each crossing, while Move 3 adds two. If parity considerations force us to use Move 3 at each stage then we have obtained an \"ubercrossing projection with multiplicity $2c(K)$. However, we can always use Move 1 on the bottommost sequence of twists unless it only contains one crossing, in which case we apply Move 1 to the bottommost twist as well as the bottommost right sequence of twists, giving the result.
\end{proof}


\begin{lemma}\label{arcindexlowerbound}
If $K$ is an alternating knot, then 
\[
p(K)\geq\begin{cases}
c(K)+2&\text{if $c(K)$ is odd}\\
c(K)+3&\text{if $c(K)$ is even.}
\end{cases}
\]
\end{lemma}

\begin{proof}
By \cite{BP}, we have $\alpha(K)=c(K)+2$ since $K$ is alternating. Moreover, a petal projection of $K$ with $n$ loops gives rise to an arc presentation of $K$ with $n$ arcs, so $p(K)\geq\alpha(K)$. Hence $p(K)\geq c(K)+2$. If $c(K)$ is even, then we have $p(K)\geq c(K)+3$ because $p(K)$ must be odd.
\end{proof}

\begin{thm}
If $B$ is a 2-bridge knot with Conway notation $a_1a_2\cdots a_n$, with $a_i$ odd if and only if $i=1$, then $p(B)=c(B)+2$. If $T$ is a twist knot, then 
\[
p(T)=\begin{cases}
c(T)+2&\text{if $c(T)$ is odd}\\
c(T)+3&\text{if $c(T)$ is even.}
\end{cases}
\] 
\end{thm}

\begin{proof}
Put $B$ in its alternating bridge form as shown in Figure \ref{fig:rational_knot_ubertangle} such that the bottommost sequence of twists has an odd number of crossings. We apply Move 1a$'$ to the bottommost sequence of twists, considering the bottommost twist to be a petaltangle of the form in Figure \ref{fig:petaltangle}. Now we can apply Move 1a$''$ to the next sequence of twists, and repeat the process. This ultimately yields a pre-petal projection with $c(B)+1$ loops, giving a petal projection with $c(B)+2$ loops, so $p(B)\leq c(B)+2$. The result now follows from Lemma \ref{arcindexlowerbound}.

Now we consider $T$. Since $T$ is a twist knot, it can be put in the form shown in Figure \ref{fig:twist}(a). Note that we can choose to move from \ref{fig:twist}(a) to either \ref{fig:twist}(b) or \ref{fig:twist}(c). As a last preparation, we leave to the reader the inductive proof that the move shown in \ref{fig:twist}(d) is valid.

\begin{figure}[htbp]
\centering
\begin{subfigure}[t]{.2\textwidth}
\centering
\includegraphics[height=28mm]{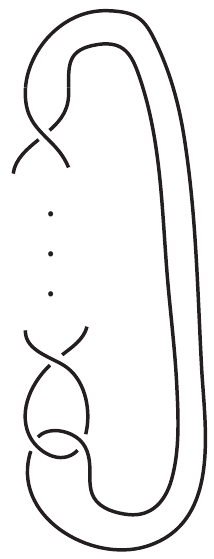}
\label{fig:twist1}
\caption{}
\qquad
\end{subfigure}
\begin{subfigure}[t]{.2\textwidth}
\centering
\includegraphics[height=28mm]{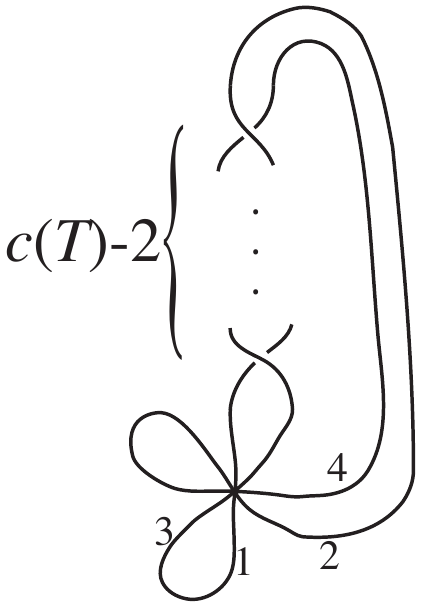}
\label{fig:twist2}
\caption{}
\qquad
\end{subfigure}
\begin{subfigure}[t]{.2\textwidth}
\centering
\includegraphics[height=28mm]{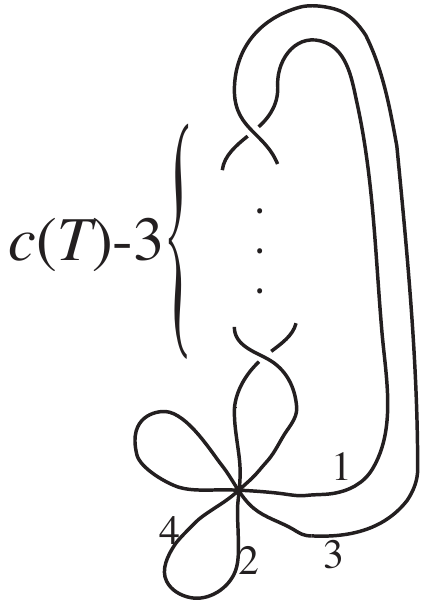}
\label{fig:twist3}
\caption{}
\qquad
\end{subfigure}
\begin{subfigure}[t]{.2\textwidth}
\centering
\includegraphics[height=28mm]{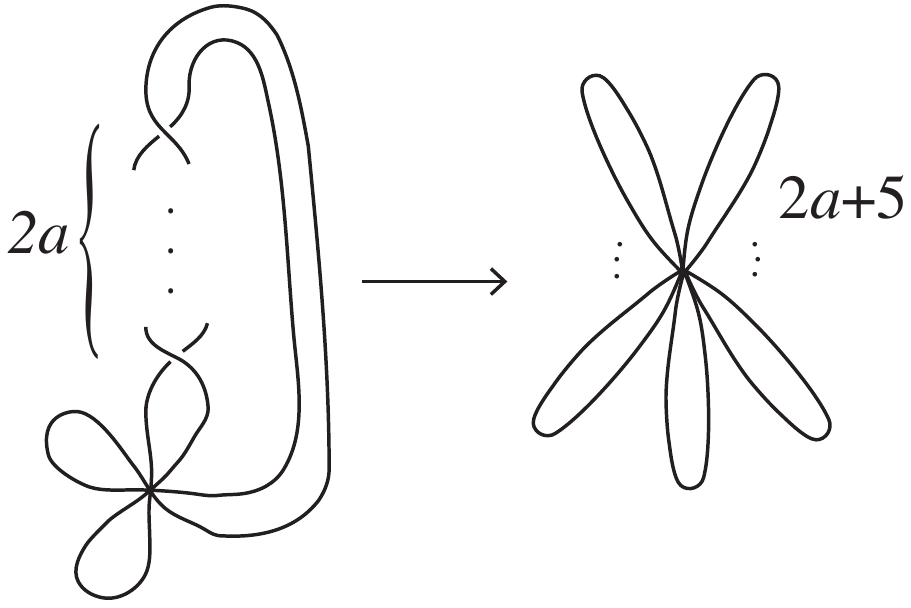}
\label{fig:twist4}
\caption{}
\end{subfigure}
\caption{Obtaining petal projections of twist knots.}
\label{fig:twist}
\qquad
\end{figure}

Now suppose $c(T)$ is odd. Move to \ref{fig:twist}(c) and then apply \ref{fig:twist}(d), yielding a petal projection of $T$ with $c(T)-3+5=c(T)+2$ loops. If $c(T)$ is even, move to \ref{fig:twist}(b) and then apply \ref{fig:twist}(d), yielding a petal projection of $T$ with $c(T)-2+5=c(T)+3$ loops. Hence the lower bounds on $p(T)$ from Lemma \ref{arcindexlowerbound} are also upper bounds, proving the second claim.
\end{proof}


\begin{lemma}\label{Move4}
(Move 4) Two \"ubertangles which are attached as shown in Figure \ref{fig:move_4_actual} may be simplified as shown.
\end{lemma}

\begin{figure}[!htbp]
\centering
\includegraphics[height=12mm]{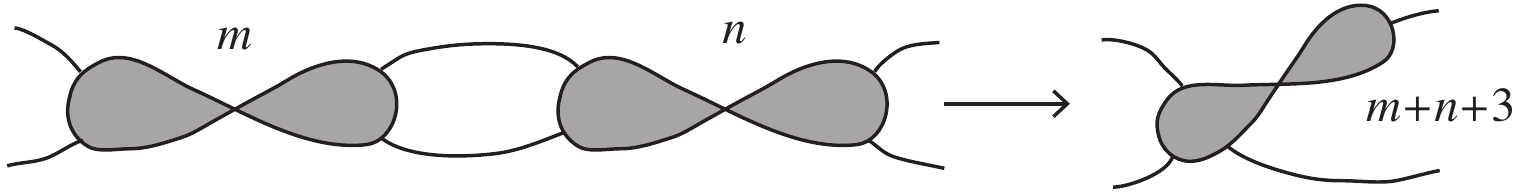}
\caption{Move 4.}
\label{fig:move_4_actual}
\end{figure}

\begin{proof}
Figure \ref{fig:move_4} illustrates a sequence of moves which proves the claim.
\begin{figure}[h]
\centering
\includegraphics[height=45mm]{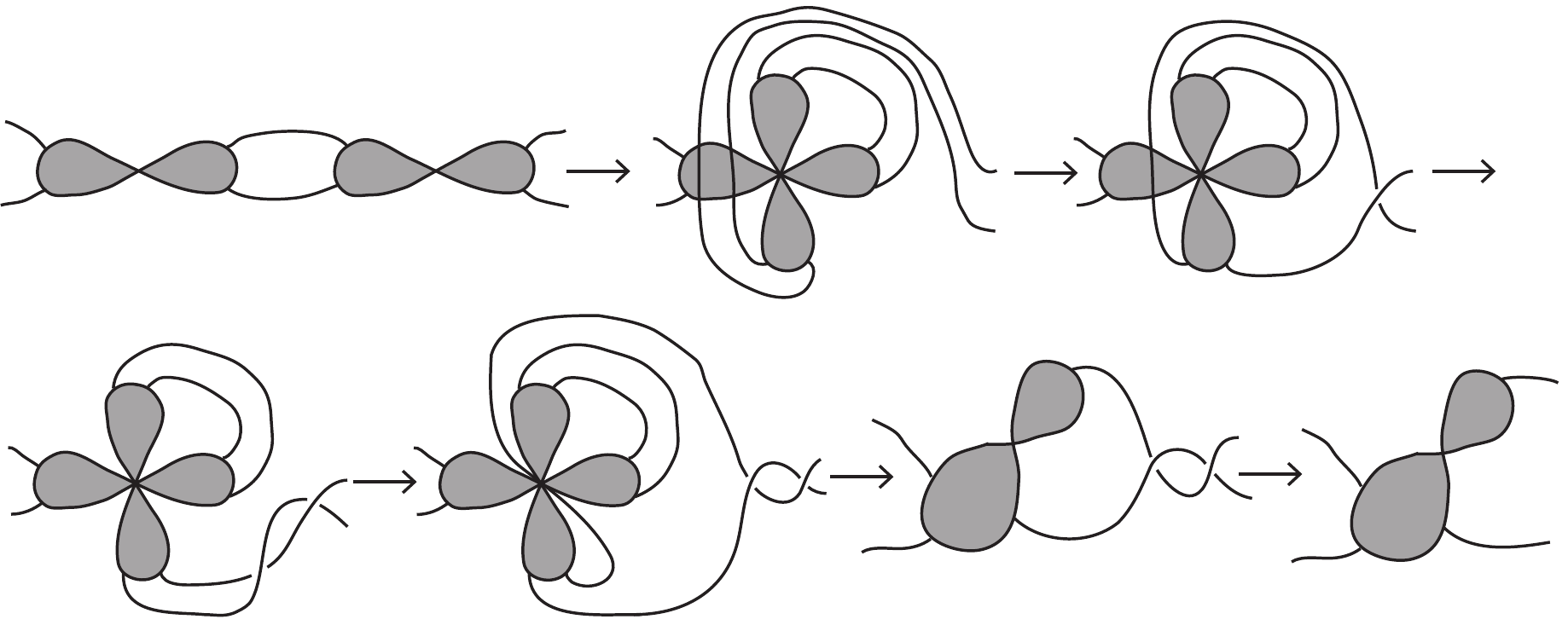}
\caption{Proving Lemma \ref{Move4}.}
\label{fig:move_4}
\end{figure}
\end{proof}


\begin{lemma} \label{Move5}
(Move 5) Two \"ubertangles which are attached as shown in Figure \ref{fig:move_5} may be simplified as shown.
\begin{figure}[here]
\centering
\includegraphics[height=15mm]{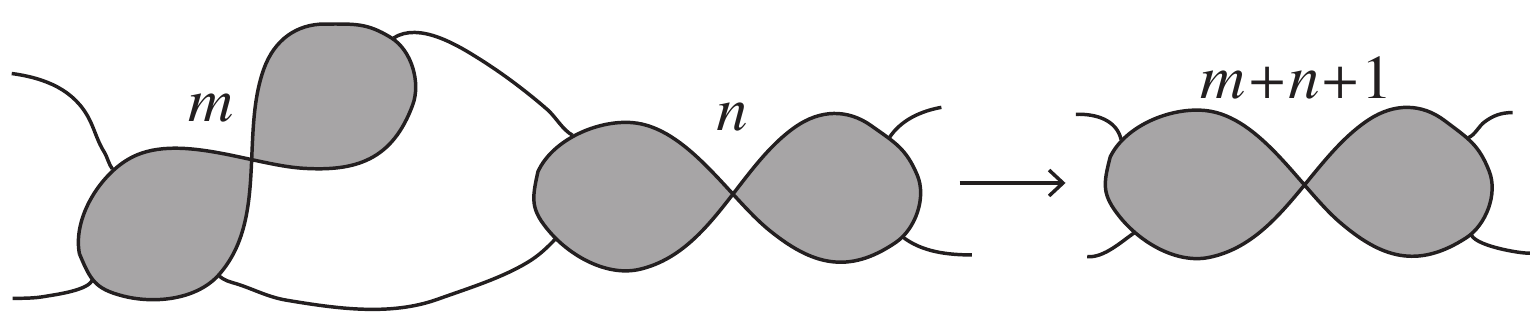}
\caption{Move 5.}
\label{fig:move_5}
\end{figure}
\end{lemma}

\begin{proof}
Figure \ref{fig:move_5_proof} shows one appropriate sequence of moves.
\begin{figure}[here]
\centering
\includegraphics[height=18mm]{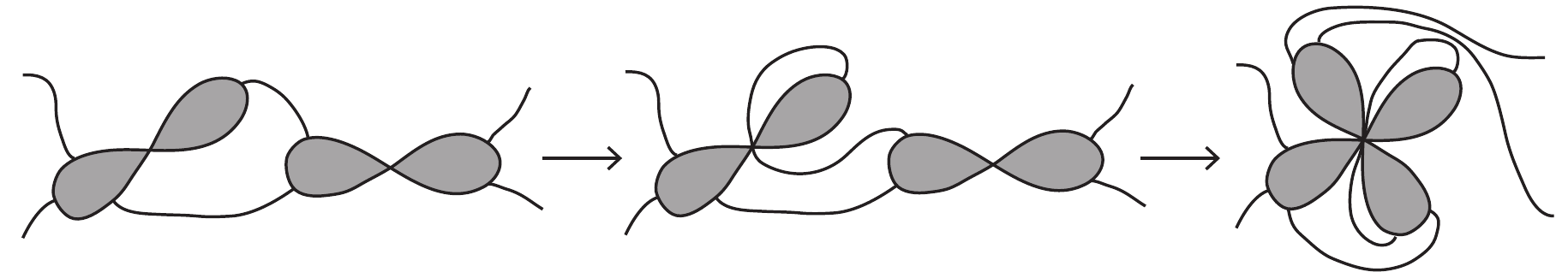}
\caption{Proving Lemma \ref{Move5}.}
\label{fig:move_5_proof}
\end{figure}
\end{proof}


\begin{remark}
The next theorem provides an upper bound on the \"ubercrossing number of pretzel knots. Pretzel links are formed from vertical columns of twists; we observe that at most one of the columns can contain an even number of crossings if a pretzel link forms a knot.
\end{remark}

\begin{thm}\label{pretzel}
Let $K$ be a pretzel knot whose Conway notation is a finite sequence of integers of the same sign separated by commas. If $K$ has an even number $\ell$ in its Conway sequence, then
\[
\ub(K)\leq
\begin{cases}
c(K)+\ell+3k-3&\text{if $k$ is odd}\\
c(K)+\ell+3k-2&\text{if $k$ is even},
\end{cases}
\]
where $k$ is length of the Conway sequence. If the sequence contains only odd numbers, then
\[
\ub(K)\leq
\begin{cases}
c(K)+3k-2&\text{if $k$ is odd}\\
c(K)+3k-1&\text{if $k$ is even},
\end{cases}
\]
\end{thm}

\begin{proof}
Suppose $K$'s Conway sequence contains an even number $\ell$. Take the canonical projection of $K$, shown in Figure \ref{fig:pretzel}(a), which is alternating and reduced and so must realize $c(K)$. We apply Move 1a$'$ to each column of twists with an odd number of crossings, and Move 3 to the column with $\ell$ crossings.
\begin{figure}[h]
\centering
\begin{subfigure}[b]{.3\textwidth}
\includegraphics[height=23mm]{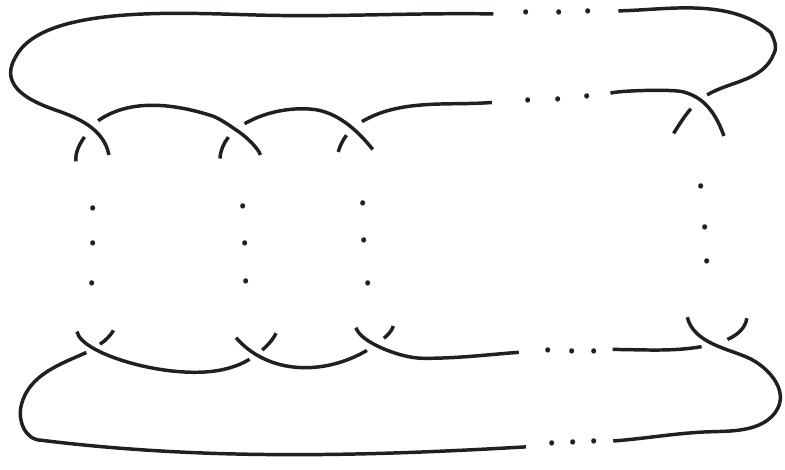}
\caption{}
\label{fig:pretzel_knot}
\qquad
\end{subfigure}
\begin{subfigure}[b]{.3\textwidth}
\centering
\includegraphics[height=20mm]{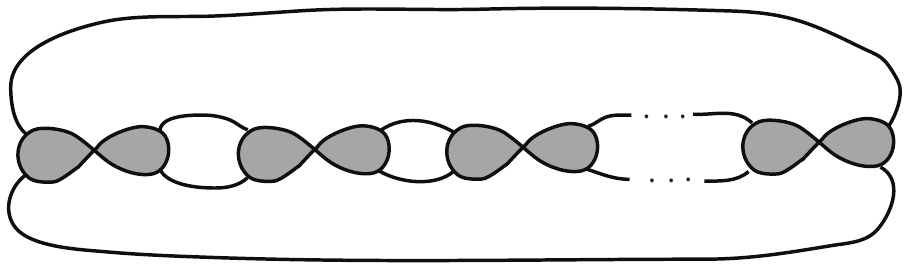}
\caption{}
\label{fig:pretzel_solution}
\qquad
\end{subfigure}
\caption{Proving Theorem \ref{pretzel}.}
\label{fig:pretzel}
\end{figure}
In this way we obtain a projection of $K$ as in Figure \ref{fig:pretzel}(b) in which the sum of all crossing multiplicities is $c(K)-\ell+(k-1)+2\ell=c(K)+\ell+k-1$.

Now we may apply Moves 4 and 5 alternately beginning on the left of the diagram in Figure \ref{fig:pretzel}(b). When we apply Move 4 we combine two \"ubertangles and increase their combined multiplicity by 3, and application of Move 5 increases their combined multiplicity by 1. Hence we obtain an \"ubercrossing projection of $K$ with 
\[
c(K)+\ell+k-1+\overbrace{3+1+3+1+\cdots+[3\text{ or } 1]}^{k-1}
\]
loops, as desired.

Now suppose the sequence contains no even numbers. We can use the same method without using Move 3, which gives an \"ubercrossing projection of $K$ with
\[
c(K)+k+\overbrace{3+1+3+1+\cdots+[3\text{ or } 1]}^{k-1}
\]
loops, proving the second fact.
\end{proof}

\subsection{Torus Knots}

We now obtain  bounds for \"ubercrossing number and petal number for various torus knots and links. We develop two more moves for this purpose.

\begin{lemma}\label{Move6}
(Move 6) A suitable segment of a braid may be reduced as shown in Figure \ref{fig:torus_knot_2}.
\end{lemma}

\begin{figure}[here]
\centering
\includegraphics[height=30mm]{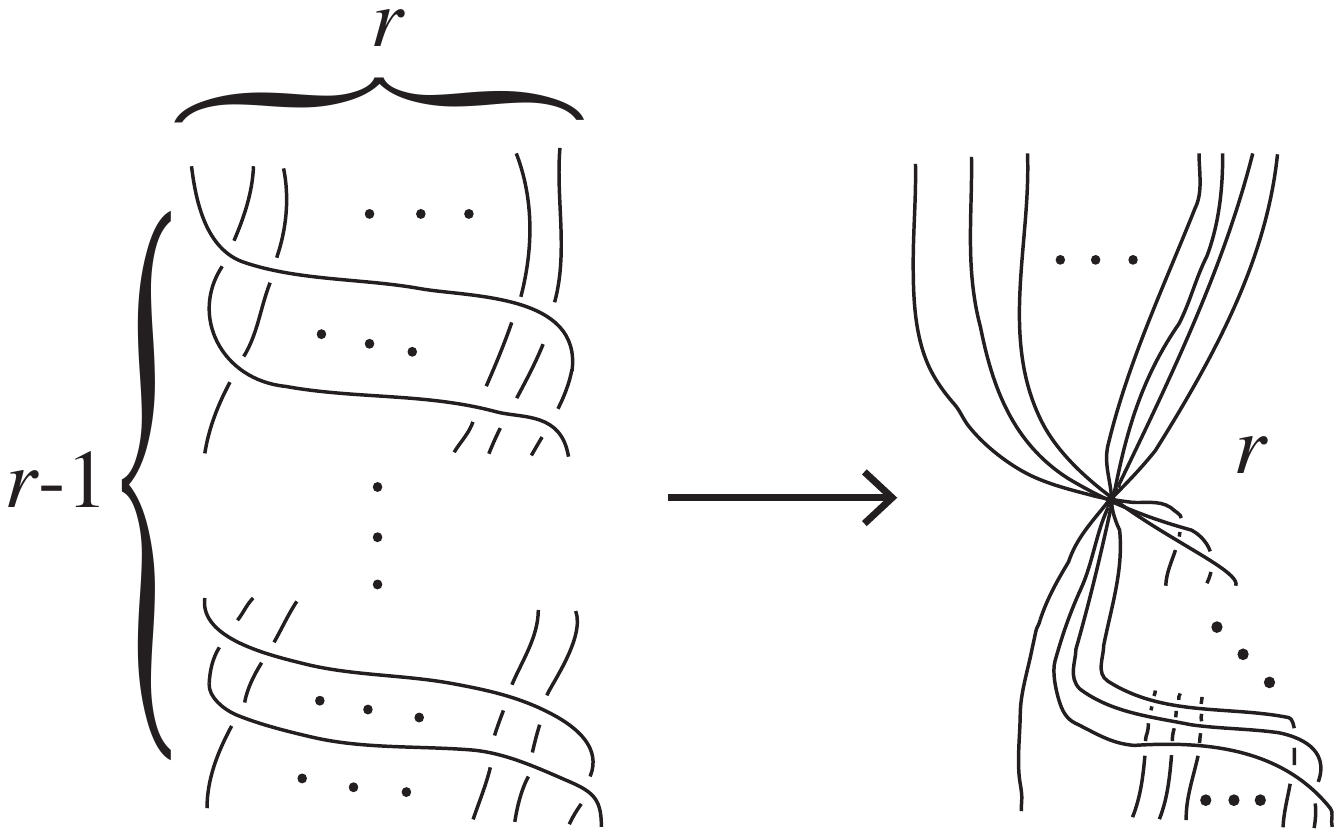}
\caption{Move 6.}
\label{fig:torus_knot_2}
\end{figure}

\begin{proof}
We induct on $r$, with a trivial base case when $r=2$. The inductive step is shown in figure \ref{fig:torus_knot_1a}; we move the black strand as shown and apply the inductive hypothesis to the indicated region. Then the black strand may be slid into the crossing, adding 1 to its multiplicity, to complete the induction.
\begin{figure}[here]
\centering
\includegraphics[height=35mm]{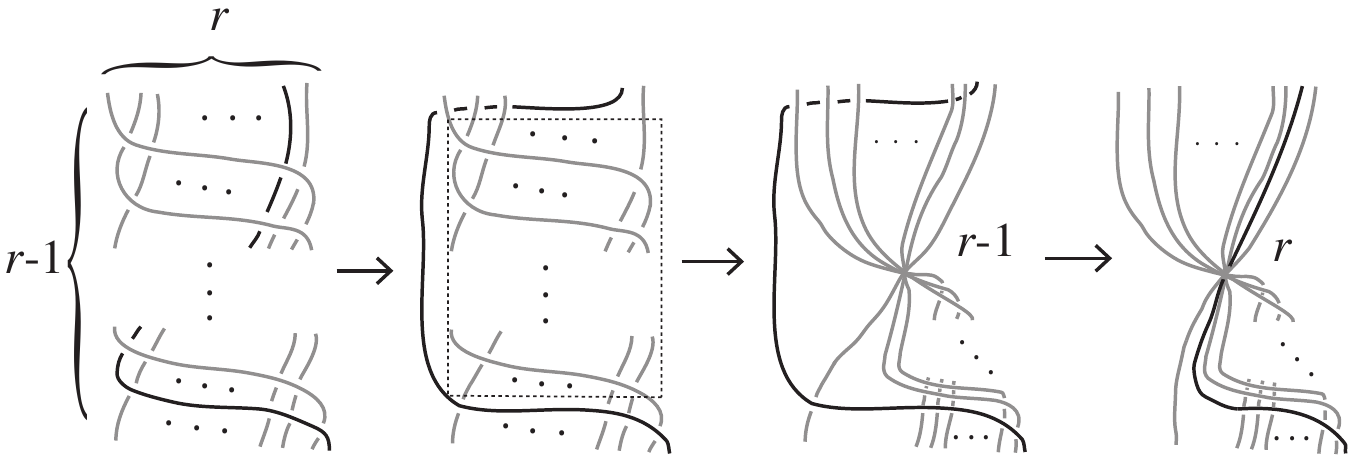}
\caption{Proving Lemma \ref{Move6}.}
\label{fig:torus_knot_1a}
\qquad
\end{figure}
\end{proof}

At this point it becomes convenient to work with \"ubertangles with more than two lobes as in Figure \ref{fig:torus_knot_3}. This notation indicates that a line through any pair of opposite regions bisects the crossing.

\begin{lemma}\label{Move7}
(Move 7). A portion of a projection may be simplified as in Figure \ref{fig:torus_knot_3}.
\end{lemma}

\begin{figure}[here]
\centering
\includegraphics[height=35mm]{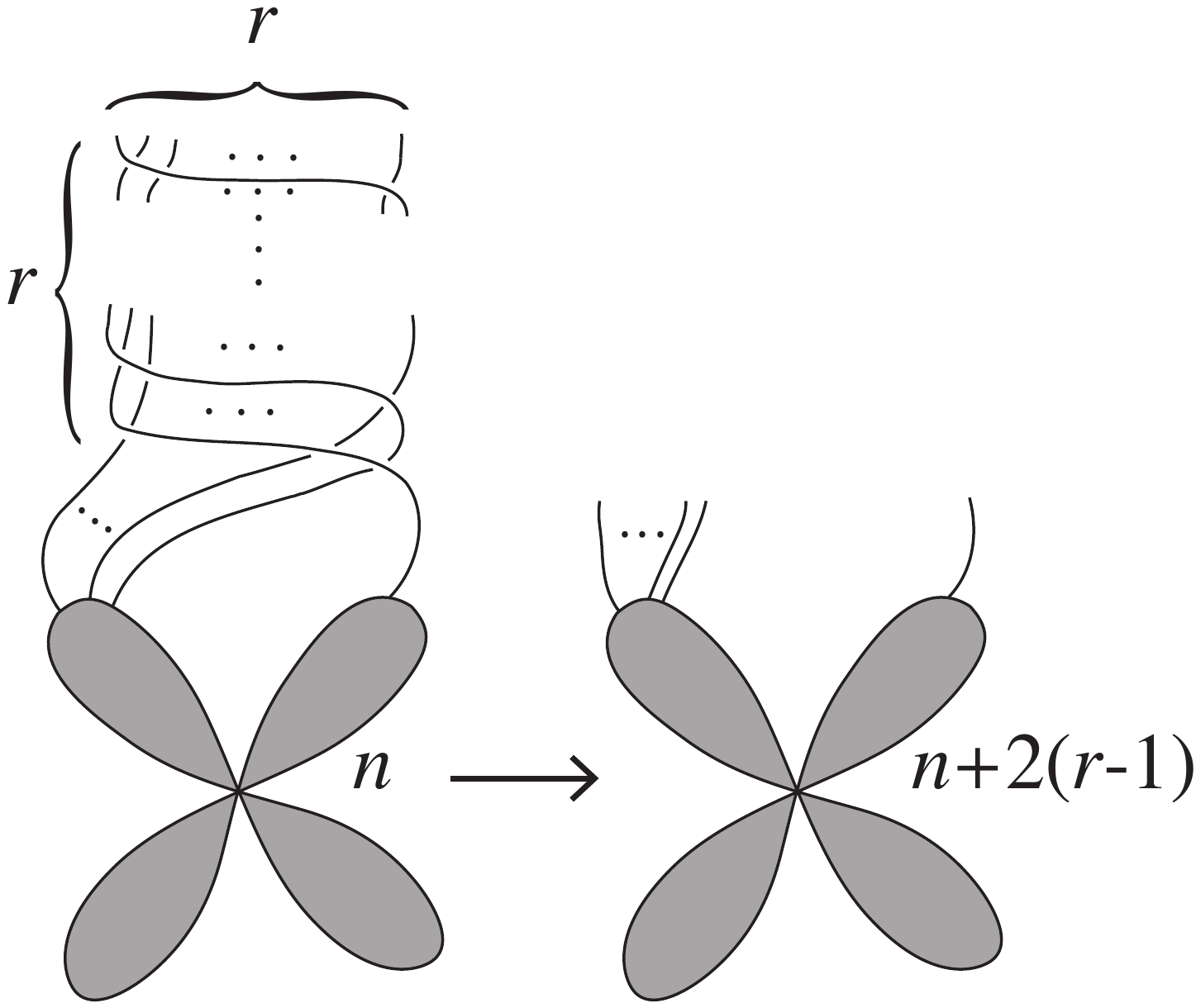}
\caption{Move 7.}
\label{fig:torus_knot_3}
\end{figure}

\begin{proof}
We induct on $r$, with base case $r=2$ shown in Figure \ref{fig:torus_induction}(a). Assuming the move is valid for $r-1$, we proceed as in Figure \ref{fig:torus_induction}(b).
\end{proof}

\begin{figure}[here]
\centering
\begin{subfigure}[b]{.3\textwidth}
\includegraphics[height=25mm]{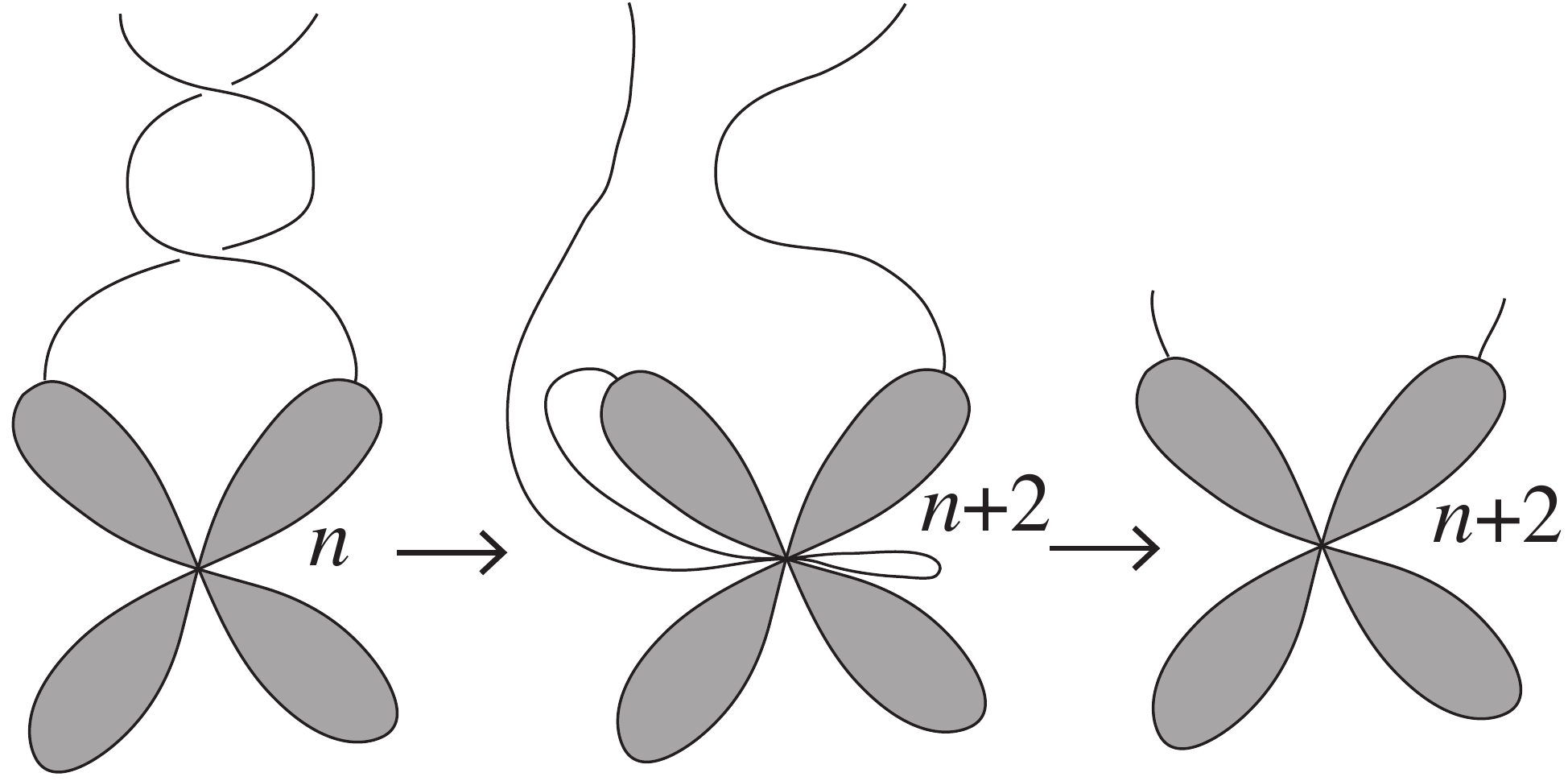}
\caption{}
\label{fig:torus_knot_4}
\qquad
\end{subfigure}
\qquad
\begin{subfigure}[b]{.3\textwidth}
\includegraphics[height=30mm]{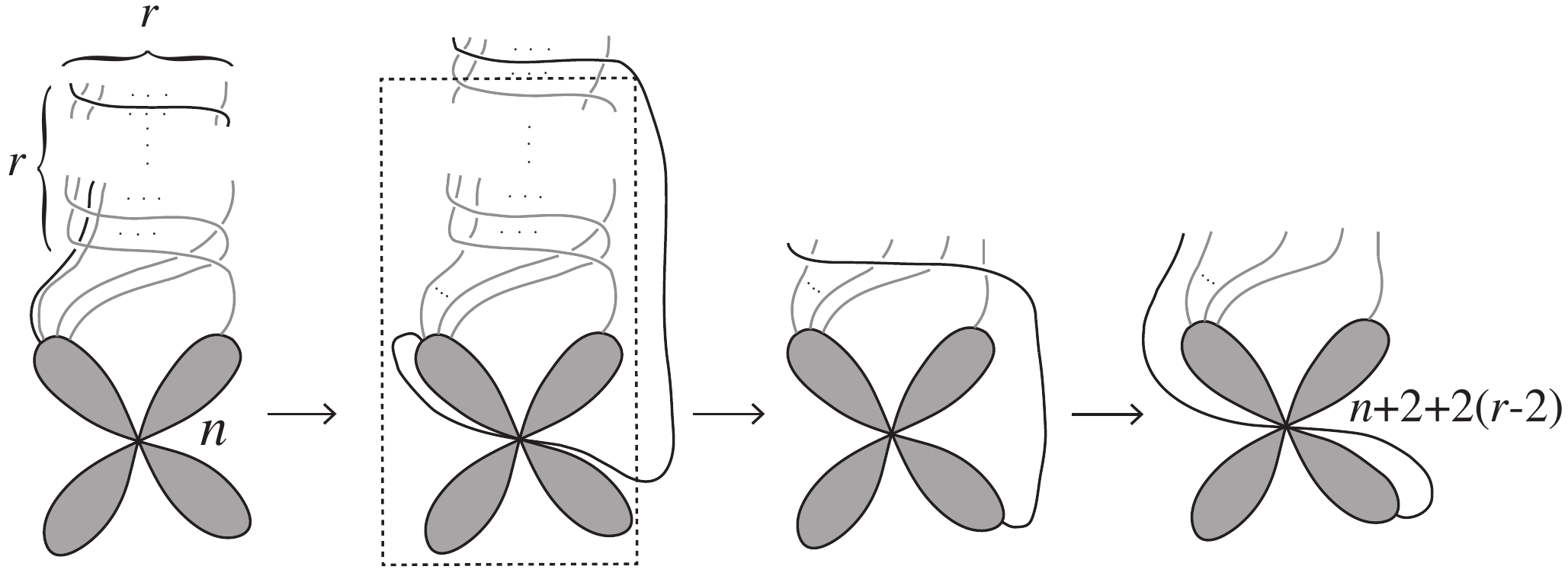}
\caption{}
\label{fig:torus_knot_5}
\qquad
\end{subfigure}
\caption{Proving Lemma \ref{Move7}.}
\label{fig:torus_induction}
\qquad
\end{figure}

\begin{thm}
The torus knot $T_{3,s}$ satisfies $\ub(T_{3,s})\leq 4\lfloor\frac{s-2}{3}\rfloor+[s-2]_3+4$, where $[a]_3$ denotes the residue of $a$ modulo 3.
\end{thm}

\begin{proof}
Take $T_{3,s}$ and place it in its closed braid representation as shown in Figure \ref{fig:T_3_qpic}(a). Applying Move 6 to the bottommost two overpasses yields a diagram which can be simplified as in Figure \ref{fig:T_3_qpic}(b). There are $s-2$ remaining overpasses, which we eliminate in groups of 3 using Move 7, adding 4 strands to the crossing for every 3 overpasses removed. At the end of this process we are left with 0, 1, or 2 overpasses, which we eliminate as shown in Figure \ref{fig:T_3_reduction}.
\end{proof}
\begin{figure}[here]
\centering
\begin{subfigure}[b]{.3\textwidth}
\centering
\includegraphics[height=35mm]{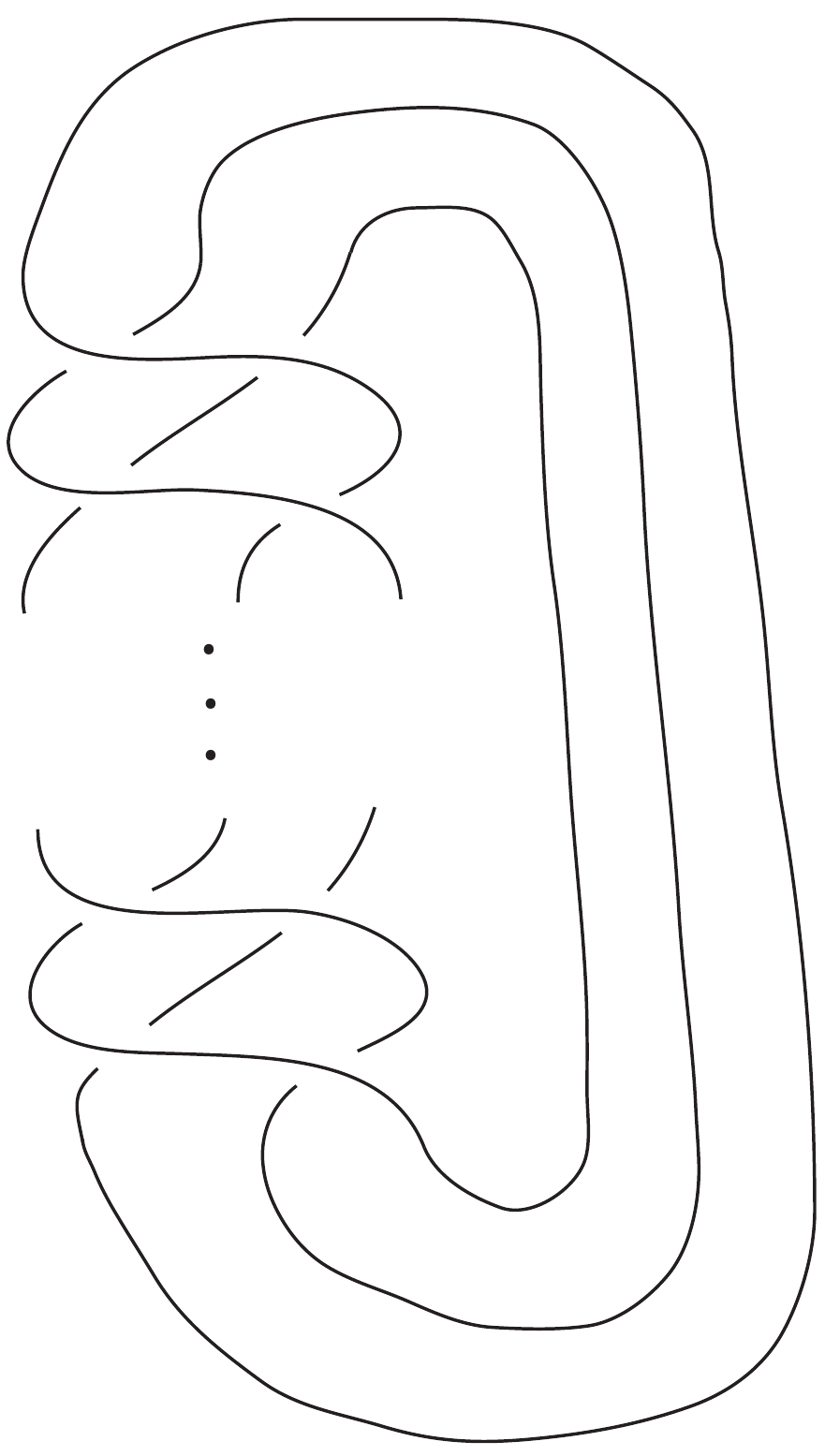}
\caption{}
\label{fig:T_3_q}
\qquad
\end{subfigure}
\begin{subfigure}[b]{.3\textwidth}
\centering
\includegraphics[height=35mm]{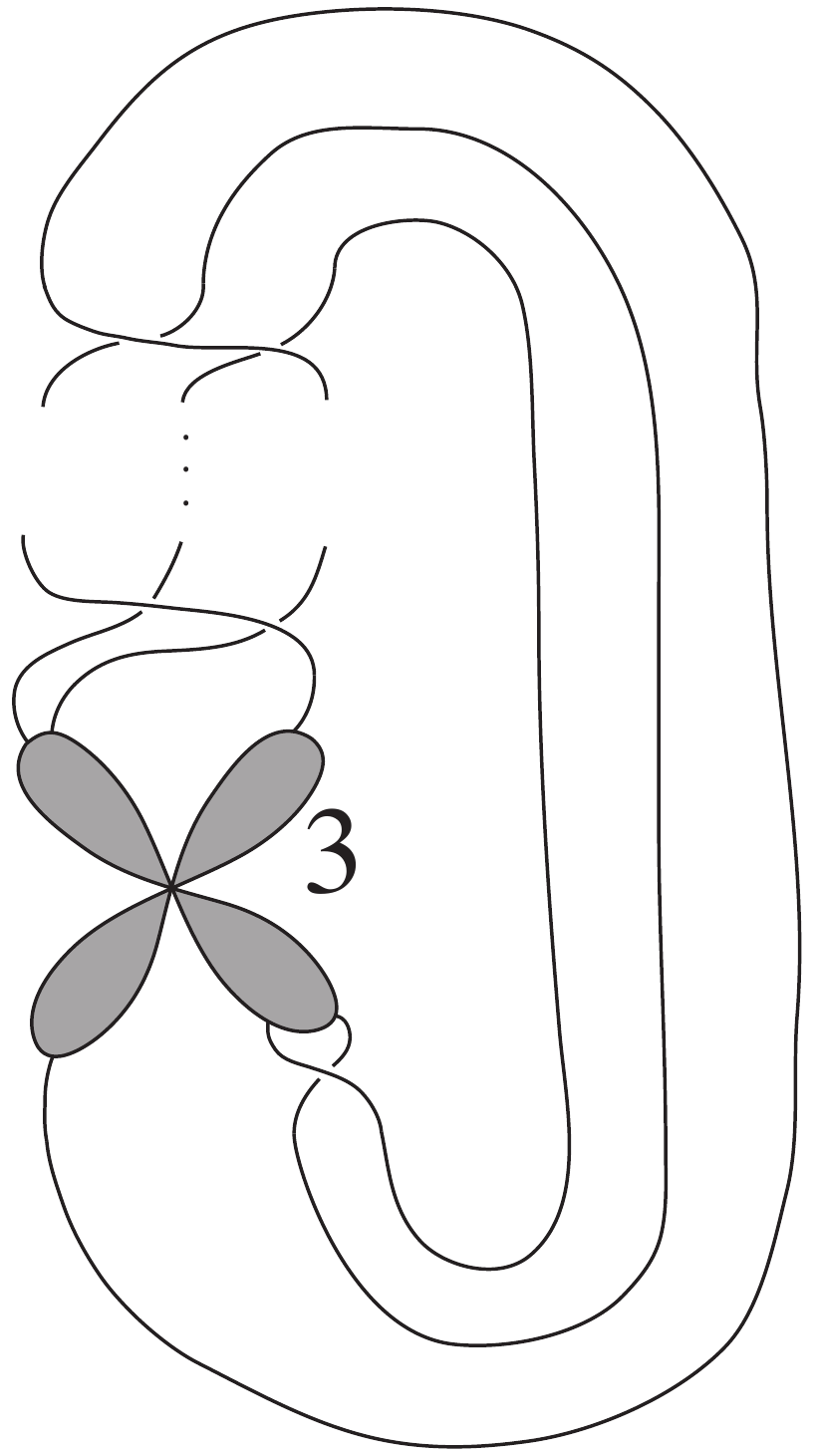}
\caption{}
\label{fig:T_3_q2}
\qquad
\end{subfigure}
\caption{Initial move on a $T_{3,s}$ torus knot.}
\label{fig:T_3_qpic}
\qquad
\end{figure}

\begin{figure}[here]
\centering
\begin{subfigure}[b]{.3\textwidth}
\includegraphics[height=32mm]{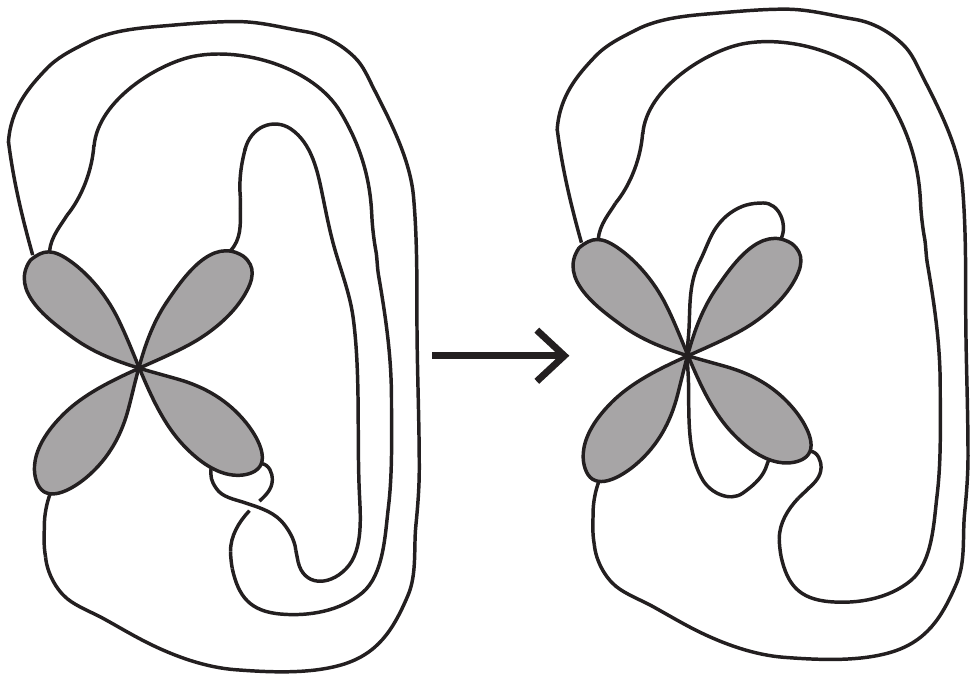}
\caption{$s-2\equiv0$ (mod 3).}
\label{fig:T_3_q3}
\qquad
\end{subfigure}
\begin{subfigure}[b]{.3\textwidth}
\includegraphics[height=32mm]{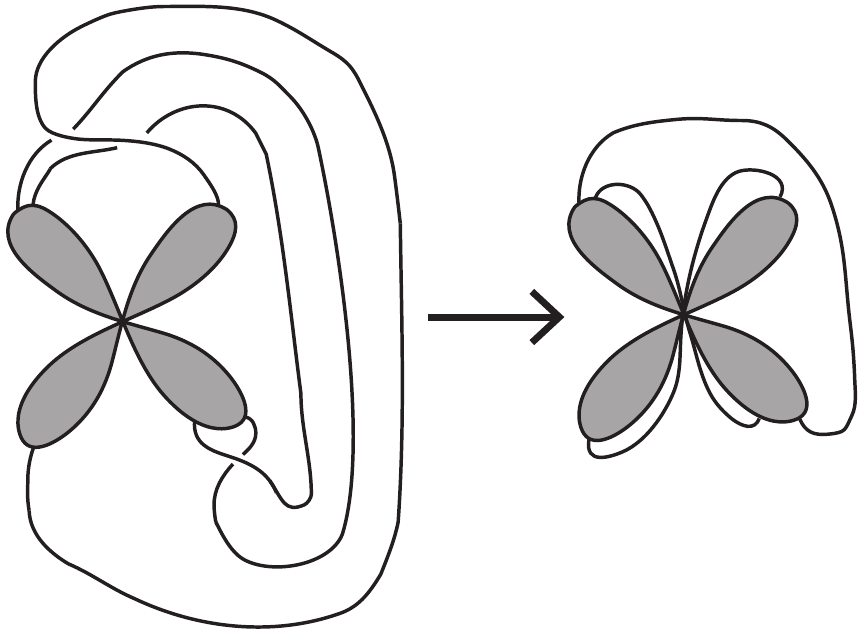}
\caption{$s-2\equiv1$ (mod 3).}
\label{fig:T_3_q4}
\qquad
\end{subfigure}
\\\vspace{4mm}
\begin{subfigure}[b]{.3\textwidth}
\includegraphics[height=32mm]{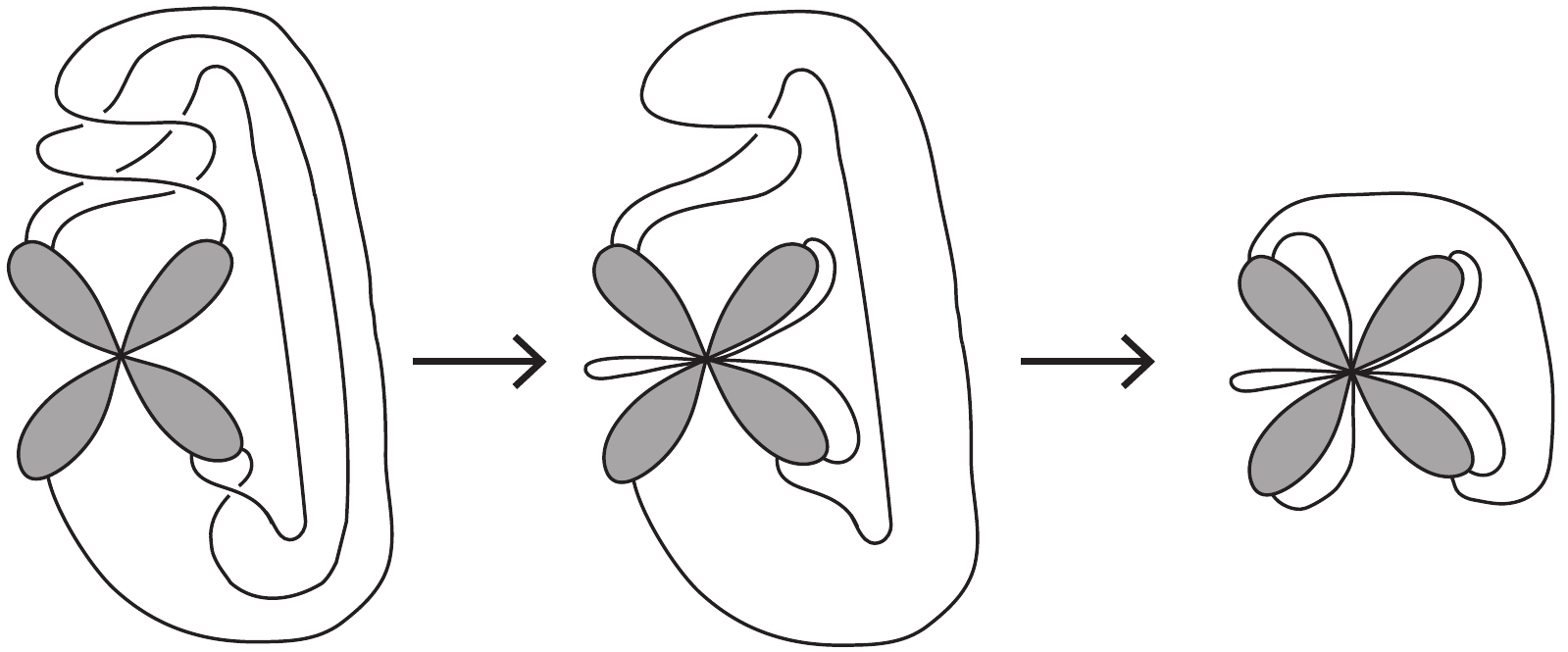}
\caption{$s-2\equiv2$ (mod 3).}
\label{fig:T_3_q5}
\qquad
\end{subfigure}
\caption{Subsequent moves on a $T_{3,s}$ torus knot.}
\label{fig:T_3_reduction}
\qquad
\end{figure}

We now consider petal number for certain torus knots and links.

\begin{thm} \label{torusthm}

Let $T_{r,s}$ be a torus knot with $s\equiv \pm1$(mod $r$).  Then
	\[p(T_{r,s}) \leq \left\{\begin{array}{cc}
		2s -1 \quad \text{for} & s\equiv 1(mod \, r) \\
		2s + 3 \quad \text{for} & s\equiv -1(mod \, r).\end{array}\right.\]

\end{thm}

\begin{proof}  Let $K = T_{r,s}$ with $s\equiv 1$(mod $r$).  Let $B$ be the canonical braid projection of $K$ in the plane, in polygonal form as in Figure \ref{torus3}.  Label the horizontal segments which cross over $r$ other segments $1, 2, ..., s$ from bottom to top.  Call these segments the non-trivial horizontal segments.  Label the slanted segments $1', 2', ..., s'$, where segment $n'$ is the segment whose bottommost endpoint is on horizontal segment $n$.  This leaves the bottom $r-1$ segments without labels; let each of these segments share the label of the slanted segment at the top of the braid which it is immediately connected to along the closure of the braid.

\begin{figure}
	\includegraphics[height= 140mm]{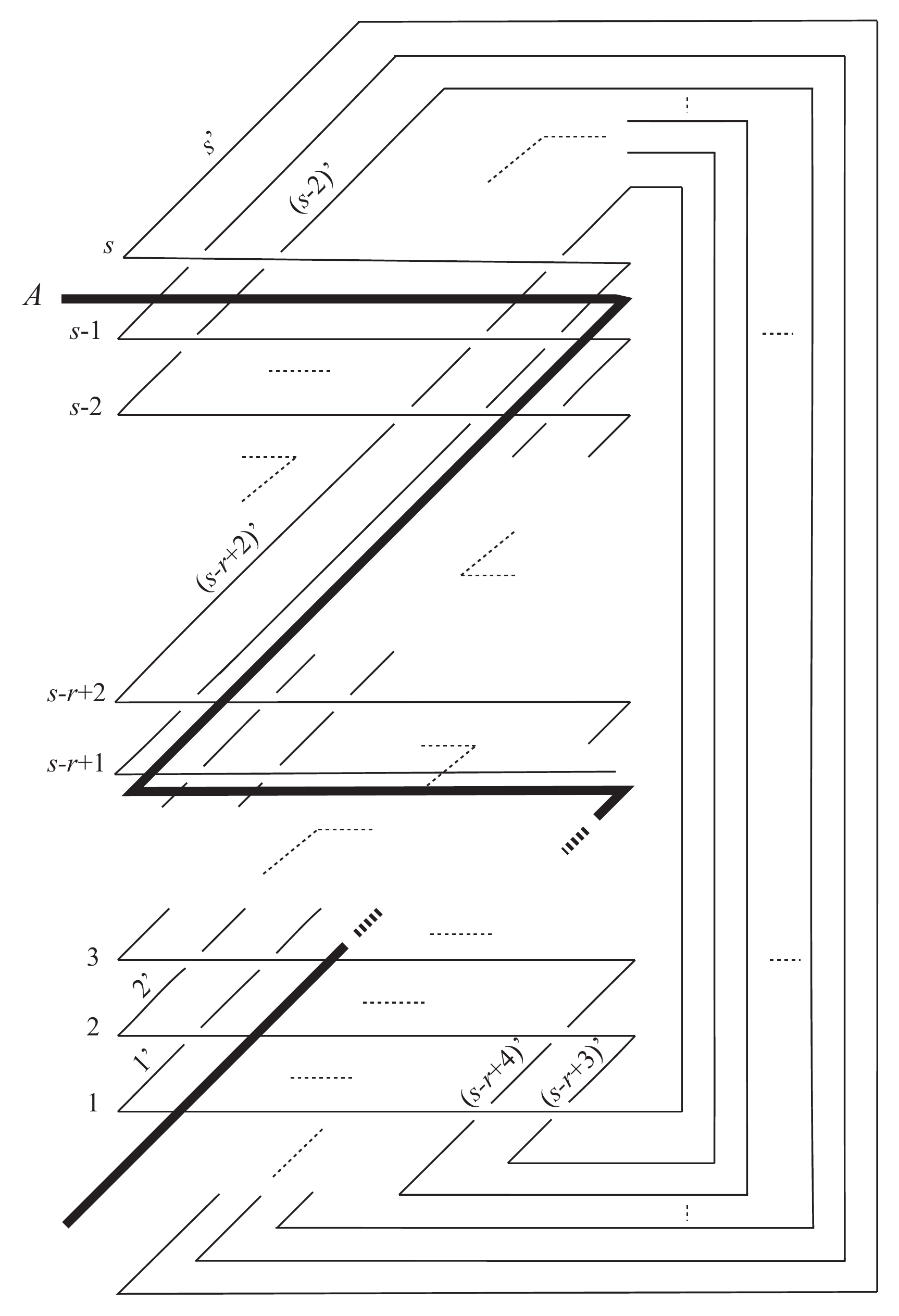}
	\caption{Defining the axis $A$ for a projection of certain torus knots with overcrossings to the right and undercrossings to the left.}
	\label{torus3}
\end{figure}

Beginning at the left endpoint of segment $s$, travel along segment $s$ and continue to traverse the entire braid.  Label each crossing the first time it is passed as either an overcrossing or an undercrossing, as in Step 2 of the algorithm of Section 2.  We shall explicitly construct the axis $A$ of the algorithm, which partitions these labels of over and undercrossings and gives us information on the number of petals in one petal diagram of $K$.  To do so, the following observation is handy.

Suppose we are traveling along a slanted segment $m'$ and we encounter the horizontal segment $n$ at the $i$th crossing along it, ordering the crossings along segment $n$ from left to right.  Continuing along this slanted segment, we find that the next non-trivial horizontal segment is encountered at the $(i-1)$-th crossing on it, the next successive non-trivial horizontal segment is encountered at the $(i-2)$-th crossing on it, and so on.  In general, the $k$-th successive non-trivial horizontal segment (after the first crossing of segment $n$) is encountered at the $(i-k)$(mod $r$) crossing on segment $(i-k)$(mod $s$).  Here we take the 0-th crossing on a segment to be its leftmost endpoint.  Note that an `encounter' with a non-trivial horizontal segment can entail either crossing the segment or traveling along it.

The above observation is vital in building the partition of undercrossings and overcrossings.  For assume the crossing of segments $n$ and $m'$, denoted $c_{n,m'}$, is labeled as an undercrossing.  Traveling through $s$ non-trivial horizontal segments, the constraint $s\equiv 1$(mod $r$) brings us to the crossing on segment $n$ directly left of $c_{n,m'}$, provided $c_{n,m'}$ is not the leftmost crossing on segment $n$.  This crossing has not yet been labeled, else $c_{n,m'}$ would be labeled as an overcrossing; therefore, it becomes labeled as an undercrossing.  Thus, all crossings to the left of the rightmost undercrossing on a horizontal segment are also undercrossings.  In particular, the crossings on any horizontal segment $n$ can be partitioned into two sections, the left section containing only undercrossings and the right section containing only overcrossings.  Furthermore, the rightmost undercrossing on a segment is the first undercrossing on that segment to receive a labeling as the knot is traversed.  With this machinery, we may construct $A$.

Let $A$ be a piecewise line consisting of horizontal and slanted segments; the horizontal segments lie directly below each segment with a label $\ell\equiv 1$(mod $r$) $(\ell >1)$, so that the slanted segments connect the rightmost point of a horizontal segment with the leftmost point of the horizontal segment below it.  Note that the slanted segments lie directly to the right of each segment $\ell'$, where $\ell\equiv 1$(mod $r$) (Figure \ref{torus3}).

On segments $n$ with $n\not\equiv 1$(mod $r$), the rightmost undercrossing occurs at the crossing with a segment $\ell'$, where $\ell\equiv 1$(mod $r$).  These are precisely the crossings directly to the left of the slanted segments of $A$.  Furthermore, segments labeled $\ell\equiv 1$(mod $r$), $\ell >1$, contain only overcrossings, while segments $\ell\equiv 0$(mod $r$) contain only undercrossings.  These are the segments between which the horizontal segments of $A$ lie.  Thus, traveling from the bottom of $A$ to the top, one has only undercrossings to the immediate left of $A$ and only overcrossings to the immediate right of $A$.  But then all crossings to the left of $A$ must be undercrossings and all crossings to the right of $A$ must be overcrossings.  Thus, $A$ provides the above partition of over and undercrossings.

Isotope $A$ and $B$ without changing crossings, so that $A$ is a straight line.  Taking $A$ as in the algorithm of Section 2, the projection $B$ is already in the form the algorithm requires.  Using the algorithm, a petal projection of $K$ is produced with $|A\cap B|+1$ petals.  Each line segment comprising $A$ intersects $B$ $r$ times.  There are a total of $2(\frac{s-1}{r})$ segments in $A$, yielding a petal number of no more than $2s-1$.

The proof for $s\equiv -1$(mod $r$) is in exactly the same vein, but the changed constraint forces all crossings to the right of an undercrossing on a non-trivial horizontal strand to be labeled as an undercrossing.  This forces the horizontal segments of $A$ to lie directly above each segment with a label $\ell\equiv 1$(mod $r$).  Each line segment of $A$ still intersects $B$ $r$ times, but now there are $2(\frac{s+1}{r})$ such segments, yielding a petal number of no more than $2s+3$.
\end{proof}

\begin{cor}\label{toruscor}
The $T_{r,r+1}$ torus knot satisfies $p(T_{r,r+1})=2r+1$.
\end{cor}

\begin{proof}
Since $\alpha(T_{r,s})=r+s$ by \cite{EtHo}, we have $\alpha(T_{r,r+1})= 2r+1$. We have $p(T_{r,r+1})\geq \alpha(T_{r,r+1})$, so the upper bound on $p(T_{r,r+1})$ from Theorem \ref{torusthm} is also a lower bound.
\end{proof}

\section{Knot Table}

Given an oriented petal projection of $K$ with $n$ loops, follow $K$ in the direction of the orientation, starting at the top strand of the \"ubercrossing. This gives rise to a sequence of levels at which the \"ubercrossing is traversed, starting at 1, which identifies $K$ up to chirality. Together with the orientation, it completely identifies $K$. 

This permutation of $1,\dots, n$ is not a unique representation of $K$.  For example, given a petal projection of $K$ with $n$ loops, we may perform a Type I Reidemeister move and fold the new loop over the \"ubercrossing, creating a petal projection of $K$ with $n+2$ petals from which we can obtain a permutation of $1,\dots, n+2$ representing $K$. Moreover, the original permutation is not necessarily the unique permutation of $1,\dots, n$ representing $K$. For example, $(1, 6, 3, 5, 7, 2, 4)$ and $(1, 3, 5, 2, 7, 4, 6)$ both represent  $4_1$. By a \emph{minimal representation} of $K$ we mean a permutation of $1,\dots, p(K)$ which corresponds to $K$. Table 1 lists the petal number and a minimal representation of the prime knots with fewer than 10 crossings, assuming the corresponding petal projection is traversed counterclockwise. This table was produced by considering permutations of 11 or fewer odd numbers. The corresponding petal diagrams were then fed into the Culler-Dunfield-Weeks  program SnapPy  \cite{Culler}, which can be used to identify the hyperbolic knots that result. The prime non-hyperbolic knots of nine or fewer crossings are only the 2-braid knots $3_1$, $5_1$, $7_1$ and  $9_1$ and the torus knot $8_{19}$. These can be handled by hand. The knots $9_{34}$ and $9_{40}$ do not appear in the list of knots with petal number at most 11, and they can be shown to be realized with petal number 13, completing the list.

\begin{remark}Observe that if a sequence $(a_1,\dots, a_n)$ representing a knot $K$ contains $a_i$, $a_{i+1}$ (where subscripts are computed modulo $n$) such that $|a_i-a_{i+1}|=1$, the loop corresponding to $a_i$ and $a_{i+1}$ can be removed from the corresponding petal diagram, reducing the number of  petals by two. One might hope that this adjacency property would manifest itself in any non-minimal representation of $K$, allowing us to reduce any sequence to a minimal one representing $K$. In particular this would give an algorithm for detecting the unknot: place any diagram in a petal projection using our algorithm, and repeatedly eliminate loops formed by adjacent strands. This is not the case, however. For example, the sequence $(1,9,3,5,7,10,2,4,8,11,6)$ represents the unknot.
\end{remark}

It is then natural to ask: how can we characterize the equivalence class of sequences representing a knot? In particular, how can we determine if two sequences represent the same knot? It would be nice to develop an analogue of Reidemeister moves for petal projection sequences.

\begin{center}
	\begin{table}[htbp!]
		\begin{minipage}{.48\linewidth}
		\begin{tabular}{| c | c | c |}
			\hline
			Knot & $p(K)$ & A minimal representation \\\hline
			$3_1$ & 5 & $(1,3,5,2,4)$\\\hline
			$4_1$ & 7 & $(1,3,5,2,7,4,6)$ \\\hline
			$5_1$ & 7 & $(1, 3, 6, 2, 5, 7, 4)$ \\\hline
			$5_{2}$ & 7 & $(1, 3, 6, 2, 4, 7, 5)$\\\hline
			$6_1$ & 9 & $(1, 3, 5, 2, 8, 4, 6, 9, 7)$\\\hline
			$6_{2}$ & 9 & $(1, 3, 5, 2, 8, 4, 7, 9, 6)$\\\hline
			$6_{3}$ & 9 & $(1, 3, 5, 2, 9, 7, 4, 8, 6)$\\\hline
			$7_1$ & 9 & $(1,8,4,9,5,3,6,2,7)$\\\hline
			$7_{2}$ & 9 & $(1, 3, 6, 9, 7, 2, 4, 8, 5)$\\\hline
			$7_{3}$ & 9 & $(1, 3, 6, 9, 7, 2, 5, 8, 4)$\\\hline
			$7_{4}$ & 9 & $(1, 3, 6, 4, 8, 2, 5, 9, 7)$\\\hline
			$7_{5}$ & 9 & $(1, 3, 6, 4, 8, 2, 7, 9, 5)$\\\hline
			$7_{6}$ & 9 & $(1, 3, 6, 4, 9, 7, 2, 8, 5)$\\\hline
			$7_7$ & 9 & $(1, 3, 7, 9, 4, 6, 2, 8, 5)$\\\hline
			$8_1$ & 11 & $(1, 3, 5, 2, 8, 11, 9, 4, 6, 10, 7)$\\\hline
			$8_{2}$ & 11 & $(1, 3, 5, 2, 9, 4, 7, 11, 8, 10, 6)$\\\hline
			$8_{3}$ & 11 & $(1, 3, 6, 2, 9, 5, 11, 4, 7, 10, 8)$\\\hline
			$8_{4}$ & 11 & $(1, 3, 5, 8, 6, 2, 10, 4, 7, 11, 9)$\\\hline
			$8_{5}$ & 11 & $(1, 3, 5, 8, 6, 11, 9, 2, 10, 4, 7)$\\\hline
			$8_{6}$ & 11 & $(1, 3, 5, 2, 8, 6, 10, 4, 9, 11, 7)$\\\hline
			$8_7$ & 11 & $(1, 3, 5, 2, 10, 7, 4, 8, 11, 9, 6)$\\\hline
			$8_8$ & 11 & $	(1, 3, 5, 2, 8, 6, 11, 9, 4, 10, 7)$\\\hline
			$8_9$ & 11 & $(1, 3, 5, 9, 2, 7, 11, 6, 4, 10, 8)$\\\hline
			 $8_{10}$ & 11 & $(1, 3, 5, 2, 9, 7, 11, 8, 4, 10, 6)$\\\hline
			$8_{11}$ & 11 & $(1, 3, 5, 2, 8, 11, 9, 4, 7, 10, 6)$\\\hline
			$8_{12}$ & 11 & $(1, 3, 5, 2, 9, 11, 8, 6, 10, 4, 7)$\\\hline
			$8_{13}$ & 11 & $(1, 3, 5, 2, 10, 7, 4, 9, 11, 8, 6)$\\\hline
			$8_{14}$ & 11 & $(1, 3, 5, 2, 10, 8, 11, 6, 9, 4, 7)$\\\hline
			$8_{15}$ & 11 & $(1, 3, 5, 2, 8, 11, 7, 9, 4, 10, 6)$\\\hline
			$8_{16}$ & 11 & $(1, 3, 5, 8, 6, 2, 11, 9, 4, 10, 7)$\\\hline
			$8_{17}$ & 11 & $(1, 3, 5, 8, 6, 2, 10, 4, 9, 11, 7)$\\\hline
			$8_{18}$ & 11 & $(1, 3, 7, 4, 10, 2, 8, 6, 11, 9, 5)$\\\hline
			 $8_{19}$ & 7 & $(1, 4, 7, 3, 6, 2, 5)$\\\hline
			$8_{20}$ & 9 & $(1, 3, 5, 8, 2, 6, 9, 4, 7)$\\\hline
			$8_{21}$ & 9 & $(1, 3, 5, 8, 2, 7, 4, 9, 6)$\\\hline
			$9_1$ & 11 & $(1,10,5,11,6,4,7,3,8,2,9)$\\\hline
			$9_{2}$ & 11 & $(1, 3, 6, 10, 7, 2, 4, 8, 11, 9, 5)$\\\hline
			$9_{3}$ & 11 & $(1, 3, 7, 5, 9, 2, 6, 11, 8, 10, 4)$\\\hline
			$9_{4}$ & 11 & $(1, 3, 6, 10, 7, 2, 5, 8, 11, 9, 4)$\\\hline
			$9_{5}$ & 11 & $(1, 3, 6, 4, 8, 11, 9, 2, 5, 10, 7)$\\\hline
			$9_{6}$ & 11 & $(1, 3, 6, 4, 9, 2, 7, 11, 8, 10, 5)$\\\hline
			$9_7$ & 11 & $(1, 3, 6, 10, 7, 2, 4, 9, 11, 8, 5)$\\\hline
		 \end{tabular}
 	\end{minipage}\begin{minipage}{.48\linewidth}
		\begin{tabular}{| c | c | c |}
			\hline
			Knot & $p(K)$ & A minimal representation \\\hline
			$9_8$ & 11 & $(1, 3, 6, 10, 8, 4, 11, 7, 2, 9, 5)$\\\hline
			$9_9$ & 11 & $(1, 3, 6, 10, 7, 2, 5, 9, 11, 8, 4)$\\\hline
			$9_{10}$ & 11 & $(1, 3, 7, 5, 8, 11, 9, 2, 6, 10, 4)$\\\hline
			$9_{11}$ & 11 & $(1, 3, 6, 4, 10, 7, 2, 8, 11, 9, 5)$\\\hline
			$9_{12}$ & 11 & $(1, 3, 6, 10, 5, 7, 2, 8, 11, 9, 4)$\\\hline
			$9_{13}$ & 11 & $(1, 3, 6, 4, 9, 2, 5, 11, 8, 10, 7)$\\\hline
			$9_{14}$ & 11 & $(1, 3, 7, 10, 5, 2, 9, 11, 8, 4, 6)$\\\hline
			$9_{15}$ & 11 & $(1, 3, 6, 4, 10, 8, 2, 7, 11, 9, 5)$\\\hline
			$9_{16}$ & 11 & $(1, 3, 7, 4, 10, 2, 9, 11, 6, 8, 5)$\\\hline
			$9_{17}$ & 11 & $(1, 3, 7, 10, 4, 6, 2, 9, 11, 8, 5)$\\\hline
			$9_{18}$ & 11 & $(1, 3, 6, 4, 8, 11, 9, 2, 7, 10, 5)$\\\hline
			$9_{19}$ & 11 & $(1, 3, 7, 5, 9, 11, 4, 8, 2, 10, 6)$\\\hline
			 $9_{20}$ & 11 & $(1, 3, 6, 4, 10, 8, 2, 9, 5, 11, 7)$\\\hline
			$9_{21}$ & 11 & $(1, 3, 6, 4, 10, 7, 2, 9, 11, 8, 5)$\\\hline
			$9_{22}$ & 11 & $(1, 3, 6, 4, 9, 7, 2, 10, 5, 11, 8)$\\\hline
			$9_{23}$ & 11 & $(1, 3, 6, 4, 9, 11, 7, 2, 8, 10, 5)$\\\hline
			$9_{24}$ & 11 & $(1, 3, 6, 11, 5, 7, 2, 9, 4, 10, 8)$\\\hline
			$9_{25}$ & 11 & $(1, 3, 6, 4, 8, 11, 7, 9, 2, 10, 5)$\\\hline
			$9_{26}$ & 11 & $(1, 3, 7, 5, 10, 6, 2, 9, 11, 4, 8)$\\\hline
			$9_{27}$ & 11 & $(1, 3, 6, 4, 11, 7, 2, 8, 10, 5, 9)$\\\hline
			$9_{28}$ & 11 & $(1, 3, 6, 11, 5, 7, 2, 8, 10, 4, 9)$\\\hline
			$9_{29}$ & 11 & $(1, 3, 6, 4, 10, 7, 2, 8, 5, 11, 9)$\\\hline
			$9_{30}$ & 11 & $(1, 3, 6, 4, 10, 8, 2, 7, 11, 5, 9)$\\\hline
			$9_{31}$ & 11 & $(1, 3, 6, 10, 5, 7, 2, 9, 11, 4, 8)$\\\hline
			$9_{32}$ & 11 & $(1, 3, 6, 4, 9, 11, 7, 2, 10, 5, 8)$\\\hline
			$9_{33}$ & 11 & $(1, 3, 6, 4, 10, 2, 7, 11, 9, 5, 8)$\\\hline
			$9_{34}$ & 13 & $(1,3,7,9,13,5,11,8,2,4,6,10,12)$\\\hline
			$9_{35}$ & 11 & $(1, 3, 10, 6, 2, 9, 11, 8, 5, 7, 4)$\\\hline
			$9_{36}$ & 11 & $(1, 3, 6, 4, 9, 7, 11, 8, 2, 10, 5)$\\\hline
			$9_{37}$ & 11 & $(1, 3, 7, 10, 4, 6, 2, 8, 11, 9, 5)$\\\hline
			$9_{38}$ & 11 & $(1, 3, 6, 4, 9, 2, 7, 11, 5, 10, 8)$\\\hline
			$9_{39}$ & 11 & $(1, 3, 6, 4, 10, 2, 7, 9, 5, 11, 8)$\\\hline
			$9_{40}$ & 13 & $(1,11,7,5,13,2,10,8,6,12,4,9,3)$\\\hline
			$9_{41}$ & 11 & $(1, 3, 7, 11, 4, 8, 10, 6, 2, 9, 5)$\\\hline
			$9_{42}$ & 9 & $(1, 3, 6, 2, 9, 5, 8, 4, 7)$\\\hline
			$9_{43}$ & 9 & $(1, 3, 6, 9, 5, 8, 2, 7, 4)$\\\hline
			$9_{44}$ & 9 & $(1, 3, 6, 9, 4, 7, 2, 8, 5)$\\\hline
			$9_{45}$ & 9 & $(1, 3, 7, 4, 9, 6, 2, 8, 5)$\\\hline
			$9_{46}$ & 9 & $(1, 3, 6, 9, 5, 2, 8, 4, 7)$\\\hline
			$9_{47}$ & 11 & $(1, 3, 5, 7, 10, 4, 9, 6, 2, 11, 8)$\\\hline
			$9_{48}$ & 11 & $(1, 3, 5, 2, 9, 11, 7, 4, 10, 6, 8)$\\\hline
			$9_{49}$ & 11 & $(1, 3, 5, 2, 7, 11, 8, 4, 10, 6, 9)$\\\hline

\end{tabular}
\end{minipage}
\vspace{.5cm}
\caption{Minimal permutation representations of prime knots with fewer than 10 crossings.}
\end{table}
\end{center}

\end{document}